\documentclass[12pt]{amsart}
\usepackage{amsthm,amssymb}
\usepackage{mathtools, etoolbox}

\usepackage{dsfont}
\usepackage{mathrsfs} 
\usepackage[all,cmtip]{xy}
\usepackage{patchcmd}
\input{diagxy}
\usepackage[backref=page]{hyperref}
\usepackage{bm}
\usepackage{enumitem}
\usepackage[usenames,dvipsnames,svgnames,table]{xcolor}
\usepackage[normalem]{ulem}
\usepackage[top=2.3cm, bottom=2.3cm, left=2cm, right=2cm]{geometry}
\usepackage{tikz}
\usetikzlibrary{arrows,positioning}
\allowdisplaybreaks

\newcommand{\memph}[1]{{\color{magenta}\emph{#1}}}

\makeatletter
\patchcommand\@starttoc{\begin{quote}}{\end{quote}}
\def\@tocline#1#2#3#4#5#6#7{\relax
  \ifnum #1>\c@tocdepth 
  \else
    \par \addpenalty\@secpenalty\addvspace{#2}%
    \begingroup \hyphenpenalty\@M
    \@ifempty{#4}{%
      \@tempdima\csname r@tocindent\number#1\endcsname\relax
    }{%
      \@tempdima#4\relax
    }%
    \parindent\z@ \leftskip#3\relax \advance\leftskip\@tempdima\relax
    \rightskip\@pnumwidth plus4em \parfillskip-\@pnumwidth
    #5\leavevmode\hskip-\@tempdima
      \ifcase #1
       \or\or \hskip 1em \or \hskip 2em \else \hskip 3em \fi%
      #6\nobreak\relax
    \dotfill\hbox to\@pnumwidth{\@tocpagenum{#7}}\par
    \nobreak
    \endgroup
  \fi}
\makeatother

 \theoremstyle{plain}
 \newtheorem{thm}{Theorem}[section]
 \newtheorem{cor}[thm]{Corollary}
 \newtheorem{lem}[thm]{Lemma}
 
 \newtheorem{prop}[thm]{Proposition}

\theoremstyle{definition}
 \newtheorem{defn}[thm]{Definition}
 
 \newtheorem{hyp}[thm]{Hypothesis}

\theoremstyle{remark}
 \newtheorem{rem}[thm]{Remark}

 \newtheorem{nota}[thm]{Notation}
 
 \newtheorem{exam}[thm]{Example}

 \numberwithin{equation}{section}

\theoremstyle{plain}

\DeclareMathOperator{\VF}{VF}
\DeclareMathOperator{\ACVF}{ACVF}
\DeclareMathOperator{\RV}{RV}

\DeclareMathOperator{\MM}{\mdl M}

\DeclareMathOperator{\OO}{\mdl O}

\DeclareMathOperator{\UU}{\mathfrak{U}}

 \DeclareMathOperator{\ran}{ran}

 \DeclareMathOperator{\id}{id}

 \DeclareMathOperator{\spec}{Spec}
 
 \DeclareMathOperator{\aut}{Aut}

 \DeclareMathOperator{\ac}{\overline{ac}}

 \DeclareMathOperator{\dcl}{dcl}
 \DeclareMathOperator{\pr}{pr}

 \DeclareMathOperator{\mgl}{GL}

\DeclareMathOperator{\K}{\mathds{k}}

\DeclareMathOperator{\res}{res}  

\def\XXint#1#2#3{{\setbox0=\hbox{$#1{#2#3}{\int}$}
\vcenter{\hbox{$#2#3$}}\kern-.5\wd0}}

\newcommand{\Z}{\mathds{Z}}
\newcommand{\A}{\mathds{A}}
\newcommand{\C}{\mathds{C}}

\newcommand{\F}{\mathds{F}}
\newcommand{\G}{\mathds{G}}

\newcommand{\Q}{\mathds{Q}}
\newcommand{\N}{\mathds{N}}

\newcommand{\R}{\mathds{R}}

\newcommand{\omin}{$o$\nobreakdash}

\newcommand{\cmin}{$C$\nobreakdash}

\newcommand{\T}{$T$\nobreakdash}

\newcommand{\dand}{\quad \text{and} \quad}
\newcommand{\tand}{~\text{and}~}

\newcommand{\gR}{\mathfrak{R}}

\newcommand{\gD}{\mathfrak{D}}

\newcommand{\gF}{\mathfrak{F}}
\newcommand{\gG}{\mathfrak{G}}

\newcommand{\0}{\emptyset}

\DeclareMathAlphabet{\mathpzc}{OT1}{pzc}{m}{it}

\newcommand{\dbra}[1]{
  \,[\mkern-6.8mu[\, #1 \,]\mkern-6.8mu]}
\newcommand{\dpar}[1]{
   (\mkern-4mu (#1 )\mkern-4mu )}

\providecommand\given{}
\newcommand\SetSymbol[1][]{%
\nonscript \: #1 \vert
\allowbreak
\nonscript\:
\mathopen{}}
\DeclarePairedDelimiterX\set[1]\{\}{%
\renewcommand\given{\SetSymbol[\delimsize]}
#1
}

 \DeclarePairedDelimiterX\norm[1]\lVert\rVert{\ifblank{#1}{\:\cdot\:}{#1}}

 \newcommand{\usub}[2]{#1_{\textup{#2}}}

 \newcommand{\lan}[3]{\mathcal{L}_{#1 \textup{#2} #3}}

\newcommand{\mdl}[1]{\mathcal{#1}}  
\newcommand{\bb}[1]{\mathbb{#1}}

\newcommand{\limplies}{\rightarrow}

\newcommand{\rest}{\upharpoonright}

\newcommand{\fun}{\longrightarrow}
\newcommand{\efun}{\longmapsto}
\newcommand{\sub}{\subseteq}

\newcommand{\mi}{\smallsetminus}

\newcommand{\colim}[1]{\underset{#1}{\text{colim}}\,}  

\newcommand{\la}{\langle}
\newcommand{\ra}{\rangle}

\newbox\gnBoxA
\newdimen\gnCornerHgt
\setbox\gnBoxA=\hbox{$\ulcorner$}
\global\gnCornerHgt=\ht\gnBoxA
\newdimen\gnArgHgt

\def\code #1{%
        \setbox\gnBoxA=\hbox{$#1$}%
        \gnArgHgt=\ht\gnBoxA%
        \ifnum \gnArgHgt<\gnCornerHgt
                \gnArgHgt=0pt%
        \else
                \advance \gnArgHgt by -\gnCornerHgt%
        \fi
        \raise\gnArgHgt\hbox{$\ulcorner$} \box\gnBoxA %
                \raise\gnArgHgt\hbox{$\urcorner$}}

\DeclareMathOperator{\mVF}{{\mu}{\VF}}
\DeclareMathOperator{\mgVF}{{\mu} {\VF}}
\DeclareMathOperator{\mRV}{{\mu}{\RV}}
\DeclareMathOperator{\mgRV}{{\mu}{\RV}}
\DeclareMathOperator{\mG}{{\mu}{\Gamma}}
\DeclareMathOperator{\RES}{RES}
\DeclareMathOperator{\mgRES}{{\mu}{\RES}}
\DeclareMathOperator{\mRES}{{\mu}{\RES}}
\DeclareMathOperator{\isp}{I_{sp}}

\DeclareMathOperator{\mgisp}{{\mu}{I_{sp}}}

\DeclareMathOperator{\mgE}{{\mu}{\bb E}}
\DeclareMathOperator{\mgL}{{\mu}{\bb L}}

\DeclareMathOperator{\rv}{rv}

\DeclareMathOperator{\csn}{csn}
\DeclareMathOperator{\rcsn}{\overline {csn}}
\DeclareMathOperator{\vv}{val}

\DeclareMathOperator{\gsk}{\mathbf{K}_+}
\DeclareMathOperator{\ggk}{\mathbf{K}}

\DeclareMathOperator{\sggk}{!\mathbf{K}}

\DeclareMathOperator{\fin}{fin}

\DeclareMathOperator{\vrv}{vrv}

\DeclareMathOperator{\pvf}{pr_{\VF}}

\DeclareMathOperator{\gal}{Gal}

\DeclareMathOperator{\RCF}{RCF}
\DeclareMathOperator{\db}{d \! b}

\DeclareMathOperator{\var}{Var}
\DeclareMathOperator{\TVF}{TVF}
\DeclareMathOperator{\TRV}{TRV}
\DeclareMathOperator{\TRES}{TRES}
\DeclareMathOperator{\TCVF}{TCVF}

\newcommand{\LT}{$\lan{T}{}{}$\nobreakdash}

\DeclareMathOperator{\bfk}{\bf k}

\DeclareMathOperator{\puR}{\tilde \R}
\DeclareMathOperator{\puC}{\tilde \C}

\DeclareMathOperator{\TG}{T\Gamma}
\DeclareMathOperator{\Vol}{Vol}
\DeclareMathOperator{\Var}{Var}
\DeclareMathOperator{\RVar}{\R \kern -3.3 pt \Var}
\DeclareMathOperator{\tRVar}{\tilde{\R} \kern -3.3 pt \Var}
\DeclareMathOperator{\Def}{Def}
\DeclareMathOperator{\sgn}{sgn}
\DeclareMathOperator{\tbk}{tbk}
\DeclareMathOperator{\gdv}{{\ggk}^{\hat \delta}\var_{\R}}
\DeclareMathOperator{\gmv}{{\ggk}^{\hat \mu}\var_{\C}}
\DeclareMathOperator{\gsv}{{\ggk}^{\mu_2}\RVar}
\DeclareMathOperator{\kuq}{{\K^{\times}} \cup \Q}
\newcommand{\vta}{\vartheta}

\begin{document}

\title[Motivic integration and Milnor fiber]{Motivic integration and Milnor fiber}


\author[G. Fichou]{Goulwen Fichou}
\address{Goulwen Fichou \\ Universit\'e de Rennes \\ CNRS \\ IRMAR-UMR 6625 \\ F-35000 Rennes \\ France}
\email{goulwen.fichou@univ-rennes1.fr}

\author[Y. Yin]{Yimu Yin}
\address{Yimu Yin, Santa Monica, California}
\email{yimu.yin@hotmail.com}

\thanks{The research leading to the true claims in this paper has been partially supported by the grant ANR-15-CE40-0008 (D\'efig\'eo), CNRS, the SYSU grant 11300-18821101.  We also thank the Oberwolfach Research Institute for Mathematics and the Pacific Institute for the Mathematical Sciences at the University of British Columbia for their hospitality}



\dedicatory{In memory of Masahiro Shiota}

\begin{abstract}
We put forward in this paper a uniform narrative that weaves together several variants of Hrushovski-Kazhdan style integral, and describe how it can facilitate the understanding of the Denef-Loeser motivic Milnor fiber and closely related objects. Our study  focuses on the so-called ``nonarchimedean Milnor fiber'' that was introduced by Hrushovski and Loeser, and  our thesis is that it is a richer embodiment  of the underlying philosophy of the Milnor construction.  The said narrative is first developed in the more natural complex environment, and is then extended to the real one via descent. In the process of doing so, we are able to provide more illuminating new proofs, free of resolution of singularities, of a few pivotal results in the literature, both complex and real. To begin with, the real motivic zeta function is shown to be rational, which yields the real motivic Milnor fiber; this is an analogue of the Hrushovski-Loeser construction. We also establish, in a much more intuitive manner, a new Thom-Sebastiani formula, which can be specialized to the one given by Guibert-Loeser-Merle. Finally, applying \T-convex integration after descent, matching  the Euler Characteristics of the topological Milnor fiber and the motivic Milnor fiber becomes a matter of simple computation, which is not only free of resolution of singularities as in the Hrushovski-Loeser proof, but is also free of other sophisticated algebro-geometric machineries.
%
\end{abstract}

\maketitle

\tableofcontents

\section{Introduction}\label{sec:intro}

Recent years have seen significant development in applying  Hrushovski-Kazhdan's integration theory to the  study of Denef-Loeser's motivic Milnor fiber and related topics. The main goal of this paper is to articulate a uniform narrative on such interactions, and thereby not only recover several fundamental results regarding motivic Milnor fiber but also subjugate them to the same principles afforded by the new perspective, and hopefully open up new fronts of inquiry in the process. This narrative is summarized in the diagram (\ref{sum:narr}) below.

More concretely, we shall reconstruct motivic Milnor fibers  as motivic integrals, establish a general type of Thom-Sebastiani formula, and retrieve  invariants of the corresponding topological Milnor fibers, all without using resolution of singularities. In fact, there are several variants of the Hrushovski-Kazhdan style integration at play here and their synergy  is the driving force of our telling. Among these variants, the central one is of course the original construction as developed in \cite{hrushovski:kazhdan:integration:vf}. It works for any algebraically closed valued fields of equal characteristic $0$ and is flexible enough to allow arbitrary choice of parameter spaces that satisfy certain mild conditions. Varying the parameter space enables one to study different categories of definable sets that are equipped with suitable Galois actions, which is highly desirable in the applications we are interested in (a list of the various pairs of ambient and parameter spaces is provided at the end of this introduction for the reader's convenience). Such a perspective is first put forward in  \cite{hru:loe:lef} for the purpose of  finding a resolution-free construction of the complex  motivic Milnor fiber, among other things  (see also \cite{Nic:Pay:trop:fub, Thuong} for further developments).

To begin with, by an (algebraic) variety over a field $\bfk$, we mean a reduced separated $\bfk$-scheme of finite type. We denote by $\Var_{\bfk}$ the category of  varieties over $\bfk$.

The Grothendieck semiring $\gsk \mdl C$ of a category $\mdl C$ is the free semiring generated by the isomorphism classes of $\mdl C$, subject to the usual scissor relation $[A \mi B] + [B] = [A]$ when $B$ is a subobject of $A$, where $[A]$, $[B]$ denote the isomorphism classes and ``$\mi$'' is certain binary operation, usually just set subtraction; additional relation may be imposed, to be determined in context. Sometimes $\mdl C$ is also equipped with a binary operation --- for example, cartesian product of sets or (reduced) fiber product of varieties --- that induces multiplication in $\gsk \mdl C$, in which case $\gsk \mdl C$ becomes a commutative semiring. The formal groupification $\ggk \mdl C$ of  $\gsk \mdl C$ is then a commutative ring. If a group $G$ acts on the objects of $\mdl C$ and the morphisms of $\mdl C$ are $G$-equivariant, that is, they commute with $G$-actions, then the corresponding $G$-equivariant Grothendieck ring is denoted by $\ggk^G \mdl C$. If $G = \lim_n G_n$ is profinite then we shall always impose the condition that a $G$-action factor through some $G_n$-action. The archetypal example is the Grothendieck ring $\gmv$ of varieties over $\C$ with good $\hat \mu$-actions (with an additional condition that identifies linear actions on affine spaces with the trivial one), where $\hat \mu$ is the procyclic group of roots of unity.

In this introduction, for simplicity, we shall just consider a nonconstant polynomial function $f : (\C^d, 0) \fun (\C, 0)$ such that $0$ is a singular point, that is, $\nabla f (0) = 0$. For $0 < \eta \ll \delta \ll 1$, the topological type (or even the diffeomorphism type) of the set $F_a = \bar B(0,  \delta) \cap f^{-1}(a)$, where $\bar B(0,  \delta)$ is the closed ball of radius $\delta$ centered at $0$, is independent of the choice of $\eta$, $\delta$, and  $a \in (0, \eta]$. This topological type, referred to as the (closed) \memph{Milnor fiber} of $f$,  is denoted by $F_{f}$. The open Milnor fiber, where the open ball $B(0,  \delta)$ is used, is also of interest, but more so in the real environment than in the complex one. We will come back to this later.

Let $\mathscr L$ be the space of formal arcs on $\C^d$ at $0$. So each element in $\mathscr L$ is of the form $\gamma(t) = (\gamma_1(t), \ldots, \gamma_d(t))$, where $\gamma_i(t)$ is a complex formal power series with $\gamma_i(0) = 0$. Let $\mathscr L_m$ be the space of such arcs modulo $t^{m+1}$ (also referred to as ``truncated arcs''). Consider the  subset of $\mathscr L_m$:
\begin{equation*}
  \mdl X_{f,m}  = \set{\gamma(t) \in \mathscr L_m \given f(\gamma(t)) = t^m \mod t^{m+1} }.
\end{equation*}
It may be viewed as the set of closed points of an algebraic variety over $\C$ and carries a natural $\mu_{m}$-action. The motivic zeta function attached to $f$ is then the generating series whose coefficients are in effect the ``$\hat \mu$-equivariant motivic volumes'' of the sets of truncated arcs above:
\begin{equation}\label{old:zeta}
 Z_f(T) \coloneqq \sum_{m \geq 1} [\mdl X_{f,m}] [\A]^{-nd}T^n \in \gmv[[\A]^{-1}] \dbra{T}.
\end{equation}
Here and below $\A$ denotes the affine line in question.

It is  shown in \cite{denefloeser:arc, den:loe:2002} that $Z_f(T)$ is rational and the motivic Milnor fiber $\mathscr S_{f} \coloneqq - \lim_{T \limplies \infty} Z_{f}(T)$ is then extracted from this rational expression via a formal process of sending the variable $T$ to infinity (this process is also summarized in \cite[\S~8.4]{hru:loe:lef}). Of course, to justify calling $\mathscr S_{f}$ a ``Milnor so-and-so'' one needs to show, at the very least, that invariants of the topological Milnor fiber  $F_f$ can be recovered from it.  This is indeed the case for, say, the Euler characteristic and the Hodge characteristic.

Originally, both the proof that $Z_f(T)$ is rational and the proof that the Euler (or Hodge) characteristics coincide rely on resolution of singularities.  More recently, in \cite{hru:loe:lef}, these results are established by way of a more conceptual construction, namely the Hrushovski-Kazhdan integration. To briefly outline the methodology, we work in  the field $\puC \coloneqq \bigcup_{m \in \Z^+} \C \dpar{ t^{1 / m} }$ of complex Puiseux series. This is the algebraic closure of the field $\C \dpar t$ of complex Laurent series, where a typical element  takes the form $x = \sum_{n \in \Z} a_n t^{n/m}$ for some $m \in \Z^+$ such that, for some $n' \in \Z$,  $a_n = 0$ for all $n < n'$. We think of $\bfk \coloneqq \C$ as a subfield of $\puC$ via the embedding $a \efun at^0$. There is an obvious  valuation map $\vv: \tilde \C^\times \fun \Q$ such that the valuation ring $\OO \coloneqq \C \dbra {t^{\infty}}$ consists of those series with nonnegative exponents and the maximal ideal $\MM$ consists of those series with positive exponents. Its residue field $\K$ admits a section onto $\bfk$ and hence is isomorphic to $\C$. It is well-known that $(\puC, \OO)$ is an algebraically closed valued field.

For a series $x = \sum_{n \in \Z} a_n t^{n/m} \in \puC$ with $\vv(x) = p/m$, let $\rv(x) = a_p t^{p/m}$, which is called the \memph{leading term} of $x$. Then the motivic zeta function attached to $f$ may be expressed as
\begin{equation}\label{intro:zeta:mot}
Z_{f}(T) = \sum_{n \geq 1}  H_m([\mdl X_f]) T^n,
\end{equation}
where the coefficients $H_m([\mdl X_f])$ are certain Hrushovski-Kazhdan integrals of definable sets that take values in $\gmv [[\A]^{-1}]$,  and the so-called \memph{nonarchimedean Milnor fiber of $f$}
\[
\mdl X_f = \set{x \in \MM^d \given \rv(f(x)) = \rv(t) }
\]
is a definable set over the parameter space (the ``ground field'') $\bb S = \C \dpar t$. Formulated in this way, the rationality of $Z_{f}(T)$ essentially follows from  certain computation rules of (convergent) geometric series. That the Euler characteristics of $\mathscr S_{f}$ and $F_{f}$ coincide follows from the fact that we can express both the Euler characteristic of each coefficient of $Z_{f}(T)$ and the Euler characteristic of $F_{f}$ in terms of traces of the monodromy action on the cohomological groups of $F_{f}$, where the first expression relies on  the resolution-free proofs of the A'Campo-Denef-Loeser formula (this is the main point of \cite{hru:loe:lef}) and quasi-unipotence of local monodromy (see \cite[Remark~8.5.5]{hru:loe:lef}).

It is this kind of more conceptual viewpoint --- no arbitrary choice of a resolution for computational purposes --- we aim to emulate and develop further in this paper. Our discussion will lean toward real geometry, because that is where some of our new results are more pronounced. Here ``real geometry'' is broadly construed and may mean the study of varieties over $\R$ or, more significantly, real varieties in the sense of  \cite{BCR} (real points of varieties over $\R$), or even semialgebraic (more generally, \omin-minimal) geometry. Accordingly,  there is the issue of choosing or formulating an appropriate variant of the Hrushovski-Kazhdan integration that reflects the choices of both the kind of  motivic Milnor fiber one wants to construct and the category in which such a construction is carried out. The results are described in  detail below.

\begin{figure}[htb]
\begin{equation}\label{sum:narr}
\bfig
\iiixiii(0,0)/->`->`->`->``->``->`->```/<1100,400>[{\ggk \mgVF^\diamond[*]}`{\ggk  \mgRV^{\db}[*] / (\bm P_\Gamma)}`{\sggk \RES}`{\ggk \VF_*}`{\ggk \RV[*] / (\bm P - 1)}`{\sggk \RES}`{\ggk \VF_{\puR}}`{\ggk \RV_{\puR}[*] / (\bm P - 1)}`{\sggk \RES_{\puR}};\int^{\diamond}`\bb E^{\diamond}````\int``-/ {\gF_{\puR}}`-/{\gR_{\puR}}```]
\iiixii(0,-400)|mmmmmbb|/->``->`->`->`->`/<1100,400>[{\ggk \VF_{\puR}}`{\ggk \RV_{\puR}[*] / (\bm P - 1)}`{\sggk \RES_{\puR}}`{\ggk \TVF_*}`{\ggk \TRV[*] / (\bm P - 1)}`{\ggk \TRES};\int_{\puR}`````\int^T`]
\morphism(1100, 400)|a|/@<25\ul>/<1100,0>[{\ggk \RV[*] / (\bm P - 1)}`{\sggk \RES};\bb E_b]
\morphism(1100, 400)|b|/@{.>}@<-25\ul>/<1100,0>[{\ggk \RV[*] / (\bm P - 1)}`{\sggk \RES};\bb E_g]
\morphism(1100, 0)|a|/@<25\ul>/<1100,0>[{\ggk \RV_{\puR}[*] / (\bm P - 1)}`{\sggk \RES_{\puR}};\bb E_{b, \puR}]
\morphism(1100, 0)|b|/@{.>}@<-25\ul>/<1100,0>[{\ggk \RV_{\puR}[*] / (\bm P - 1)}`{\sggk \RES_{\puR}};\bb E_{g, \puR}]
\morphism(1100, -400)|a|/@<25\ul>/<1100,0>[{\ggk \TRV[*] / (\bm P - 1)}`{\ggk \TRES};\bb E^T_{b}]
\morphism(1100, -400)|b|/@{.>}@<-25\ul>/<1100,0>[{\ggk \TRV[*] / (\bm P - 1)}`{\ggk \TRES};\bb E^T_{g}]
\square(2200, 400)|amrm|/->`->`->`->/<800,400>[{\sggk \RES}`\gdv`{\sggk \RES}`\gdv; \Theta`\id`\id`\Theta]
\square(2200, 0)|amrm|/`->`->`->/<800,400>[{\sggk \RES}`\gdv`{\sggk \RES_{\puR}}`\gsv; `-/{\gR_{\puR}}`\Xi`\Theta_{\puR}]
\square(2200, -400)/`->`->`->/<800,400>[{\sggk \RES_{\puR}}`\gsv`{\ggk \TRES}`\Z; ``\chi^{BM}`\chi]
\efig
\end{equation}
\end{figure}

Suppose that the nonconstant polynomial function $f$ is defined over $\R$.

To begin with, we may still  work in the framework of \cite{hru:loe:lef}, that is, the original Hrushovski-Kazhdan integration theory as applied to the categories of definable sets in the $\lan{}{RV}{}$-structure $\puC$, which is an $\ACVF$-model. For the  formal definitions of the first-order language $\lan{}{RV}{}$ and the $\lan{}{RV}{}$-theory $\ACVF$ of algebraically closed valued field of equal characteristic $0$, we refer the reader to \cite[\S~2]{Yin:special:trans}.  The two sorts $\VF$, $\RV$ of $\lan{}{RV}{}$ are  interpreted as  $\puC$,  $\tilde \C^\times / (1 + \MM)$ (or, equivalently, the group of leading terms) and the cross-sort function $\rv : \VF^\times \fun \RV$ as the quotient homomorphism (or the leading term map described above). The homomorphism from $\RV$ onto the value group $\Gamma = \Q$, also referred to as the $\Gamma$-sort, with the kernel $\K^\times$, is denoted by $\vrv$. All this is encapsulated in the commutative diagram
\begin{equation}\label{LRV:diag}
\bfig
 \square(0,0)/^{ (}->`->>`->>`^{ (}->/<600, 400>[\OO \mi \MM`\VF^{\times}`\K^{\times}`
\RV;`\text{quotient}`\rv`]
 \morphism(600,0)/->>/<600,0>[\RV`\Gamma;\vrv]
 \morphism(600,400)/->>/<600,-400>[\VF^{\times}`\Gamma;\vv]
\efig
\end{equation}
where the bottom sequence is exact. There indeed exists a natural isomorphism $\RV \fun \Q \oplus \C^\times$ given by $a_qt^q \efun (q, a_q)$, though it is not definable.

Since  we intend to study real geometry, the parameter space  for definable sets should not be $\C \dpar t$ as in \cite{hru:loe:lef} but rather $\R \dpar t$, and the Galois group $\gal(\puC / \R \dpar t)$ is then identified with the profinite group $\hat \delta \coloneqq \hat \mu \rtimes \gal(\puC / \puR)$, where $\puR$ is the field of real Puiseux series, that is, the real closure of $\R \dpar t$.

The category $\VF_*$ consists of the definable subsets of $\VF^n$, $n \geq 0$, as objects (alternatively, the definable subsets of varieties over $\R \dpar t$) and the definable bijections between them as morphisms. The category $\RV[k]$ essentially consists of the finite covers of definable subsets of $\RV^k$ as objects and the definable bijections between them as morphisms. The category  $\RV[*]$  is the coproduct of $\RV[k]$, $k \geq 0$, and hence is equipped with a gradation by ambient dimensions. One of the  main results of \cite{hrushovski:kazhdan:integration:vf} is the canonical isomorphism $\int$ in (\ref{sum:narr}) between the   Grothendieck rings, where $\bm P$ stands for the element $[\rv(1 + \MM)] - [\rv(\MM \mi 0)]$  in $\ggk \RV[1]$ (so the principal ideal $(\bm P - 1)$ is not homogenous).

The structure of $\ggk \RV[*]$ can be significantly elucidated. To wit, it is isomorphic to a tensor product of two other Grothendieck rings $\ggk \RES[*]$ and $\ggk \Gamma[*]$, where  $\RES[*]$ is  the category of twisted constructible sets in the residue field $\K$  and $\Gamma[*]$ is the category of definable sets in the value group $\Gamma$ (as an \omin-minimal group), both are graded by ambient dimensions. The objects of $\RES[*]$ are twisted because the short exact sequence at the bottom of (\ref{LRV:diag}) does not admit a natural splitting, and $\ggk \Gamma[*]$ is not the Grothendieck ring of \omin-minimal groups because not all definable bijections are admitted as morphisms. Anyway, we have two retractions from $\ggk \RV[*]$ onto a quotient $\sggk \RES$ of $\ggk \RES$ (the gradation is forgotten), reflecting the fact that there are two Euler characteristics in the $\Gamma$-sort; these are labeled $\bb E_{b}$, $\bb E_{g}$ in (\ref{sum:narr}). Here $\bb E_{g}$ really takes value in $\sggk \RES[[\A]^{-1}]$; indeed, all the corresponding rings downstream thence should be localized at $[\A]$ (we take $[\A] = -1$ for $\Z$), but this is not shown for  brevity.

The isomorphism $\Theta$ is constructed as in \cite[\S~4.3]{hru:loe:lef}.

The motivic zeta function $Z_f(T)$ now resides in $\gdv[[\A]^{-1}] \dbra{T}$. However,  the coefficients of $Z_f(T)$ requires  a kind of crude volume forms and the  integral $\int$ (or other variants in \cite{hrushovski:kazhdan:integration:vf}) is not adequate for the task. Significant modifications are in order. This work has been carried out in  \cite{HL:modified} in order to  correct a technical oversight in  \cite{hru:loe:lef}, resulting in the canonical isomorphism $\int^\diamond$ in (\ref{sum:narr}). The category $\mgVF^{\diamond}[*]$  consists of the proper invariant objects of $\VF_*$  and the category $\mgRV^{\db}[*]$ the doubly bounded objects of $\RV[*]$, all  equipped with $\Gamma$-volume forms (see \cite[Definitions~6.34, 4.9]{HL:modified}). The nonarchimedean Milnor fiber $\mdl X_f$ of $f$ is an object of $\mgVF^{\diamond}[*]$ (with the trivial volume form). Note that $\mgVF^{\diamond}[*]$  is also graded since, as in classical measure theory, gradation by ambient dimensions is a necessity in the presence of volume forms (a curve has volume zero if considered as a subset of a surface). Also, the ideal $(\bm P_\Gamma)$ is homogenous but is no longer principal. We may again express $\ggk  \mgRV^{\db}[*]$ as a tensor product of two other Grothendieck rings $\ggk \mRES[*]$ and $\ggk \mG^{\db}[*]$. Since the objects of $\mG^{\db}[*]$ are doubly bounded, the two Euler characteristics coincide and consequently there is only one retraction onto $\sggk \RES$, which is labeled $\bb E^{\diamond}$ in (\ref{sum:narr}).

The henselian field $\C \dpar{ t^{1/m} }$,  $m \in \Z^+$, is considered as an $\lan{}{RV}{}$-substructure of $\puC$ and, as such, its value group $\Gamma(\C \dpar{ t^{1/m} })$ is identified with $m^{-1} \Z$. Corresponding to each $\C \dpar{ t^{1/m} }$ there is a homomorphism $\bm h_m$ from a subring $\ggk^\natural \mgRV^{\db}[*]$ of $\ggk \mgRV^{\db}[*]$ into $\gdv[[\A]^{-1}]$ that vanishes on $(\bm P_\Gamma)$. The integral $\int^\diamond [\mdl X_f]$ indeed lands in  $\ggk^\natural \mgRV^{\db}[*]$ and the coefficients $H_m([\mdl X_f])$ in (\ref{intro:zeta:mot}) are given by $\bm h_m ( \int^\diamond [\mdl X_f])$. Then $\mathscr S_{f}$, that is, $- \lim_{T \limplies \infty} Z_{f}(T)$, is equal to
\[
\Big( \Theta \circ \bb E^{\diamond} \circ \int^\diamond \Big)([\mdl X_f]) = \Big( \Theta \circ \bb E_b \circ \int \Big)([\mdl X_f]) \in \gdv[[\A]^{-1}].
\]
Of course the element $(\Theta \circ \bb E_b \circ \int)([\mdl X_f])$ may be attached to $f$ directly, but to establish its significance, we need to compare it with the zeta function construction. It is this reason that forces us to work with an integral whose target only involves doubly bounded sets in $\RV$, namely $\int^{\diamond}$, instead of $\int$, so as to facilitate the computation of the coefficients of $Z_f(T)$.

Without the top row,  (\ref{sum:narr}) commutes with the dotted arrows too. The element $(\Theta \circ \bb E_g \circ \int)([\mdl X_f])$ may be attached to $f$ directly as well, but then its geometric significance is unclear, except in the bottom row. We will say more about this below.

Let $\RVar$ be the category of real varieties in the sense of \cite{BCR}.  Taking real points and forgetting the $\hat\delta$-actions, we can specialize $H_m([\mdl X_f])$ to $\ggk \RVar$ and thereby obtain the real motivic Milnor fiber of $f$ in $\ggk \RVar[[\A]^{-1}]$.  However, we are more interested in a subtler construction  that is indigenous to the real algebraic environment.

Since $f$ is assumed to be defined over $\R$, it may be realized as a real function $(\R^d, 0) \fun (\R, 0)$. The  open and closed Milnor fibers are constructed as before, but denoted by $F_f^+$, $\bar F_f^+$ since, in the absence of monodromy, replacing $(0, \eta]$ with $[-\eta, 0)$ will, in general, result in different  topological types $F_f^{-}$, $\bar F_f^{-}$. So the qualifiers  ``positive'' and ``negative'' should be tagged on in the terminology if we are to look at the whole picture. The difference between $F_f^+$ and $\bar F_f^+$ is more significant in real geometry.

The sets of real truncated arcs are denoted by $\mathscr L_m(\R)$. Replacing $\mathscr L_m$ with $\mathscr L_m(\R)$ in $\mdl X_{f, m}$, we get a real variety $\mdl X_{f, m}^1$. The complexification $\mdl X_{f, m}^1 \otimes \C$ of $\mdl X_{f, m}^1$ is a variety over $\C$, which is isomorphic to $\mdl X_{f, m}$, and carries a natural $\delta_m$-action, where $\delta_m = \mu_m \rtimes \gal(\C / \R)$. Consequently, $\mdl X_{f, m}^1$ inherits a natural $\mu_2$-action from  $\mdl X_{f, m}^1 \otimes \C$. This is indeed how  the homomorphism $\Xi$ in (\ref{sum:narr}) is constructed.

As a subfield, $\puR$ inherits from $\puC$ a valuation map, a valuation ring, etc. The pair $(\puR, \OO(\puR))$ forms a henselian valued field. There is a general procedure to specialize the integral $\int$ to sets in any henselian subfield of $\puC$, in particular, for those in $\tilde \R $ over $\R \dpar t$. 
The corresponding homomorphisms between the Grothendieck rings are marked in the third row of arrows in (\ref{sum:narr}).

Applying $\Xi$ termwise to $Z_f(T)$ brings about a (positive) motivic zeta function $Z^1_f(T)$, which belongs to $\gsv[[\A]^{-1}] \dbra T$; there is of course a negative one too. The rationality of  $Z_f^{1}(T)$ and hence the existence of the real motivic Milnor fiber $\mathscr S_{f}^1$ in $\gsv[[\A]^{-1}]$ follows. Let $\mdl X_f^{1}$ be the $\puR$-trace of $\mdl X_f$. The image of $[\mdl X_f]$ in $\ggk \VF_{\puR}$ is $[\mdl X_f^{1}]$ and hence $\mathscr S_{f}^1$ may indeed be computed purely in the real algebraic environment as $(\Theta_{\puR} \circ \bb E_{b, \puR} \circ \int_{\puR})([\mdl X_f^1])$.

The next step is to justify calling $\mathscr S_{f}^1$ a Milnor fiber by recovering invariants of $\bar F_f^+$ from $\mathscr S_{f}^1$. Actually the  only known additive invariant of $\bar F_f^+$ is the topological (or semialgebraic) Euler characteristic $\chi(\bar F_f^+)$. It is shown in \cite[Theorem~4.4]{Comte:fichou}  that $\chi(\bar F_f^+)$ does agree with  $\chi^{BM}(\mathscr S_{f}^{1})$, where $\chi^{BM}$ is the Borel-Moore Euler characteristic, also labeled as such in (\ref{sum:narr}); note that  the real motivic Milnor fiber in \cite{Comte:fichou} is the forgetful image of $\mathscr S_{f}^{1}$ in $\ggk \RVar[[\A]^{-1}]$. Their method relies on a real analogue of the A'Campo-Denef-Loeser formula, which  needs resolution of singularities. Unfortunately, upon the absence of  monodromy in the real environment, we cannot follow the method of \cite{hru:loe:lef} outlined above to get a resolution-free proof, at least  not without further elucidating the effect of the monodromy action on the complexification of the real Milnor fibers as suggested by \cite{McParu:mono}.

Going through a different route, we use the theory of motivic integration for \T-convex valued fields as developed in \cite{Yin:tcon:I}. This theory is rich in expressive power and hence can handle all the definable objects in the algebraic environment. On the other hand, its expressive power is also its limitation in yielding algebro-geometric information since, in the corresponding categories of definable sets, there are much more morphisms that can cause loss of algebro-geometric data when passing to the Grothencieck rings. Nevertheless, it should retain much of the  numerical information.

We work in $\puR$, which is now viewed as a real closed field equipped with both a total ordering and a valuation (or more generally a polynomially bounded \T-convex valued field). This structure is expressed in a first-order language $\lan{T}{RV}{}$, which still has two sorts $\VF$ and $\RV$. The categories $\TVF_*$, $\TRV[*]$, $\TRES[*]$, etc., are all defined similarly as before. Again, there are the canonical isomorphism $\int^T$ in (\ref{sum:narr}) between the Grothendieck rings (this is the so-called generalized Euler characteristic of definable sets in $\puR$), the tensor expression $\ggk \TRES[*] \otimes \ggk \TG[*]$ of $\ggk \TRV[*]$, and the two retractions $\bb E^T_b$, $\bb E^T_g$ in (\ref{sum:narr}). The definable sets in the residue field are precisely the semialgebraic sets and hence   $\ggk \TRES$ is canonically isomorphic to $\Z$; this is labeled $\chi$  in (\ref{sum:narr}) since it is indeed the semialgebraic Euler characteristic.

Applying $\chi^{BM}$ termwise to  $Z^1_f(T)$, we obtain a power series in $\Z \dbra T$, which is understood as a topological zeta function attached to $f$. The definable set $\mdl X^{1}_f$ may be approximated by a sequence of semialgebraic sets $\bar F_r$, $r \in \R^+$, whose semialgebraic homology eventually stabilizes. The Euler characteristic of this stabilized semialgebraic homology is equal to, on the one hand, $\chi(\bar F_f^+)$ and, on the other hand, $(\chi \circ \bb E^T_{b} \circ \int^T) ([\mdl X_f^{1}])$ and hence $\chi^{BM}(\mathscr S_{f}^1)$.

The same argument shows that $\chi(F_f^+) = (\chi \circ \bb E^T_{g} \circ \int^T) ([\mdl X_f^{1}])$, where $\bar F_f^+$, $\bb E^T_{b}$ are replaced by $F_f^+$, $\bb E^T_{g}$. As a corollary, we get
\[
\chi([\bar{F}_f^+]) = (-1)^{d+1} \chi([F_f^+]).
\]
This comes from an equality at the motivic level (the second and third rows in (\ref{sum:narr})), and may be construed as a specialization of Bittner's computation of the dual of motivic Milnor fiber in \cite{B2}.

This approach also works in the complex setting, considering $\puC$ as $\tilde \R^2$ and hence $\mdl X_f$ as an object of $\TVF_*$. It shows in particular that $\chi(F_f)$ is equal to the Euler characteristic of $\mathscr S_{f}$, as in \cite[Remark~8.5.5]{hru:loe:lef}, but without using  even  quasi-unipotence of local monodromy. Note that, over $\C$, the Euler characteristics of the open and  closed Milnor fibers coincide, so if $\mathscr S_{f}$ encodes information on both of them, one cannot see it at this level.

We can extend $\Theta \circ \bb E_b \circ \int$ further by composing the Hodge-Deligne polynomial map. According to  \cite[Reamrk~3.24]{Nic:Pay:trop}, this actually gives the Hodge-Deligne polynomial of the limit mixed Hodge structure associated with a variety over $\C \dpar t$. Extending $\Theta_{\puR} \circ \bb E_{b, \puR} \circ \int_{\puR}$ by composing the  virtual Poincar\'e polynomial map, we get a similar homomorphism into $\Z[u]$. It would be interesting to investigate if it too encodes information on limit structures. But of course  we are ahead of ourselves here because such limit structures are not yet available in the real setting.


Finally, in showcasing the potential of the framework underlying (\ref{sum:narr}), we describe another main result, namely a new (local) Thom-Sebastiani formula in mixed variables, extending that in \cite{guibert2006} (the results in \cite{GLM:IMRN, GLM:MRL} are for separate  variables and hence overlap to a much lesser extent with the case we establish here). A precursor of our method has already been used in \cite{Thuong} to recover the Thom-Sebastiani formula of \cite{DL:tom:seb, Loo:mot}. This part of the paper is somewhat independent of the previous discussion that concentrates on real geometry; on the other hand, since here we completely abandon the zeta function point of view, one does need to be convinced, at the outset, of the significance of the construction $(\Theta \circ \bb E_b \circ \int)([\mdl X_f])$.

We still work in $\puC$, but change the parameter space to ${\bfk} \cup {\Gamma} \cong \C \cup \Q$.  Unlike in the previous situations, the (model-theoretic) automorphism group $\aut(\puC / \C \cup \Q)$ is much larger than the automorphism group $\aut(\RV / \kuq)$. It is this latter group, henceforth abbreviated as $\hat \tau$, that we need. This group is isomorphic to  $\lim_n \C^\times_n$, where each $\C^\times_n$ is just a copy of $\C^\times$ and the transition morphisms are the same as in  $\hat \mu = \lim_n \mu_n$.  More concretely, the elements in $\hat \tau$  may be identified as sequences $\hat a=(a_n)_n$ of $n$th roots of $a$, $a\in \C^\times$, satisfying $a_{kn}^n=a_k$. Such an element acts on $\puC$ by the equation $\hat a \cdot t^{1/n}= a_n t^{1/n}$.

The Thom-Sebastiani formalism is typically concerned with expressing the Milnor fiber of a compound function $h(f_1, \ldots, f_l)$ in terms of the Milnor fibers of the component functions $f_1, \ldots, f_l$. The classical results and most of the later generalizations can only handle the case of separate variables, that is, $h(f_1, \ldots, f_l)$ is regarded as a function on the product $\prod_i X_i$, where $X_i$ is the source variety of  $f_i$, and often $h$ is just a linear form. Our formula, on the other hand, is much more sophisticated.

For functions $\phi : X \fun A$ and $\psi : Y \fun B$, we write  $\phi \oplus \psi$ for the function $X \cap Y \fun A \times B$ given by $x \efun (\phi(x), \psi(x))$.

Let  $f, g : (\C^d, 0) \fun (\C, 0)$ be nonconstant polynomial functions, singular at $0$, and $h(x, y)$ a polynomial of the form $y^{N} + \sum_{2 \leq \imath \leq \ell} x^{m_{\imath}}$; we may rename $N$ as  $m_1$, but its role is somewhat different and hence is denoted differently (the case $\ell = 2$ is dealt with in \cite{guibert2006}). For each $1 \leq \imath \leq \ell$, let
$f_{(\imath)} = \sum_{2 \leq i \leq \imath} f^{m_i} : (\C^d, 0) \fun (\C, 0)$ and $\vta^{(\imath)} = (m_2 / m_\imath, m_i / m_\imath)_{2 \leq i \leq \imath}$; here $f_{(1)}$ is interpreted as the zero function and $\vta^{(1)}$ as $1 \in \Q$. Let $g^N_{\imath} \oplus f_{\imath}$ denote the restriction of $g^N \oplus f$ to the set ${\MM^d} \cap (\vv \circ (g^N \oplus f))^{-1}(m_2 / m_\imath, 1 / m_\imath)$.


The category $\Var_{\C}^{\vta^{(\imath)}}$ consists of varieties over $\G_m^2$ with  $(\vta^{(\imath)}, n)$-diagonal $\G_m$-actions; see \S~\ref{cat:ang} for the unexplained terms. Each object of $\Var_{\C}^{\vta^{(\imath)}}$ may be thought of as equipped with a $\hat \tau$-action that factors through, for some $n \in \Z^+$, the canonical epimorphism $\tau_n : \hat \tau \fun \C^\times_n$, and the corresponding Grothendieck ring is denoted by $\ggk^{\vta^{(\imath)}} \Var_{\C}$. If $\imath = 1$ then we abbreviate $\Var_{\C}^{\vta^{(\imath)}}$, $\ggk^{\vta^{(\imath)}} \Var_{\C}$ as $\Var_{\C}^{\bm 1}$, $\ggk^{\bm 1} \Var_{\C}$; actually $\Var_{\C}^{\bm 1}$ is just the category $\Var_{\G_m}^{\G_m}$ in \cite{guibert2006} and hence is equivalent to the category of varieties over $\C$ with good $\hat \mu$-actions. There is a $\ggk^{\hat \tau} \Var_{\C}$-module homomorphism
\[
\Psi_{\vta^{(\imath)}} : \ggk^{\vta^{(\imath)}} \Var_{\C} \fun \ggk^{\bm 1} \Var_{\C},
\]
which is referred to as a convolution operator.

Suppose that $m_2 \ll N \ll m_3 \ll \ldots \ll m_\ell$. Then there is an operator $\Theta^{\ac}_{\vta^{(\imath)}} \circ \bb E^{\ac}_{b, \vta^{(\imath)}} \circ \int^{\ac}_{\vta^{(\imath)}}$ on classes of functions of the form $g^N_{\imath} \oplus f_{\imath}$, with target $\ggk^{\vta^{(\imath)}} \Var_{\C}$, which may be roughly understood as $\Theta \circ \bb E_{b} \circ \int$ applied fiberwise. Abbreviate $(\Theta^{\ac}_{\vta^{(\imath)}} \circ \bb E^{\ac}_{b, \vta^{(\imath)}} \circ \int^{\ac}_{\vta^{(\imath)}})([g^N_{\imath} \oplus f_{\imath}])$ as $\mathscr S^\sharp_{g^N_{\imath} \oplus f_{\imath}}$; we call it the \memph{motivic Milnor fiber of $g^N_{\imath} \oplus f_{\imath}$ over $\G_m^2$}. Then our Thom-Sebastiani formula states that, in $\ggk^{\bm 1} \Var_{\C} \cong \gmv$, $\mathscr S_{h(f, g)}^\sharp$ is equal to
\begin{equation}\label{intro:TS}
\mathscr S_{g^N}^\sharp([Z_{f}]) + \mathscr S_{f^{m_{2}}}^\sharp + \sum_{2 < \imath \leq \ell} \mathscr S^\sharp_{f^{m_{\imath}}}([Z_{g^N + f_{(\imath-1)}}]) - \sum_{2 \leq \imath \leq \ell} \Psi_{\vta^{(\imath)}} (\mathscr S^\sharp_{g^N_{\imath} \oplus f_{\imath}});
\end{equation}
here the first and the third terms are the motivic Milnor fibers  over $\G_m$ but restricted to the indicated zero sets (in \cite{guibert2006} a variant of this is called iterated motivic vanishing cycles).

As before, the whole construction can be specialized to the real setting if $f$, $g$ are defined over $\R$, of which  the Thom-Sebastiani formula obtained in \cite{Campesato} is a special case.

A novel perspective behind (\ref{intro:TS}) is that $\mathscr S_{h(f, g)}^\sharp$ may be decomposed into terms corresponding to combinatorial data that can be read off of the tropical curve of $h(x,y)$. This actually suggests that our method can handle polynomials more complicated than $h(x,y)$, for instance, those with more variables and even mixed terms. However,  the complexity of the combinatorics involved will become quite heavy, perhaps disproportionately so, as it is unclear how the ground gained can shed new light on the geometry and topology of the singularities in question. Thus we have chosen to just present a simple case that is already beyond what is known in the literature.

\subsubsection*{Ambient  and parameter spaces}
For the general integration theory summarized in \S~\ref{sec:HK:main}, we work in a   sufficiently saturated $\ACVF$-model $\bb U$ with parameters in an arbitrary substructure $\bb S$ that is definably closed.

Throughout \S~\ref{section:milnor}, the parameter space  $\bb S$ is $\R \dpar{ t }$. In \S~\ref{section:spec:hen} is described a general descent procedure from $\bb U$ to an arbitrary henselian substructure $\bb M$. For the rest of the section, the pair $(\bb U, \bb M)$ is specialized to $(\puC, \puR)$.

In \S~\ref{sec:tcon}, we work in the $\TCVF$-model $\puR$ with all parameters allowed (restricting to $\R \dpar{ t }$ makes no sense in the presence of a total ordering). There also appears in \S~\ref{top:complex} the ambient space $\puC$ with the parameter space $\C \dpar{ t }$,  both as interpreted in $\puR$.

We have mentioned above that the complex Thom-Sebastiani formula is obtained  in $\puC$ with $\bb S = \C \cup \Q$, treating $\Gamma \cong \Q$ as a definable sort in the model-theoretic sense. For the real case in \S~\ref{sect:TS:real}, $\bb S$ is changed to $\R \cup \Q$ so that descent from $\puC$ to $\puR$ may be carried out.

The parameter space may change in context, which we do not always point out. For instance, if $\bb S = \C \cup \Q$ then studying a $t$-definable set  in $\puC$ in effect changes $\bb S$ to $\C \dpar{ t }$.

\section{Hrushovski-Kazhdan style integration}\label{sec:HK:main}

The first part of the paper relies heavily on \cite{HL:modified}, and hence we shall make extensive use of the notations and terminologies therein, starting with those in \cite[\S~2.1]{HL:modified}; pointers will be provided along the way.

In this section, following the tradition in the model-theoretic literature, we work in a sufficiently saturated model $\bb U$ of $\ACVF$, together with a fixed parameter space $\bb S $,  which is a substructure of $\bb U$. This is of course a matter of convenience, otherwise one needs to change the model one is working in whenever compactness is applied. We assume that  the map $\rv$ is surjective in $\bb S$ (but the value group $\Gamma(\bb S)$ of $\bb S$ could be trivial) and the definable closure $\dcl \bb S$ of $\bb S$ equals $\bb S$. Among other things, this latter condition implies that if $\Gamma(\bb S)$ is nontrivial then the underlying valued field of $\bb S$ is henselian (in fact this is equivalent to the condition $\dcl \bb S = \bb S$). So by a definable set  we shall always mean an $\bb S$-definable set, unless indicated otherwise.

\begin{rem}\label{imag:Gam}
Semantically, we shall treat the value group $\Gamma$ as a definable sort (the $\Gamma$-sort) consisting of imaginary elements, that is, classes of definable equivalence relations. However, syntactically, any reference to $\Gamma$ may be eliminated in the usual way and we can still work with $\lan{}{RV}{}$-formulas for the same purpose.

If $\gamma \in \Gamma$ is definable then it is in the divisible hull $\Q \otimes \Gamma(\bb S)$ of $\Gamma(\bb S)$, and vice versa. This does not mean, though, that the definable set $\gamma^{\sharp} = \vrv^{-1}(\gamma) \sub \RV$ (see \cite[Notation~2.7]{HL:modified}) contains a definable point unless $\gamma \in \Gamma(\bb S)$.
\end{rem}

\begin{rem}\label{pillars}
A pillar of the theory of definable sets in $\bb U$ is \memph{\cmin-minimality}, meaning that every definable subset of $\VF$ is a boolean combination of  (definable) valuative discs. Another one is the so-called \memph{orthogonality} between the $\K$-sort and the $\Gamma$-sort, meaning that every definable subset $A$ of $\bb U^n$ with $\pr_{\leq k}(A)$ in $\K$ and $\pr_{> k}(A)$ in $\Gamma$ (see \cite[Notation~2.1, Terminology~2.3]{HL:modified}) is a finite union of products $A' \times A'' \sub \K^k \times \Gamma^{n-k}$; in particular, if $A$ is the graph of a function on $\pr_{\leq k}(A)$ or $\pr_{> k}(A)$ then its image is finite.

Various regions in $\bb U$, such as the sorts $\RV$, $\K$, and $\Gamma$, are \memph{stably embedded}. In our context, this simply means that, for instance, if a set in $\RV$ is definable then it is $\RV(\bb S)$-definable,
\end{rem}

The definitions of the various categories of definable sets, with or without $\Gamma$-volume forms, are all listed in \cite[\S~3]{HL:modified} in full detail, which  shall not be repeated here. We only remark that ambient dimension plays a role in determining the volume, whatever that means, of a definable set, and this is why for an object $(A, \omega) \in \mgVF[k]$, $A_{\VF}$ is required to be a subset of $\VF^k$, whereas an object $A \in \VF[k]$ is only required to be of dimension  at most $k$.

There is  a homomorphism of graded semirings
\begin{equation}\label{tensor}
  \Psi: \gsk \RES[*] \otimes_{\gsk \Gamma^{\fin}[*]} \gsk \Gamma[*] \fun \gsk \RV[*],
\end{equation}
which is determined by the assignment
\begin{equation}\label{assgn:Psi}
([(U, f)], [I]) \efun [(U \times I^\sharp, f \times \id)].
\end{equation}
There is also the version with volume form
\[
{\mu}\Psi: \gsk \mgRES[*] \otimes_{\gsk \mG^{\fin}[*]} \gsk \mG[*] \fun \gsk \mgRV[*].
\]
The groupifications of these homomorphisms are  denoted the same.

\begin{prop}[{\cite[Prop.~10.10(1)]{hrushovski:kazhdan:integration:vf}}]\label{red:D:iso}
Both $\Psi$ and ${\mu}\Psi$ are isomorphisms of graded semirings.
\end{prop}

Note that \cite[Proposition~10.10(2)]{hrushovski:kazhdan:integration:vf} does not hold. This oversight has caused issues for certain constructions in \cite{hru:loe:lef} that depend on it. These issues have now been resolved in \cite{HL:modified}. The modified construction will be  summarized below.

Recall from \cite[Notation~3.20]{HL:modified} that  $\isp$ is the (nonhomogenous) semiring congruence relation on $\gsk \RV[*]$ generated by the pair $([1], [\RV^{\circ\circ}_\infty])$ and hence the corresponding principal ideal of $\ggk \RV[*]$ is generated by the element $\bm P - 1 \in \gsk \RV[{\leq} 1]$, where $\bm P = [1] - [\RV^{\circ \circ}] \in \ggk \RV[1]$, as we have mentioned above. Similarly, $\mgisp$ is the  semiring congruence relation on $\gsk \mgRV[*]$ generated by the pair $([1], [\RV^{\circ\circ}])$, which is homogenous, and hence the corresponding principal ideal of $\ggk \mgRV[*]$ is generated by the element $\bm P$ (the volume forms here are all $0$, see \cite[Notation~3.19]{HL:modified}).

Recall from \cite[Notation~3.21]{HL:modified} the homomorphisms (and their groupifications)
\[
\mathbb{L} : \gsk \RV[*] \fun \gsk \VF_* \dand \mgL : \gsk \mgRV[*] \fun \gsk \mgVF[*].
\]
We have $\bb L ([1]) = [1 + \MM] = [\MM]$ and $\bb L ([\RV^{\circ\circ}])=[\MM \mi 0]$, and hence $\bb L (\bm P - 1)=0$. Moreover, $\mgL (\bm P)=0$ since $\MM$ and $\MM \mi 0$ are in essential bijection (see \cite[Definition~3.5]{HL:modified}). It so happens that these relations are the only ones needed to describe the kernels of $\bb L$ and $\mgL$.

\begin{thm}\label{main:prop}
For each $k \geq 0$ there is a canonical isomorphism of semigroups
\[
\int_{+} : \gsk  \VF[k] \fun \gsk  \RV[{\leq}k] /  \isp
\]
such that $ \int_{+} [A] = [\bm U]/  \isp$ if and only if $[A] = [{\bb L}{\bm U}]$. Passing to  the colimit of the groupifications, we obtain a canonical isomorphism of rings
\[
\int : \ggk \VF_* \fun \ggk  \RV[*] /  (\bm P - 1).
\]
Similarly, for each $k \geq 0$ there is a canonical isomorphism of semigroups
\[
\int_{+}^\mu : \gsk  \mgVF[k] \fun \gsk  \mgRV[k] /  \mgisp
\]
such that $\int_{+} [\bm A] = [\bm U]/  \mgisp$  if and only if $[\bm A] = [{\mgL}{\bm U}]$. Taking the direct sum of the groupifications, we obtain a canonical isomorphism of graded rings
\[
\int^\mu : \ggk \mgVF[*] \fun \ggk  \mgRV[*] /  (\bm P).
\]
\end{thm}

This is a combination of two main theorems, Theorems~8.8 and 8.29, of \cite{hrushovski:kazhdan:integration:vf}. But it is not enough for our purpose. To recover motivic zeta function and thence motivic Milnor fiber, we shall need  the theory developed in \cite{HL:modified}, complementing the work in \cite{hru:loe:lef}.  We explain here only the gist of this theory. The reader should consult \cite{HL:modified} for full details.

A set, possibly with $\Gamma$-coordinates, is \memph{bounded} if, after applying the maps $\vv$, $\vrv$, $\id$ in the $\VF$-, $\RV$-, $\Gamma$-coordinates, respectively, it is contained in a box of the form $[\gamma, \infty]^n$, and \memph{doubly bounded} if the box is of the form $[- \gamma, \gamma]^n$. An object $(A, \omega) \in \mgVF[k]$ is bounded  or doubly bounded if the graph of $\omega$ is so; similarly in the other categories. In particular, an object $(U, f, \omega) \in \mgRV[k]$ is bounded  if the graphs of $f$ and $\omega$ are both bounded; actually, by \cite[Lemma~3.26]{HL:modified}, if $U$ is doubly bounded then the images of these functions are necessarily doubly bounded.

The full subcategories of  $\mgRV[*]$, $\mG[*]$ of  doubly bounded objects are denoted by $\mgRV^{\db}[*]$, $\mG^{\db}[*]$.
The corresponding restriction of ${\mu}\Psi$ is indeed an isomorphism (see \cite[Remark~4.10]{HL:modified}):
\begin{equation}\label{bdd:to:2bdd}
\gsk \mgRES[*] \otimes \gsk \mG^{\db}[*] \fun \gsk \mgRV^{\db}[*].
\end{equation}

For each $\gamma \in \Gamma$, let  $\MM_\gamma = \set{a \in \VF \given \vv(a) > \gamma }$.  If $\gamma = (\gamma_1, \ldots, \gamma_n) \in \Gamma^n$ then $\MM_{\gamma}$ denotes the product of  $\MM_{\gamma_i}$; each coset of $\MM_\gamma$ in $\VF^n$ is called a \memph{polydisc of radius $\gamma$}. A subset of $\VF^n$ is \memph{$\gamma$-invariant} if it is a union of polydiscs of radius $\gamma$. For example, finite subsets of $\VF$ are not invariant (or rather they are $\infty$-invariant, but $\infty$ is not allowed in the  definition), the maximal ideal $\MM$ is $\gamma$-invariant for every $\gamma \in \Gamma^+$, whereas $\MM \mi 0$ is not $\gamma$-invariant for any $\gamma \in \Gamma^+$, because the radii of its maximal open subdiscs tend to $\infty$ as they approach $0$. A \memph{proper invariant} set is an invariant set $A$ such that $A_{\VF}$ is bounded and $A_{\RV}$ is doubly bounded. For example,  if $\bm U \in \RV^{\db}[k]$ then $\bb L \bm U$ is a proper invariant set (it is actually doubly bounded).

The subcategory $\mVF^{\diamond}[k]$ of $\mVF[k]$ consists the proper invariant objects and the morphisms that are compositions of relatively unary proper covariant homeomorphisms. It is rather involved to make sense of what the morphisms are, the reader is referred to \cite[Definitions~6.1, 6.2, 6.24]{HL:modified}) for detail. A crucial point to keep in mind is that every morphism in $\mVF^{\diamond}[k]$ is an honest  bijection, as opposed to merely an essential bijection, and is in effect required to admit an inverse.  So $\mgVF^{\diamond}[k]$ is already a groupoid and there is no need to pass to a quotient category as described in \cite[Remark~3.7]{HL:modified}, and it makes sense to speak of the forgetful homomorphism $\ggk \mgVF^\diamond[*] \fun \ggk \VF_*$.

\begin{nota}\label{pgamma}
For each $\gamma \in \Gamma^+(\bb S)$, let $\RV^{\circ \circ}_\gamma = \rv(\MM_{\gamma} \mi 0)$ and
\[
\bm P_\gamma = [\RV^{\circ \circ} \mi \RV^{\circ \circ}_{\gamma}] + [\{t_\gamma\}] - [1] \in \ggk \RV^{\db}[1],
\]
where $t_\gamma \in \gamma^\sharp$ is any definable point. So $[\{t_\gamma\}] = [1]$ in $\ggk \RV^{\db}[1]$.
But $t_\gamma $ also stands for an element in $\ggk \mRV^{\db}[1]$ (with the constant volume form $0$), and $[\{t_\gamma\}] \neq [1]$ in $\ggk \mRV^{\db}[1]$  unless $\gamma = 0$.

Observe that  $\bm P_\gamma$ does not depend on the choice of $t_\gamma \in \gamma^\sharp$. The ideal of $\ggk \mRV^{\db}[*]$   generated by the elements $\bm P_\gamma$ is denoted by $(\bm P_\Gamma)$. The images of $(\bm P_\Gamma)$ are contained in $(\bm P - 1)$, $(\bm P)$ under, respectively, the natural (forgetful) homomorphisms
\[
\ggk  \mgRV^{\db}[*] \fun \ggk \RV[*], \quad \ggk  \mgRV^{\db}[*] \fun \ggk  \mgRV[*].
\]
\end{nota}

By  \cite[Corollary~6.37]{HL:modified}, the map ${\mu}\bb L$ induces a surjective homomorphism of graded Grothendieck rings $\ggk \mRV^{\db}[*] \epi \ggk \mVF^{\diamond}[*]$. By  \cite[Proposition~7.24]{HL:modified}, its kernel  is precisely $(\bm P_\Gamma)$.

\begin{thm}[{\cite[Theorem~7.26, Corollary~7.27]{HL:modified}}]\label{main:prop:dia}
There is a canonical isomorphism of graded Grothendieck rings:
\[
 \int^{\diamond} : \ggk \mgVF^\diamond[*] \fun \ggk  \mgRV^{\db}[*] / (\bm P_\Gamma).
\]
It interpolates the two isomorphisms $\int$, $\int^\mu$ in the sense that the following diagram commutes:
\begin{equation}\label{diag-diamond-interpol}
\bfig
  \hSquares(0,0)/<-`->`->`->`->`<-`->/<400>[{\ggk \VF_*}`{\ggk \mgVF^\diamond[*]}`{\ggk \mgVF[*]}`{\ggk \RV[*] / (\bm P - 1)}`{\ggk \mgRV^{\db}[*] / (\bm P_\Gamma)}`{\ggk \mgRV[*] / (\bm P)}; ``\int`\int^{\diamond}`\int^\mu``]
\efig
\end{equation}
\end{thm}

Next, in light of the isomorphism $\Psi$, we have a procedure to  replace classes of polytopes with numbers (\omin-minimal Euler characteristics, see \cite[Remark~4.2]{HL:modified}) and thereby construct two retraction maps from the Grothendieck ring in $\RV$ to that in $\RES$, which is dubbed ``uniform retraction to $\RES$'' (see Example~\ref{mil:exam} for an explicit  computation):

\begin{prop}\label{base:Eb:Eg}
There are two ring homomorphisms
\[
  \bb E_g : {\ggk} \RV[*] \fun \sggk \RES[[\A]^{-1}] \dand \bb E_b : {\ggk} \RV[*]  \fun \sggk \RES[[1]^{-1}] \cong \sggk \RES.
\]
such that
\begin{itemize}[leftmargin=*]
  \item $\bm P - 1 \in \ggk \RV[1]$ vanishes under both of them,
  \item for all $x \in \ggk \RES[k]$ and all $y \in \ggk \Gamma[l]$,
  \begin{equation}\label{eb:eg:defn}
      \bb E_{g}(x \otimes y) = \chi_g(y) x [\G_m]^l[\A]^{-(k+l)} \dand
     \bb E_{b}(x \otimes y) = \chi_b(y)x [\G_m]^l [1]^{-(k+l)},
  \end{equation}
  where, for simplicity, $x \otimes y$ stands for the  element  $\Psi^{-1}(x \otimes y) \in \ggk \RV[*]$.
\end{itemize}
\end{prop}

By the first clause, they may be regarded as homomorphisms on ${\ggk} \RV[*] / (\bm P - 1)$.

\begin{rem}\label{eg:eb:differ}
For a proper invariant object $A$ of $\VF_*$ of dimension $n$, the homomorphisms  $\bb E_g \circ \int$, $\bb E_b \circ \int$ only differ by a factor in $\sggk \RES[[\A]^{-1}]$; see \cite[Remark~7.28]{HL:modified}. Consequently, we have
\begin{equation}\label{off:fac:A}
\bb E_g \Big (\int [A] \Big ) = \bb E_b \Big (\int [A] \Big )[\A]^{-n}.
\end{equation}

In \S~\ref{sec:tcon}, the \T-convex versions of $\bb E_b \circ \int$, $\bb E_g \circ \int$ yield the Euler characteristics of the closed and  open topological Milnor fibers. Then (\ref{off:fac:A}) may be specialized to one between these two numerical quantities; see Corollary~\ref{link:open:closed}.\

More generally, (\ref{off:fac:A}) is a manifestation of the Bittner duality \cite{B1}; this will be explained elsewhere.
\end{rem}

There is also  a version for the doubly bounded category with volume forms:

\begin{prop}\label{prop:eu:retr:k:db}
There is a graded ring homomorphism
\[
  \mgE^{\db}: \ggk \mgRV^{\db}[*] \fun \sggk \RES[*]
\]
such that $(\bm P_\Gamma)$ vanishes and, for all $x \in \ggk \mgRES[*]$, $\mgE^{\db}(x) = \phi(x)$, where $\phi$ is the forgetful homomorphism $\phi : \ggk \mgRES[*] \fun \sggk \RES[*]$.
\end{prop}

See \cite[\S~4]{HL:modified} for the notation and a full explanation on how to obtain these results.

The composition of $\mgE^{\db}$ with the forgetful homomorphism $\sggk \RES[*] \fun \sggk \RES$ is denoted by $\bb E^{\diamond}$.  By \cite[Remark~4.14]{HL:modified}, the diagram commutes:
\begin{equation}\label{edb:eb:com}
\bfig
 \Vtriangle(0,0)/->`->`->/<400, 400>[{\ggk \mgRV^{\db}[*]}`{\ggk \RV[*]}`\sggk \RES; `\bb E^{\diamond}`\bb E_b]
\efig
\end{equation}
which may serve as an alternative and more direct (but less comprehensive) construction of $\bb E^{\diamond}$.

\begin{rem}\label{big:dia:up}
Appending (\ref{edb:eb:com}) to the left square in (\ref{diag-diamond-interpol}), we obtain the upper portion of (\ref{sum:narr}) (the homomorphism $\Theta$ is given in (\ref{hom:theta})).
\end{rem}

\begin{rem}\label{eb:not:eg}
The homomorphism $\bb E_b$ will be used in the construction of motivic Milnor fiber in \S~\ref{section:milnor}, but not  $\bb E_g$, because it does not quite commute with $\bb E^{\diamond}$ (say, by (\ref{off:fac:A}), $\bb E_g([1]) = [\A]^{-1}$ whereas $\bb E^{\diamond}([1]) = \bb E_b([1]) = 1$). It is intriguing to ponder how $\bb E_g$ fit in with the grand scheme of things (more on this in Remark~\ref{g:comm:alm} below).

For the Thom-Sebastiani formula in \S~\ref{section:TS} to hold, we must also use $\bb E_b$ (otherwise certain terms in the computation would not vanish, see Remark~\ref{fub:why:xb}).
\end{rem}

As in \cite{HL:modified}, we shall carry out computations (in \S~\ref{subsec:milnor} and throughout \S~\ref{section:TS}) using    reduced cross-section.  Recall that a \memph{cross-section} of $\Gamma$ is a group homomorphism $\csn : \Gamma \fun \VF^{\times}$ with $\vv \circ \csn = \id$. The corresponding \memph{reduced cross-section} of $\Gamma$ is the function $\rcsn = \rv \circ \csn : \Gamma \fun \RV$. If such a reduced cross-section exists then it induces an isomorphism $\RV  \cong \Gamma \oplus \K^\times$. In general this is not guaranteed, that is, the short exact sequence in (\ref{LRV:diag}) may not split (definably).

\begin{exam}\label{exam:pui:C}
We may think of the procyclic group $\hat {\mu} = \lim_n \mu_n$ as the Galois group $\gal(\puC / \C \dpar t)$, since they are canonically isomorphic. For each element $\xi = (\xi_n)_n \in \hat \mu$, the assignment $n \efun \xi_n t^{1/n}$ gives a reduced cross-section $\rcsn_\xi : \Q \cong \Gamma \fun \RV$, and the map given by $\xi \efun \rcsn_\xi$ is  a bijection between $\hat \mu$ and the set $\Omega$ of reduced cross-sections $\rcsn$ with $\rcsn(1) = \rv(t)$; in other words, $\hat \mu$ acts freely and transitively on $\Omega$ via multiplication.
\end{exam}

\begin{defn}\label{twistback}
Relative to a reduced cross-section $\rcsn$, the \memph{twistback} function $\tbk : \RV \fun \K$ is given by $u \efun u / \rcsn(\vrv(u))$, where $\infty / \infty = 0$. For any set $U \sub \RV^n_\infty$ and $\gamma \in \Gamma_{\infty}^n$, the set $\tbk(U_{\gamma}) \sub \K^n$ is called the \memph{$\gamma$-twistback} of $U$. If $\tbk(U_{\gamma}) = \tbk(U_{\gamma'})$ for all $\gamma, \gamma' \in \vrv(U)$ then $U$ is called a \memph{twistoid}, in which case we simply write $\tbk(U)$ for the unique twistback. A definable finite partition $(U_i)_i$ of $U$ is called a \memph{twistoid decomposition} of $U$ if every $U_i$ is a twistoid.
\end{defn}

The few facts on reduced cross-sections, twistbacks, and \memph{bipolar} twistoid decompositions (see \cite[Definition~5.10]{HL:modified}) that we shall need are all explained in \cite[\S~5]{HL:modified}.

\section{Motivic Milnor fiber}\label{section:milnor}

In this section we first specialize the integration theory above to henselian subfields of $\bb U$, such as the field $\puR$ of  real Puiseux series. This is followed by a discussion on how to construct various Grothendieck rings, especially those in real algebraic geometry. We then adapt  \cite[\S~8]{HL:modified} to the current context and show that   the real motivic zeta function and thence the motivic Milnor fiber can be recovered from an object of  $\VF_*$, namely the nonarchimedean Milnor fiber. Our method, when combined with the virtual Poincar\'e polynomial of real varieties, yields a new invariant.

\subsection{Specialization to henselian subfields}\label{section:spec:hen}

The descent procedure described here is based on \cite[\S~12]{hrushovski:kazhdan:integration:vf}. We work in $\bb U$ with the parameter space $\bb S = \R \dpar{ t }$.

It may seem at first glance that we may as well take $\bb S = \puR$  since every element in $\puR$ is, after all, definable over $\R \dpar t$. However, generally speaking, elements in $\puR$ are definable over $\R \dpar t$ only in $\puR$, not in $\bb U$, in other words, they are not quantifier-free definable over $\R \dpar t$ in $\puR$ (to define them one needs to use the ordering, which is not quantifier-free definable, and this point will come up again in \S~\ref{sec:tcon}).

Let $\bb M$ be a  substructure of $\bb U$ in which the map $\rv$ is surjective. Recall that the substructure $\bb S = \R \dpar{ t }$ is regarded as a part of the language and hence all other substructures contain it. The main case of interest is $\bb M = \puR$ (see also \cite[\S~8.1]{HL:modified}).

If $X \sub \VF^n  \times \RV^m$ is a definable (and hence quantifier-free definable) set then the \memph{trace} of $X$ in $\bb M$, denoted by $X(\bb M)$, is the set of $\bb M$-rational points of $X$, that is,
\[
X(\bb M) = X \cap (\VF(\bb M)^n  \times \RV(\bb M)^m).
\]
Such a trace is also called a \memph{constructible set} in $\bb M$ since it is indeed quantifier-free definable in $\bb M$. Note that, however, if $f : X \fun \Gamma$ is a definable function then the image $f(X(\bb M))$  is not necessarily a set in $\Gamma(\bb M)$, but rather a set in the divisible hull $\Q \otimes \Gamma(\bb M)$  of $\Gamma(\bb M)$. For instance, if $\bb M = \C \dpar t$ then $\Gamma(\bb M) = \Z$ and hence $\gamma \in \Gamma$ is definable if and only if $\gamma \in \Q \otimes \Gamma(\bb M) = \Q$. On the other hand, if $X$ is a set in $\Gamma$ and $f$ is a piecewise $\mgl_k(\Z)$-transformation on $X$ then $f(X(\bb M))$ is of course a set in $\Gamma(\bb M)$; this is the situation in the $\Gamma$-categories.

Assume  that the valued field $(\VF(\bb M), \OO(\bb M))$  is henselian. This is equivalent to the condition that $\bb M$ is definably closed and $\Gamma(\bb M)$ is nontrivial (see \cite[Example~12.8]{hrushovski:kazhdan:integration:vf} or  \cite[Lemma~2.11]{HL:modified}). It follows that $\bb M$ is \memph{functionally closed}, that is,  for any definable set $X$ and any definable function $f$ on $X$ (no $\Gamma$-coordinates are allowed), the image $f(X(\bb M))$ is a set in $\bb M$ and hence  is definable (constructible)  in $\bb M$, or more concisely, $f(X(\bb M)) = f(X)(\bb M)$.

The rest of this subsection is devoted to explaining the middle portion of (\ref{sum:narr}):
\begin{equation}\label{hen:sum}
\bfig
  \hSquares(0,0)/->`->`->`->`->`->`->/<400>[{\ggk \VF_*}`{\ggk \RV[*]/(\bm P - 1)}`{\sggk \RES}`{\ggk \VF_{\bb M}}`{\ggk \RV_{\bb M}[*]/(\bm P - 1)}`{\sggk \RES_{\bb M}}; \int`\bb E_b`-/ {\gF_{\bb M}}`-/{\gR_{\bb M}}`-/{\gR_{\bb M}}`\int_{\bb M}`\bb E_{b, \bb M}]
\efig
\end{equation}
where $\bb E_{b}$, $\bb E_{b, \bb M}$ can be replaced by  $\bb E_{g}$, $\bb E_{g, \bb M}$ if $[\A]$ is inverted in the right column.

\begin{defn}[$\bb M$-constructible categories]\label{def:hen:cat}
An object of the category $\RV_{\bb M}[k]$ is a pair of the form $\bm U(\bb M) = (U(\bb M),  f \rest U(\bb M))$, where $\bm U = (U, f) \in \RV[k]$. Any constructible function of the form $F(\bb M) : \bm U(\bb M) \fun \bm V(\bb M)$, where $F : \bm U \fun \bm V$ is a $\RV[k]$-morphism, is a morphism of $\RV_{\bb M}[k]$.

The categories $\VF_{\bb M} = \bigcup_k \VF_{\bb M}[k]$, $\Gamma_{\bb M}[k]$, $\RES_{\bb M}[k]$, etc., are formulated analogously.
\end{defn}

We call $\gsk \VF_{\bb M}$,  etc.,  \memph{$\bb M$-constructible Grothendieck semirings} associated with $\bb M$.

Since $\bb M$ is functionally closed, the following binary relation is well-defined and is indeed a semiring congruence relation:
\[
\gF_{\bb M} = \set{([A], [B]) \in (\gsk \VF_*)^2 \given [A(\bb M)] = [B(\bb M)] \text{ in } \gsk \VF_{\bb M} }.
\]
The semiring congruence relations $\gR_{\bb M} \sub (\gsk \RV[*])^2$,  $\gG_{\bb M} \sub (\gsk \Gamma[*])^2$ are defined analogously. The restriction of $\gR_{\bb M}$ to $(\gsk \RES[*])^2$ and the corresponding ideal of $\sggk \RES$ are both still denoted by $\gR_{\bb M}$. We have $\gsk \VF_{\bb M} \cong {\gsk \VF_*} / \gF_{\bb M}$, etc.

Suppose that $([A], [B]) \in \gF_{\bb M}$. By the definition of  $\VF_{\bb M}$, we can find a  $\VF_*$-morphism  $F : A'\fun B'$ with $A(\bb M) \sub A' \sub A$ and $B(\bb M) \sub B' \sub B$ that witnesses this (there may not exist a $\VF_*$-morphism between $A$ and $B$, though). Let $A'' = A \mi A'$ and $B'' = B \mi B'$. By Theorem~\ref{main:prop}, there are $U, V \in \RV[*]$ such that $[A''] = [{\bb L}{\bm U}]$ and $[B''] = [{\bb L}{\bm V}]$. By functional closedness, if $\bm U(\bb M) \neq \0$ then ${\bb L}\bm U (\bb M) \neq \0$ and hence $A''(\bb M) \neq \0$, which contradicts the choice of $A''$. So $\bm U(\bb M) = \0$ and similarly $\bm V(\bb M) = \0$. This means that $([\bm U], [\bm V]) \in \gR_{\bb M}$, in other words, $\int_+ [A''] = \int_+ [B'']$ modulo $\gR_{\bb M}$. Therefore,
\[
\int_+ [A] = \int_+ [A'] + \int_+ [A''] =_{\gR_{\bb M}} \int_+ [B'] + \int_+ [B''] = \int_+ [B].
\]
Conversely, by a similar reasoning, if $([\bm U], [\bm V]) \in \gR_{\bb M}$ then ${\bb L}{\bm U}(\bb M)$, ${\bb L}{\bm V}(\bb M)$ are isomorphic in $\VF_{\bb M}$, in other words, $([{\bb L}{\bm U}], [{\bb L}{\bm V}]) \in \gF_{\bb M}$. So $\int$ induces an isomorphism $\int_{\bb M}$ between $\ggk \VF_{\bb M}$ and $\ggk \RV_{\bb M}[*]/(\bm P - 1)$.

Next, let $\gR_{\bb M} \otimes \gG_{\bb M}$ be the semiring congruence relation on $\gsk \RES[*] \otimes \gsk \Gamma[*]$ generated by $\gR_{\bb M}$ and $\gG_{\bb M}$. By the universal mapping property of tensor product, there exists a canonical isomorphism
\[
\gsk \RES[*] \otimes \gsk \Gamma[*] / \gR_{\bb M} \otimes \gG_{\bb M} \cong \gsk \RES_{\bb M}[*] \otimes_{\gsk \Gamma^{\fin}_{\bb M}[*]} \gsk \Gamma_{\bb M}[*].
\]
So the assignment (\ref{assgn:Psi})  induces a $\gsk \Gamma^{\fin}_{\bb M}[*]$-linear map
\begin{equation*}
  \Psi_{\bb M}: \gsk \RES[*] \otimes \gsk \Gamma[*] / \gR_{\bb M} \otimes \gG_{\bb M} \fun \gsk \RV[*] / \gR_{\bb M} \cong \ggk \RV_{\bb M}[*].
\end{equation*}
Functional closedness and Proposition~\ref{red:D:iso} show that the two semiring congruence relations match exactly via $ \Psi_{\bb M}$ and hence $\Psi_{\bb M}$ is an isomorphism as well.

Finally,  suppose that $\Gamma(\bb M)$ is divisible,  for instance, $\bb M=\puR$. By \omin-minimal cell decomposition and induction on dimension,  for any $I, J \in \Gamma [*]$, if $I(\bb M) = J(\bb M)$ then $\chi_b(I) = \chi_b(J)$ and $\chi_g(I) = \chi_g(J)$, and hence this is so if, more generally, $([I], [J]) \in \gG_{\bb M}$. So $\ggk \Gamma_{\bb M}[*]$ also admits two ring homomorphisms into $\Z$ and the assignment (\ref{eb:eg:defn}) yields two ring homomorphisms $\bb E_{g, \bb M}$, $\bb E_{b, \bb M}$ from $\ggk \RV_{\bb M}[*]/(\bm P - 1)$ into $\sggk \RES_{\bb M}$.


\subsection{Grothendieck rings in real and complex geometry}\label{sect:Groth:real}

The complexification of a  variety $X$ over $\R$ is denoted by $X \otimes_{\R}\C$, which is a variety over $\C$ endowed with an antiholomorphic involution coming from the complex conjugation $\bm c$ over $\C$; the Grothendieck ring of the corresponding category is denoted by $\ggk^{\bm c} \Var_{\C}$. Conversely, to every quasi-projective variety $Y$ over $\C$ endowed with an antiholomorphic involution there corresponds a unique variety $X$ over $\R$ such that $Y \cong X \otimes_{\R}\C$. So extension of scalars induces an isomorphism $\ggk\Var_{\R} \to \ggk^{\bm c}\Var_{\C}$.

Taking the fixed points of the set $X(\C)$ of the complex points of a variety $X$ over $\R$ under the complex conjugation gives a real
variety in the sense of \cite{BCR}; this is denoted by $X(\R)$. Such sets of real points of varieties over $\R$, considered with their sheaves of regular functions over $\R$, form the category $\RVar$ of real  varieties, and taking  real points induces a surjective homomorphism $\ggk\Var_{\R} \to \ggk{\RVar}$ with kernel generated by the varieties without real points (see \cite[Corollary~1.11]{GF:MA}).

We consider also an equivariant version of the Grothendieck ring of complexified varieties over $\R$, taking into account  group actions by roots of unity that are compatible with the complex conjugation.

\begin{nota}
Denote the dihedral group $\gal(\C / \R) \ltimes \mu_n$ by $\delta_n$, where the $\gal(\C / \R)$-action  on $\mu_n$ corresponds to taking the inverse. Set $\hat {\delta} = \lim_n \delta_n$, which is canonically isomorphic to $\gal(\puC / \puR) \ltimes \hat \mu$, where the action of $\gal(\puC / \puR)$ on $\hat \mu$ corresponds again to taking the inverse. It may also be identified with the Galois group $\gal(\puC / \R \dpar t)$. More explicitly, if $\sigma = (\sigma_n)_{n\geq 1}$ is a coherent system of roots of unity, then it acts on a complex Puiseux series $\phi=\sum_k a_k t^{k/m}$ by
$\sigma \cdot \phi=\sum_k a_k \sigma_m^k t^{k/m}$. The conjugation automorphism of $\puC$ is also denoted by $\bm c$. It acts on $\phi$ by $\bm c \cdot \phi= \sum_k \overline a_k t^{k/m}$.
\end{nota}

\begin{lem}\label{chch} For any element $\sigma$ of $\hat \mu$, the relation $\bm c \sigma \bm c \sigma = 1$ holds in $\hat \delta$.
\end{lem}

\begin{proof} Denote by $\xi=(\xi_n)_{n\geq 1}$ the topological generator of $\hat \mu$ given by $\xi_n=\exp ({2i\pi /n})$. So the orbit of $\xi$ is dense in the topological group $\hat \mu$. The equality $\overline{\xi_n \overline{\xi_n z}} =z$ holds for any $n\geq 1$ and any  $z \in \C$, which implies that $\bm c \xi \bm c \xi = 1$ in $\hat \delta$. Moreover, for any integer $m \geq 0$,
\[
\bm c \xi^{m+1} \bm c  \xi^{m+1} =\bm c \xi^{m} (\xi \bm c \xi) \xi^m =\bm c \xi^{m} \bm c (\bm c \xi \bm c  \xi )\xi^{m} = \bm c \xi^{m} \bm c \xi^{m},
\]
and hence, by induction, the relation $\bm c \xi^{m} \bm c \xi^{m} = 1$ holds for any $m$. So the relation $\bm c \sigma \bm c \sigma = 1$ holds for any element of $\hat \mu$ because $\xi$ is a topological generator of $\hat \mu$.
\end{proof}

\begin{defn}\label{ct:var:over:R}
A $\hat \delta$-action $\hat h$ on a complexified variety $X$ over $\R$ is \memph{good} if it factors through some $\delta_n$-action and the induced $\gal(\C / \R)$-action is the canonical antiholomorphic involution.

The category of complexified varieties over $\R$ with good $\hat \delta$-actions consists of objects of the form $\bm X = (X, \hat h)$, where $X$ is a complexified quasi-projective variety over $\R$ and $\hat h$ is a good $\hat \delta$-action on $X$, and $\hat \delta$-equivariant morphisms between such objects.

The Grothendieck ring of this category is denoted by  $\ggk^{\flat, \hat \delta}\Var_{\R}$. The ring $\gdv$ is the quotient of $\ggk^{\flat, \hat \delta}\Var_{\R}$ by the ideal generated by the elements of the form
\begin{equation}\label{compl:flat}
[\bm X \times ( \A^n_\C, \hat h)] - [\bm X \times (\A^n_\C, \bm c)],
\end{equation}
where $\hat h$ is a good linear $\hat \delta$-action.
\end{defn}

\begin{rem}\label{rem-inv}
Let $\bm X = (X, \hat h)$ be a complexified variety over $\R$ with a good $\hat \delta$-action. Then $\hat h$ and the group involution of $\hat \delta$ together induce another good $\hat \delta$-action $\hat h'$ on $X$. This means that there is a natural ring involution of  $\ggk^{\flat, \hat \delta}\Var_{\R}$ and also $\gdv$.
\end{rem}

An arc $\spec \C \dbra t \fun X$ on a variety $X$ over $\C$ may have branches, which are represented by complex Puiseux series in $\puC$. Galois actions over $\C \dbra t$ on these branches encode certain information on the singularity in question and hence are an integral part of the construction in \cite{hru:loe:lef}. These Galois actions are gone when we restrict to real branches of real arcs, corresponding to the  pair $\puR$ and $\R \dbra t$, albeit a faint trace remains.

\begin{rem}\label{shat:section}
We have seen in Example~\ref{exam:pui:C} above that there is a natural bijection between $\hat \mu$ and the set of reduced cross-sections $\rcsn : \Q \fun \RV$ with $\rcsn(1) = \rv(t)$ in $\puC$. Similarly, there is such a bijection between $\hat \delta$ and such a set but with $\rcsn(1) = \rv(\pm it)$.

In contrast, there is only one such reduced cross-section in $\puR$, which is but another way of saying the fact that $\gal(\puR / \R \dpar t)$ is trivial. Nevertheless, if $n$ is even then $\gal(\R \dpar{t^{1/n}} / \R \dpar t) \cong \mu_2$, and there are two such reduced cross-sections in $\R \dpar{t^{1/n}}$, determined by the two choices $\pm t^{1/n}$, and if $n$ is odd then there is only one.
\end{rem}

\begin{defn}\label{Rvar:mu2}
The category of real varieties with $\mu_2$-actions consists of objects of the form $\bm X = (X, h)$, where $X$ is a real variety and $h$ is a $\mu_2$-action on $X$, and $\mu_2$-equivariant morphisms between such objects (the actions and morphisms are all given by regular maps in the sense of \cite{BCR}).

The Grothendieck ring of this category is denoted by $\ggk^{\flat,\mu_2}\RVar$. The ring $\gsv$ is the quotient of $\ggk^{\flat,\mu_2}\RVar$ by the ideal generated by the elements of the form
\begin{equation}\label{real:flat}
 [\bm X \times ( \A^n_\R, h)] - [\bm X \times (\A^n_\R, \id)],
\end{equation}
where $h$ is any linear $\mu_2$-action.
\end{defn}

Let $(X,\hat h)$  be a complexified variety over $\R$ with a good $\hat \delta$-action, factoring through a $\delta_n$-action.

\begin{rem}
Since $\mu_n$ is cyclic, the induced $\hat \mu$-action on the $\hat h$-orbit  of any (closed) point $x \in X$  factors through a faithful $\mu_{d_x}$-action with $d_x | n$. Thus, $\hat h$ factors through a $\delta_{d_x}$-action $h_{d_x}$.

For all $\sigma \in \mu_{d_x} < \delta_{d_x}$, since $\bm c\sigma\bm c \sigma = 1$, we have $\bm c h_{d_x}(\sigma) \bm c h_{d_x}(\sigma)(x) = x$. Moreover, if $x$ is a real closed point, that is, if $x \in X(\R)$, then $h_{d_x}(\sigma) \bm c h_{d_x}(\sigma)(x) = x$.
In particular, if $d_x$ happens to be even and if we choose $\sigma \in \mu_2 < \delta_{d_x}$ (so that $h_{d_x}(\sigma)^2=1$), then
$\bm c h_{d_x}(\sigma)(x) = h_{d_x}(\sigma)(x)$, meaning that $h_{d_x}(\sigma)(x)$ is a real closed point as well.
\end{rem}

\begin{defn}\label{def:mu2:var}
The $\mu_2$-action $\hat h(\R)$ on $X(\R)$ is given by $x \efun h_{d_x}(\sigma)(x)$ for $\sigma \in \mu_2$ if $d_x$  is even and $x \efun x$ otherwise.
\end{defn}

See  Remark~\ref{mu2:failure} for an illuminating example.

\begin{lem}\label{mu2:var}
The assignment $(X,\hat h)\efun (X(\R),\hat h(\R))$ induces a surjective group homomorphism
\begin{equation}\label{down:mu2}
 \Xi : \ggk^{\hat \delta}\Var_{\R} \fun \gsv.
\end{equation}
\end{lem}
\begin{proof}
If the induced $\mu_n$-action is faithful on $X$, then $(X(\R),\hat h(\R))$ is a real variety with a $\mu_2$-action in the sense of Definition \ref{Rvar:mu2}. Otherwise, we stratify $X(\R)$ into Zariski constructible subsets on which the $\mu_2$-action is regular as follows. For any positive integer $d$ dividing $n$, the set $X_d$ of fixed points of the $\mu_d$-action on $X$ induced by the morphism $\mu_n \fun \mu_d : \sigma \efun \sigma^{n/d}$ is a complexified variety over $\R$ with a $\delta_{n/d}$-action. Since the $\mu_{n/d}$-action is faithful on  $X_d' = X_d \mi \bigcup_{d | l| n, d < l} X_l$, the $\mu_2$-action induced on the Zariski constructible subset $X_d'(\R)$ of $X(\R)$ is regular as expected.

So  summing up the classes of these constructible sets with $\mu_2$-actions makes sense in $\gsv$. Since isomorphic complexified varieties $\bm X$, $\bm Y$ over $\R$ with $\hat \delta$-actions give rise to isomorphic constructible sets with $\mu_2$-actions, and (\ref{compl:flat}), (\ref{real:flat}) are essentially the same condition, it follows that $\Xi$ is well-defined, and the assignment $[\bm X] \efun [\bm X(\R)]$ does respect addition by construction. For surjectivity, only note that a $\mu_2$-action that is regular on a real variety can be linearized and hence gives rise to a complexified variety over $\R$ with a $\delta_2$-action in a natural way.
\end{proof}

Note, however, that $\Xi$ fails to respect product and hence is not a ring homomorphism: if $x$, $y$ are two points belonging to  $\bm X$, $\bm Y$ then $d_{(x,y)} = \gcd(d_x, d_y)$ and hence the $\mu_2$-action on $(x, y)$ as a point belonging to $\bm X \times \bm Y$ is not necessarily the product of the   $\mu_2$-actions on $x$, $y$. On the other hand, in light of (\ref{compl:flat}) and (\ref{real:flat}), it can be upgraded to an $\mdl A_{\C}$-module homomorphism via the natural ring homomorphism $\mdl A_{\C} \fun \mdl A_{\R}$, where $\mdl A_{\C}$ is the subring of $\gdv$ generated by $[(\A_{\C}, \bm c)]$ and $\mdl A_{\R}$ is the subring of $\gsv$ generated by $[\A_{\R}]$.

A similar construction at the level of $\ggk^{\flat ,\hat \delta}\Var_{\R}$ instead of $\ggk^{\hat \delta}\Var_{\R}$ yields a group homomorphism $\Xi^\flat : \ggk^{\flat ,\hat \delta}\Var_{\R} \fun \ggk^{\flat,\mu_2}{\RVar}$.

\begin{figure}[htb]
\begin{equation}\label{tauhat:forget}
\bfig
 \iiixiii(0,0)/..>`->`<..`<-`<-`->`->`->`->`->`->`->/<800, 400>[{\ggk^{\flat, \mu_2}{\RVar}}`{\gsv}`{\ggk{\RVar}}`{\ggk^{\flat,\hat \delta}\Var_{\R}}`{\ggk^{\hat \delta}\Var_{\R}}`{\ggk\Var_{\R}}`{\ggk^{\flat, \hat \mu}\Var_{\C}}`{\ggk^{\hat \mu}\Var_{\C}}`{\ggk\Var_{\C}}; ``\Xi^\flat`\Xi`\Xi_{\R}``\Phi`````]
\efig
\end{equation}
\end{figure}
In \cite{hru:loe:lef}, similar Grothendieck rings $\ggk^{\flat, \hat \mu}\Var_{\C}$, $\ggk^{\hat \mu}\Var_{\C}$ are defined for categories of varieties over $\C$. We summarize the situation in the diagram (\ref{tauhat:forget}), where the right column of horizontal arrows are the forgetful homomorphisms,  the first row of vertical arrows are obtained by taking real points, and the second row of vertical arrows are obtained by  forgetting the  antiholomorphic (that is, the  real) structure. This diagram does commute except for the upper left square:  the two routes from $\ggk^{\flat,\hat \delta}\Var_{\R}$ to $\gsv$ are actually not identical  (see Remark~\ref{mu2:failure}), and the construction below uses the one that passes through $\Xi$.

\begin{rem}[Polynomial realizations]\label{rem:beta}
There is a unique ring homomorphism $\beta: \ggk \RVar \fun \Z[u]$ that coincides with the Poincar\'e polynomial
$\sum_{i \in \N} \dim H_i(X,\F_2)u^i$ for compact nonsingular real varieties $X$; see \cite{mccrory:paru:virtual:poin}. Similarly, there is a unique group homomorphism
\[
\beta^{\mu_2}: \gsv \fun \Z \dbra{u^{-1}}[u]
\]
that coincides with the equivariant Poincar\'e series
$\sum_{i\in \Z} \dim H_i(\bm X, \F_2)u^i$ for compact nonsingular real varieties $\bm X$ endowed with $\mu_2$-actions; see \cite{Fseries}. We will say more about the virtual Poincar\'e polynomial in \S~\ref{vPp}.
\end{rem}

\subsection{Piecewise retraction to $\RES$}

We now proceed to replicate the construction in \cite[\S~8]{hru:loe:lef} so to recover the motivic zeta function with coefficients in $\gdv[[\A]^{-1}]$ and the corresponding motivic Milnor fiber; many arguments are formally the same, hence we shall be brief. The subsequent specialization to $\gsv[[\A]^{-1}]$  is new and points to deeper phenomena in the real algebraic environment.

For the remainder of this section we  work in $\puC$ with $\bb S = \R \dpar t$.

As in  \cite[\S~4.3]{hru:loe:lef}, using the  twistback function given by a chosen reduced cross-section $\rcsn$, we construct a homomorphism
\begin{equation}\label{hom:theta}
\Theta : \sggk \RES \fun \gdv.
\end{equation}
This does not depend on the choice of $\rcsn$.

\begin{rem}\label{first:theta}
Several variants of $\Theta$ will appear below. To show that they are  injective, we may simply follow the argument in the proof of \cite[Proposition~4.3.1]{hru:loe:lef}. For surjectivity, however, some modification is needed, and how much of it is needed varies.

For $\Theta$ it is quite simple. Let $[(X,\hat h)] \in \ggk^{\hat \delta}\Var_{\R}$ with $\hat h$ factoring through a $\delta_n$-action. We may assume that $X$ is quasi-projective and irreducible. The induced $\mu_n$-action on $X$ gives a quotient variety $X / \mu_n$, which is also quasi-projective and carries an antiholomorphic involution and hence is defined over $\R$. Then the Kummer-theoretic construction in the proof of \cite[Proposition~4.3.1]{hru:loe:lef} yields a $U \in \RES$ with $\Theta([U]) = [(X, \hat h)]$.

The situation in $\puR$ is  trickier (even for injectivity, due to how the category $\RES_{\puR}$ is formulated in Definition~\ref{def:hen:cat}). Let $U \in \RES_{\puR}$. For each $u \in U$, let $d_u$ be the least positive integer such that $u$ is a tuple in $\RV(\R \dpar{t^{1/d_u}})$, or equivalently, $\vrv(u) \in d_u^{-1}\Z$. As implied by Remark~\ref{shat:section},  there is a nontrivial $\mu_2$-action on a two-element set $\{u, u'\} \sub U$ if $d_u$ is even. Thus, similar to Definition~\ref{def:mu2:var}, we can construct a $\mu_2$-action on $U$ by $u \efun u'$  if $d_u$ is even and $u \efun u$ otherwise. If we choose the reduced cross-section  that is  in $\puR$ then the twistback function yields a homomorphism
\begin{equation}\label{theta:puR}
  \Theta_{\puR} : \sggk \RES_{\puR} \fun \gsv
\end{equation}
such that the diagram
\begin{equation}\label{real:RV:com}
\bfig
\Square(0,0)/->`->`->`->/<400>[{\sggk \RES}`{\gdv}`{\sggk \RES_{\puR}}`{\gsv}; \Theta`\Xi_{\puR}`\Xi`{\Theta_{\puR}}]
\efig
\end{equation}
commutes, where $\Xi_{\puR}$ is just the map $-/{\gR_{\puR}}$ with $\bb M = \puR$ in (\ref{hen:sum}), but has to be treated as an $\mdl A_{\C}$-module homomorphism instead of a ring homomorphism, for the same reason that  $\Xi$ has to be treated as such. So $\Theta_{\puR}$ must be surjective. Comparing the constructions of $\Xi_{\puR}$ and $\Xi$, we see that $\Theta_{\puR}$ must be injective as well.
\end{rem}

\begin{rem}\label{note:vol}
Appending (\ref{real:RV:com}) to (\ref{hen:sum}) with $\bb M = \puR$, we derive  two volume operators
\[
\Vol^{\mu_2}: \ggk \VF_*\to^{\Vol^{\hat \delta} = \Theta \circ \bb E_{b} \circ \int} \gdv \to^{\Xi} \gsv
\]
and
\[
\Vol_{\puR}^{\mu_2} : \ggk \VF_{\puR} \to^{\Theta_{\puR} \circ \bb E_{b, \puR} \circ \int_{\puR}} \gsv,
\]
satisfying $\Vol^{\mu_2} = \Vol^{\mu_2}_{\puR} \circ (- /\gF_{\puR})$. Also, according to \cite[Notation~5.14]{HL:modified},  for any $\bm U \in \RV[*]$ and any bipolar twistoid decomposition $(U_i)_i$ of $\bm U$, we have
\begin{equation}\label{twist:C}
  (\Theta \circ \bb E_b)([\bm U]) = \sum_i \chi_b(\vrv(U_i))[\tbk(U_{i})]^{\hat \delta} \in \gdv,
\end{equation}
where $[\tbk(U_i)]^{\hat \delta} \in \gdv$ is just $\Theta([U_i \cap \gamma^\sharp])$ for any $\gamma \in I_i = \vrv(U_i)$, but is a more useful notation since it is computationally more suggestive.

We can also consider changing $\bb E_b$ to $\bb E_g$ (see Remark~\ref{g:comm:alm}).
\end{rem}

For any ring $R$, let $R[T^{\Q}]$ denote the ring of Puiseux polynomials over $R$, that is, the group ring of $\Q$ over $R$.
Consider the subring $\gdv[[\A]^{-1}][T, T^{-1}]$ of $\gdv[[\A]^{-1}][T^{\Q}]$. The canonical image of $\gdv[T^{\Q}]$ in  $\gdv[[\A]^{-1}][T^{\Q}]$ is still denoted as such. The assignment $T \efun [\A]$ determines a ring homomorphism
\begin{equation}\label{pui:laurent}
 \bm \eta: \gdv[[\A]^{-1}][T, T^{-1}] \fun \gdv[[\A]^{-1}].
\end{equation}

Let $\bm U = (U, f, \omega) \in \mgRV^{\db}[k]$. Following the discussion leading to \cite[(8.2)]{HL:modified}, using the notation therein in particular, we assign to $\bm U$ the expression
\begin{equation}\label{hm:ass}
  h_m(\bm U) = \sum_i [\tbk(U_i)]^{\hat \delta} \sum_{\gamma \in I_{i,m}} T^{-m \sigma_i(\gamma)} \in \gdv[[\A]^{-1}][T^{\Q}].
\end{equation}

The ``piecewise retraction'' formula (\ref{hm:ass}) is indeed quite similar to the ``uniform retraction'' formula (\ref{twist:C}). The difference is just that the coefficient $\chi_b(I_i)$ in the latter is replaced by a formal ``counting'' sum $\sum_{\gamma \in I_{i,m}} T^{-m \sigma_i(\gamma)}$ in the former, which also takes the volume form into account. That the ``uniform retraction'' formula does not depend on the choice of the bipolar twistoid decomposition is simply a consequence of the construction of $\bb E_b$ itself. That  (\ref{hm:ass}) does not depend on the choice of the bipolar twistoid decomposition and is invariant on isomorphism classes needs an argument, see the proof of \cite[Lemma~8.2]{HL:modified} (it is here  that we need functional closedness in $\C \dpar{ t^{1/m} }$). Therefore, we have constructed a ring homomorphism
\[
\bm h_m : \ggk \mgRV^{\db}[*] \fun  \gdv[[\A]^{-1}][T^{\Q}].
\]

Let $\ggk^\natural_m \mgRV^{\db}[*]$ denote the subring $(\bm h_m)^{-1}(\gdv[[\A]^{-1}][T, T^{-1}])$ of $\ggk \mgRV^{\db}[*]$. By the proof \cite[Lemma~8.3]{HL:modified}, the homomorphism
\[
\bm \eta \circ \bm h_m : \ggk^\natural_m \mgRV^{\db}[*] \fun \gdv[[\A]^{-1}]
\]
vanishes on $(\bm P_{\Gamma})$. Note that the ideal $(\bm P_\Gamma)$ of $\ggk \mgRV^{\db}[*]$ in Notation~\ref{pgamma} is now generated by the elements $\bm P_\gamma$ with $\gamma \in \Z^+ = \Gamma^+(\R \dpar t)$.

\begin{rem}\label{dia:to:db:dag}
If $\bm U = (U, f, l) \in \mgRV^{\db}[*]$ with $l \in m^{-1}\Z$ (constant volume form) then the exponents in (\ref{hm:ass}) are all integers and hence $[\bm U] \in \ggk^\natural_m \mgRV^{\db}[*]$. Actually we shall only need the case $l = 0$.

The ring $\bigcap_{m \in \Z^+} \ggk^\natural_m \mgRV^{\db}[*]$ is denoted by $\ggk^\natural \mgRV^{\db}[*]$.

If $\bm A = (A, l) \in \mgVF^\diamond[*]$ with $l$  constant then $\int^\diamond [\bm A]$ may be expressed as $[(U, f, l)]/ (\bm P_{\Gamma})$, and hence if $l \in \Z$ then $\int^\diamond [\bm A]$ belongs to the quotient $\ggk^\natural \mgRV^{\db}[*]/ (\bm P_{\Gamma})$. In that case,  for every $m \in \Z^+$, the expression $(\bm \eta \circ \bm h_m \circ \int^\diamond)([\bm A]) \in \gdv[[\A]^{-1}]$ makes sense  since $\bm \eta \circ \bm h_m$ vanishes on $(\bm P_{\Gamma})$.
\end{rem}

Denote by $\RES_{m}$ the full subcategory of $\RES$ such that $U \in \RES_{m}$ if and only if every $\gamma \in \vrv(U)$ is a tuple in $m^{-1} \Z$, or equivalently, $U \in \RES_{m}$ if and only if the action on $U$ of the kernel of the canonical projection $\hat \delta \fun \delta_m$ is trivial.

Let $\beta = (\beta_1, \ldots, \beta_n) \in (m^{-1} \Z)^n$ and $A \sub \OO^n \times \RV^l$ be a proper $\beta$-invariant definable set such that $\pvf \rest A$ is finite-to-one. Then we may associate with $A$ an element $[A[m;\beta]][\A]^{- m\sum \beta }$ in $\ggk \RES_{m}[[\A]^{-1}]$, see \cite[\S~4.2]{hru:loe:lef} for detail. By \cite[Lemma~4.2.1]{hru:loe:lef}, it actually does not depend on $\beta$ by \cite[Lemma~4.2.1]{hru:loe:lef} and hence may be denoted by $\tilde A[m]$.

The localization of $\Theta$ at $[\A]$ is still denoted by $\Theta$. The image of $\tilde A[m]$ in $\sggk \RES[[\A]^{-1}]$ is still denoted as such. The following result is crucial for recovering the motivic zeta function.

\begin{lem}\label{inv:dir:com}
$(\bm \eta \circ \bm h_m \circ \int^\diamond)([(A, 0)]) = \Theta (\tilde A[m])$ in $\gdv[[\A]^{-1}]$.
\end{lem}

To show this lemma, the statement of Theorem~\ref{main:prop:dia} itself is not quite enough. We need the fact that there exists a special $\mVF^\diamond[*]$-morphism $F : (A, 0) \fun (\bb L \bm U, 0)$, called a \memph{special covariant bijection}, with $\bm U \in \RV^{\db}[*]$; see the proof of \cite[Lemma~8.6]{HL:modified} for detail.

Recall the diagrams (\ref{tauhat:forget}) and (\ref{real:RV:com}). Applying the ring homomorphism  $\Xi_{\R} \circ \Phi$, localized at $[\A]$, to both sides of the equality in Lemma~\ref{inv:dir:com}, we get an equality in $\ggk{\RVar}[[\A]^{-1}]$; similarly in $\gsv[[\A]^{-1}]$ if the $\mdl A_{\C}$-module homomorphism $\Xi$, localized at $[(\A_{\C}, \bm c)]$, is used.

\subsection{Motivic zeta function and motivic Milnor fiber}\label{subsec:milnor}

Let $X$ be a nonsingular connected variety of dimension $d$ over $\R$ and $f$ a nonconstant morphism, also over $\R$, from $X$ to the affine line. Let $z \in f^{-1}(0)$ be an $\R$-rational point.  Since $X$, $f$, and $z$ are fixed, we shall not always carry them in notation and terminology.
Let $\pi$ be the reduction map $X(\OO) \fun X(\K)$. The \memph{nonarchimedean Milnor fiber} of $f$ is the set
\begin{equation}\label{com:nonarch}
\mdl X = \set{ x \in X(\OO) \given \rv(f(x)) = \rv(t) \text{ and } \pi(x) = z }.
\end{equation}
Note that $\mdl X$ may be constructed in an affine neighborhood of $z$ and hence is indeed a (quantifier-free) definable set. Moreover, it is $\beta$-invariant for every $\beta \geq 1$. Thus, $\mdl X$ is an object of $\mgVF^\diamond[*]$ equipped with the constant volume form $0$. We define the  \memph{positive motivic zeta function} of $f$ to be the power series
\begin{equation}\label{real:zeta}
Z^{1}(T) = \sum_{m \in \Z^+} \Big(\Xi \circ \bm \eta \circ \bm h_m \circ \int^{\diamond} \Big)([\mdl X]) T^m \in \gsv[[\A]^{-1}] \dbra T.
\end{equation}

Recall from \S~\ref{sec:intro}, for each $m \in \Z^+$, the set of truncated arcs at $z$:
\[
\mdl X_{m} = \set{ \varphi \in X(\C[t] / t^{m+1} ) \given f(\varphi) = t^m \mod t^{m+1} \tand \varphi(0) = z  },
\]
which may be considered as a complexified variety over $\R$ with a natural $\delta_m$-action. Taking real points gives rise to a real variety with a $\mu_2$-action:
\[
\mdl X_{m}(\R) = \set{ \varphi \in X(\R[t] / t^{m+1} ) \given f(\varphi) = t^m \mod t^{m+1} \tand \varphi(0) = z  };
\]
we denote it by $\mdl X_{m}^1$ to emphasize the sign of $t^m$. Note that $\mdl X_{m}^{1} \otimes_{\R}\C=\mdl X_m$.

\begin{prop}
The zeta function in (\ref{real:zeta}) coincides with the classical one as in (\ref{old:zeta}).
\end{prop}
\begin{proof}
By Remark~\ref{dia:to:db:dag}, the expression $(\bm \eta \circ \bm h_m \circ \int^{\diamond} )([\mdl X]) \in \gdv[[\A]^{-1}]$ makes sense, and by Lemma~\ref{inv:dir:com}, we have $(\bm \eta \circ \bm h_m \circ \int^{\diamond} )([\mdl X]) = \Theta(\tilde{\mdl X}[m])$. Notice that
\[
\mdl X_{m} \cong \set{ \varphi \in X(\C[t^{1/m}] / t^{(m+1)/m} ) \given \rv(f(\varphi)) = \rv(t) \tand \varphi(0) = z }.
\]
So  $\Theta (\tilde{\mdl X}[m]) = [\mdl X_{m}][\A]^{-md}$. So the coefficients of $Z^{1}(T)$ can be recast as
\begin{equation}\label{z:coeff:change}
\Big(\Xi \circ \bm \eta \circ \bm h_m \circ \int^{\diamond} \Big)([\mdl X]) = \Xi ([\mdl X_m][\A]^{-md}) = [\mdl X_m^1][\A]^{-md}. \qedhere
\end{equation}
\end{proof}

\begin{rem}\label{mu2:failure}
The $\mu_2$-action on $\mdl X_{m}^{1}$ considered here is in general different from the one  in \cite{Fseries}, where it is simply induced by $t \efun -t$ for $m$ even. For instance, suppose that $X$ is the affine line, $f$ is the square function, and $z = 0$, then the $\mu_2$-action given by $\Xi$ on
\[
\mdl X_4^1=\set{ \pm t^2 + bt^3 +ct^4 \in \R[t] / t^5 \given  (b,c) \in \R^2 } \cong \{x^2=1\} \times \R^2
\]
swaps any two elements of the form $\pm t^2 +ct^4$ and hence induces the unique nontrivial $\mu_2$-action on the first factor of $\{x^2=1\}\times 0 \times \R$,  whereas the action induced by $t \efun -t$ is entirely trivial. Actually, the $\mu_2$-action induced by $t \efun -t$ corresponds to the dotted route from $\ggk^{\flat ,\hat \delta} \Var_{\R}$ to $\gsv$ in (\ref{tauhat:forget}).

The motivic zeta function $Z^{G}(T)$ studied in \cite{Fseries} is shown to be a rational series (the rational formula given therein needs to be revised, though) and hence  one can take the limit as $T$ goes to infinity, as we shall do to $Z^{1}(T)$ too below. Unfortunately, this process of ``taking the limit'' kills off the $\mu_2$-actions on the coefficients of $Z^{G}(T)$ and, consequently, the limit of $Z^{G}(T)$ does not actually carry any $\mu_2$-action. In contrast, the limit of $Z^{1}(T)$ often retains a $\mu_2$-action and lives in $\gsv[[\A]^{-1}]$.
\end{rem}

There is the negative counterpart $Z^{-1}(T)$ of $Z^{1}(T)$. Since the situation is the same, we shall concentrate on the positive one and drop the qualifier ``positive'' from the terminology. We do remark that, although complexification has seen success to some extent, for instance, the result on the Euler characteristics in \cite{McParu:mono} or the fact that the involution defined in Remark \ref{rem-inv} exchanges $Z^{1}(T)$ and $Z^{-1}(T)$ (as observed in \cite[Lemma 3.2]{F-Rfib} for truncated arcs), it is unclear how the duality of ``the positive'' and ``the negative'' here works.

\begin{nota}
Let $\gsv[[\A]^{-1}][T]_\dag$ be the localization of $\gsv[[\A]^{-1}][T]$ with respect to the multiplicative family generated by the elements $1 - [\A]^{a}T^b$, where $a \in \Z$ and $b \in \Z^+$; it may be regarded as a subring of  $\gsv[[\A]^{-1}] \dbra T$. Similarly for $\ggk{\RVar}[[\A]^{-1}][T]_\dag$.
\end{nota}

Applying the forgetful homomorphism $\Phi$ in (\ref{tauhat:forget}) termwise to the coefficients of $Z^{1}(T)$, we obtain a zeta function $\bar Z^{1}(T)$. It is shown in \cite{F1}, using resolution of singularities, that $\bar Z^{1}(T)$ belongs to  $\ggk{\RVar}[[\A]^{-1}][T]_\dag$.
Letting ``$T$ go to infinity'' as described in \cite[\S~8.4]{hru:loe:lef}, we get a limit
\begin{equation}\label{pos:mil:no:act}
\bar{\mathscr S}^{1} \coloneqq - \lim_{T \limplies \infty} \bar Z^{1}(T) \in \ggk{\RVar}[[\A]^{-1}],
\end{equation}
which is understood as the \memph{real motivic Milnor fiber} of $f$. The following finer result is in the same spirit.

\begin{thm}\label{direct:mil}
The zeta function $Z^{1}(T)$ belongs to $\gsv[[\A]^{-1}][T]_\dag$ and
\begin{equation}\label{upsi:to:eb}
 \mathscr S^{1} \coloneqq - \lim_{T \limplies \infty} Z^{1}(T) = \Big( \Xi \circ \Theta \circ \bb E^{\diamond} \circ \int^{\diamond} \Big)([\mdl X]) \in \gsv[[\A]^{-1}].
\end{equation}
\end{thm}

The proof is the same as that of \cite[Theorem~8.11]{HL:modified}, one only needs to insert ``$\Xi$'' at suitably places; this is left to the reader.

As  pointed out in \cite[Remark~8.14]{HL:modified}, we cannot really take  the motivic Milnor fiber $\mathscr S^{1}$ of $f$ in $\gsv$, at least not if $\mathscr S^{1}$ is viewed as something obtained through  $Z^{1}(T)$. On the other hand, in light of Theorem~\ref{direct:mil} and Remark~\ref{big:dia:up}, we can forego the zeta function point of view and recover $\mathscr S^{1}$ in $\gsv$ directly as
\[
\Vol^{\mu_2}([\mdl X])= \Big( \Xi \circ \Theta \circ \bb E_{b} \circ \int \Big)([\mdl X]) = \Big( \Xi \circ \Theta \circ \bb E^{\diamond} \circ \int^{\diamond} \Big)([\mdl X]).
\]

\begin{rem}\label{comp:real:non}
The \memph{real nonarchimedean Milnor fiber} $\mdl X^{1}$ of $f$ is the set $\mdl X(\puR)$ of $\puR$-rational points of $\mdl X$. We can calculate $\mathscr S^{1}$ directly  as $\Vol_{\puR}^{\mu_2}([\mdl X^1])$. This is sometimes much simpler than working with the complex nonarchimedean Milnor fiber $\mdl X$; see  Example~\ref{mil:exam} below. The reason is that $\puR$ is real closed (and indeed \omin-minimal). This additional structure does give rise to a variant of the Hrushovski-Kazhdan  construction, which we shall discuss in \S~\ref{sec:tcon}.
\end{rem}

\begin{rem}\label{g:comm:alm}
By  (\ref{off:fac:A}),  the upper portion of (\ref{sum:narr}) almost commutes if $\bb E_b$ is replaced by $\bb E_g$. At any rate,  one can  consider the homomorphism $\Vol^{\mu_2}_g$, using $\bb E_g$ instead of $\bb E_b$, and attach $\Vol^{\mu_2}_g([\mdl X])$ to $f$ directly. This gives $\Vol^{\mu_2}_g([\mdl X]) = \Vol^{\mu_2}([\mdl X])[\A]^{-d}$. On the other hand,  the work in \cite{B2} shows that the Bittner dual $\mdl D(\mathscr S^{1})$ of $\mathscr S^{1} = \Vol^{\mu_2}([\mdl X])$ is $\mathscr S^{1}[\A]^{1-d}$. So $\Vol^{\mu_2}_g([\mdl X])[\A] = \mdl D(\mathscr S^{1})$. We will explain in a future paper how to recover these results around the Bittner duality without using the weak factorization theorem of \cite{WFT}.

One cannot help but wonder if there is  a more geometric interpretation of  $\Vol^{\mu_2}_g([\mdl X])$, on a par with that of the motivic Milnor fiber $\mathscr S^{1}$, and if their duality is actually a shadow of some sort of cohomological duality (categorification). In the opposite direction, going further down to the level of Euler characteristic, something definite can be said, see Corollary~\ref{link:open:closed}.
\end{rem}

\begin{exam}\label{mil:exam}
Consider the polynomial function $f(x,y)=x^6+x^2y^2+y^6$ on the affine plane and take $z$ to be the origin. We decompose the real nonarchimedean Milnor fiber $\mdl X^1$ into the following sets in $\RV(\puR)$:
\begin{gather*}
  A= \{y=0 \} \cap \{\rv(x^6) = \rv(t)\}, \quad A'= \{x=0\} \cap \{\rv(y^6) = \rv(t)\},\\
  B= \{\infty > \vv (y^6) >1\} \cap \{\vv(x^2y^2)>1 \} \cap \{\rv(x^6) = \rv(t)\},\\
    B'= \{\infty > \vv(x^6) >1\} \cap \{\vv(x^2y^2)>1\} \cap \{\rv(y^6) = \rv(t)\},\\
  C= \{\vv(x^6) =\vv(x^2y^2) = 1 \} \cap \{\rv(x^6+x^2y^2)=\rv(t)\}, \\
   C'= \{\vv(x^2y^2) = \vv (y^6) = 1\} \cap \{\rv(x^2y^2 + y^6) = \rv(t)\},\\
   D= \{\vv (x^6) >1\} \cap \{\vv (y^6) >1\} \cap \{\rv(x^2y^2)= \rv(t)\}.
\end{gather*}
By symmetry, $[ A] = [A']$, $[B] = [B']$, and $[ C] = [C']$ in $\ggk \VF_{\puR}$. Observe that if we work with the complex nonarchimedean Milnor fiber $\mdl X$ then the first term in $C$ should be  $\{ \vv(x^6) =\vv(x^2y^2) \leq 1 \}$, but  the only possibility for $C$ in $\RV(\puR)$ is the indicated condition because the leading terms of $x^6$ and $x^2y^2$ cannot cancel in $\puR$. This simplifies the computation tremendously. In comparison, we shall perform a similar decomposition in $\puC$ for a simpler polynomial (no mixed terms) in \S~\ref{section:TS}.

In terms of elements in $\ggk \RES_{\puR}[*] \otimes \ggk \Gamma_{\puR}[*]$ modulo $(\bm P - 1)$, the integrals $\int_{\puR}[A]$, $\int_{\puR}[B]$, and $\int_{\puR}[C]$ work out at, respectively,
\begin{gather*}
  [\{x^6 =\rv(t)\}]  \in \ggk \RES_{\puR}[1],\\
  [\{x^6=\rv(t)\}] \otimes [(1/3,\infty)^\sharp] \in \ggk \RES_{\puR}[1] \otimes \ggk \Gamma_{\puR}[1],\\
  [\set{(x,y) \in (1/6)^\sharp \times (1/3)^\sharp \given \tbk(x^6)+ \tbk(x^2y^2) = 1 }] \in \ggk\RES_{\puR}[2].
\end{gather*}
The assignment $(x,y)\efun (xy,y)$ gives a definable bijection between $D$ and
$$
\set{(x,y)\in \MM^2 \given \rv(x^2) =\rv(t) \tand 1/6<\vv(y)<1/3}
$$
and hence $\int_{\puR}[D]$ works out at
$$
[\{x^2=\rv(t)\}] \otimes [(1/6,1/3)^\sharp] \in \ggk \RES_{\puR}[1] \otimes \ggk \Gamma_{\puR}[1].
$$
Since $\chi_b((1/3,\infty))=0$ and $\chi_b((1/6,1/3))=-1$, we get, in $\gsv$,
\begin{equation}\label{compu:6}
\begin{split}
 \Vol_{\puR}^{\mu_2}([\mathcal X^1]) &= 2 \Vol_{\puR}^{\mu_2} ([A])+ 2 \Vol_{\puR}^{\mu_2} ([B]) + 2 \Vol_{\puR}^{\mu_2}([C])+ \Vol_{\puR}^{\mu_2} ([D])\\
    &= 2[\{x^6=1\}]+2[\{x^6+x^2y^2=1\}\cap \G_m^2]-[\G_m][\{x^2=1\}]\\
    &= 2[\{x^6+x^2y^2=1\}]-[\G_m][\{x^2=1\}],
\end{split}
\end{equation}
where the $\mu_2$-action is given by $(x,y)\efun (-x,y)$ for the first term and $x\efun -x$ for the second term. Then, applying the realization map $\beta^{\mu_2}$ in Remark \ref{rem:beta}, we get
\begin{equation}\label{exam:equi:beta}
(\beta^{\mu_2} \circ \Vol_{\puR}^{\mu_2})([\mathcal X^1]) = 2u-(u-1)=u+1.
\end{equation}
If we forget the $\mu_2$-action on $\Vol_{\puR}^{\mu_2}([\mathcal X^1])$, it becomes
$2[\{x^6+x^2y^2=1\}]-2[\G_m]$ in $\ggk \RVar$. If we take further the virtual Poincar\'e polynomial then it becomes $0$,
since $\{x^6+x^2y^2=1\}$ has the same virtual Poincar\'e polynomial as the unit circle minus two points.
\end{exam}

\subsection{Concerning the virtual Poincar\'e polynomial}\label{vPp}

Let $R$ be a real closed field. An $R$-variety is defined in the same way as a real  variety, but with  $\R$ replaced by $R$. The  corresponding  category of $R$-varieties is denoted by ${R}{\Var}$ and its Grothendieck ring by $\ggk {R}{\Var}$; we have seen the special case $R = \R$ in \S~\ref{sect:Groth:real}.

The virtual Poincar\'e polynomial is an invariant of $\RVar$, which is defined in \cite{mccrory:paru:virtual:poin}. The proof for its existence there relies on the weak factorization theorem of \cite{WFT} and Poincar\'e duality; the former is  valid over any field of characteristic $0$ and the latter is available for singular homology of compact nonsingular real algebraic varieties with $\F_2$-coefficients. Replacing singular homology with semialgebraic homology $H^{sa}$ with $\F_2$-coefficients (see \cite{DelKne} or \cite[\S~11.7]{BCR}), Poincar\'e duality still holds (in the semialgebraic setting ``compact'' means ``closed and bounded''). Thus, the proof goes through almost verbatim for ${R}{\Var}$:

\begin{thm}
There exists a unique  homomorphism $\beta^{R}: \ggk {R}{\var} \fun \Z[u]$ that assigns to each compact nonsingular $R$-variety $X$ its Poincar\'e polynomial $\sum_{i \in \N} \dim H_i^{sa}(X,\F_2)u^i$.
\end{thm}

If $R = \R$ then we  denote $\beta^{\R}$ simply by $\beta$ as in Remark \ref{rem:beta}.

\begin{rem}\label{poin:inv:ext}
Let $R \to R'$ be a real closed field extension. Let $X$ be an $R$-variety. Then the virtual Poincar\'e polynomial of the extension $X(R')$ of $X$ to $R'$ is equal to the virtual Poincar\'e polynomial of $X$. Actually, for $X$  compact and nonsingular, this follows immediately from  the invariance of semialgebraic homology under real closed field extension. The general case follows from additivity, expressing the class of $X$ in terms of classes of compact nonsingular $R$-varieties via resolution of singularities.
\end{rem}


Write $\Vol_{\puR}$ for the composition of $\Vol_{\puR}^{\mu_2}$ with the forgetful homomorphism $\gsv \fun \ggk \RVar$. Since $\tRVar$ is a subcategory of $\VF_{\puR}$, there is a natural homomorphism $\ggk \tRVar \fun \ggk \VF_{\puR}$. Composing this with $\Vol_{\puR}$ and then the virtual Poincar\'e polynomial map $\beta$, we obtain a homomorphism $\beta^{\lim} : \ggk \tRVar \fun  \Z[u]$. Thus we have found two homomorphisms $\beta^{\lim}$, $\beta^{\puR}$ from $\ggk \tRVar$ to  $\Z[u]$.

\begin{rem}
Over the algebraic closure of a henselian discretely valued field, it is shown in \cite[Proposition~3.23]{Nic:Pay:trop} that the analogue of $\beta^{\lim}$, defined with the Hodge-Deligne polynomial instead of the virtual Poincar\'e polynomial, gives the  Hodge-Deligne polynomial of the limit mixed Hodge structure associated with a variety. It would seem interesting to also compare $\beta^{\lim}$ with a similar map on limit structures, but such structures have yet to be constructed in the real framework.

Also, the duality of $\bb E_b$ and $\bb E_g$ described in Remarks~\ref{eb:not:eg} and~\ref{g:comm:alm} yields another  homomorphism $\beta^{\lim}_g : \ggk \tRVar \fun \Z[u]$.
\end{rem}

\begin{lem}\label{comsmo:same}
For all compact smooth real variety $X$, $\Vol^{\hat \delta}([X(\puC)]) = [X(\C)]$ in $\gdv$ and hence $\Vol_{\puR}^{\mu_2}([X(\puR)]) = [X]$ in $\gsv$ .
\end{lem}
\begin{proof}
Let $n$ be the  dimension of $X$ and choose a quasi-finite morphism $f : X \fun \A^n$ over $\R$. Set $\bm X = (X(\C) , f)$, which is treated as an object of $\RES[n]$. We have $[X(\C)] = \Theta(\bb E_b([\bm X]))$ in $\gdv$. Thus, for the first clause, it is enough to show $\int [X(\puC)] = [\bm X] / (\bm P - 1)$. This is  essentially the content of \cite[Lemma~13.3(2)]{hrushovski:kazhdan:integration:vf}, and the same proof works almost vertatim (the function $f$ needs to be adjusted so to become piecewise \'etale). The second clause is immediate from (\ref{hen:sum}) and (\ref{real:RV:com}).
\end{proof}

Combining this lemma with Remark~\ref{poin:inv:ext}, we get the following equality:
\begin{cor}
For any real algebraic variety $X$,
\[
\beta^{\lim} ([X(\puR)])= \beta([X]) = \beta^{\puR}([X(\puR)]).
\]
\end{cor}

However, the two homomorphisms do not coincide in general. Here is a counterexample:

\begin{exam}\label{exam:two:poin}
Consider the polynomial $f(x,y)=x^6+x^2y^2+y^6$ again. Let $X \sub \tilde \R^2$ be the $\puR$-variety given by the equation $f(x,y)=t$. Observe that we actually have $X \sub \MM(\puR)^2$ and hence $X$ is closed and bounded.

For any $t' \in \VF$ with $\rv(t') = \rv(t)$, there is an immediate automorphism $\sigma$ of $\puC$ over $\R$ with $\sigma(t') = t$, where ``immediate'' means that $\sigma$ fixes $\RV$ pointwise. Therefore, changing $t$ to $t'$ in the definition of $X$ does not change the value $\int [X(\puC)]$. It follows from compactness  that
\begin{equation}\label{fib:thick}
\int [\mdl X] = \int [\rv(t)^\sharp] \int [X(\puC)] =  [1]\int [X(\puC)]
\end{equation}
and hence, by (\ref{hen:sum}), $\int_{\puR} [\mdl X^1] = [1]\int_{\puR} [X]$, where again $\mdl X$, $\mdl X^1$ are the complex and  real nonarchimedean  Milnor fibers associated with $f$.


Now, since $X$ is  nonsingular and has only one connected component, it follows that $\beta^{\puR}([X])=1+u$.
On the other hand,
\[
(\beta \circ \Vol_{\puR})([\mdl X^1]) = (\beta \circ \Vol_{\puR})([X]) = \beta^{\lim}([X]).
\]
The expression $\Vol_{\puR}([\mdl X^1])$  may be understood as the motivic Milnor fiber $\bar{\mathscr S}^{1}$ of $f$ in (\ref{pos:mil:no:act}), taken  in $\ggk \RVar$. The computation towards the end of Example~\ref{mil:exam} shows that its virtual Poincar\'e polynomial  is $0$. So  $\beta^{\lim}([X]) \neq \beta^{\puR}([X])$.
\end{exam}


\section{In $T$-convex valued field}\label{sec:tcon}

It is also shown in \cite[\S~8]{hru:loe:lef} that one can recover, in a localization of $\gmv[\A^{-1}]$, the motivic zeta function and then the motivic Milnor fiber $\mathscr S$ of $f$ from its nonarchimedean Milnor fiber $\mdl X$. In \cite[Remark~8.5.5]{hru:loe:lef}, these results yield a proof, without using resolution of singularities but still using other sophisticated algebro-geometric machineries, that the Euler characteristic of $\mathscr S$ equals that of the topological Milnor fiber of $f$ (whether finer invariants such as the Hodge-Deligne polynomial can be recovered in this way is still unknown). In this section, we aim to prove this equality and its real analogue using a geometric argument at the level of $T$-convex sets instead. Moreover, as is already mentioned in Remark~\ref{eg:eb:differ}, in the real environment, the difference between the bounded and the geometric Euler characteristics in the $\Gamma$-sort is manifested as an equality relating  the Euler characteristics of the closed and the open topological Milnor fibers.

\subsection{The universal additive invariant}\label{intro:tcon}
We first summarize the main result of \cite{Yin:tcon:I}. To begin with, let $T$ be a complete polynomially bounded \omin-minimal \LT-theory extending the theory $\usub{\textup{RCF}}{}$ of real closed fields. It is not necessary in \cite{Yin:tcon:I}, but here we assume that $\R$ is a \T-model. Let $\mdl R \coloneqq (R, <, \ldots)$ be a nonarchimedean  \T-model containing $\R$ and $\OO \sub R$ be the convex hull of $\R$. Then $\OO$ is a \memph{proper} \T-convex subring of $\mdl R$ in the sense of \cite{DriesLew95}, that is, $\OO$  is a convex subring of $\mdl R$ such that, for every definable (no parameters allowed) continuous function $f : R \fun R$, we have $f(\OO) \sub \OO$. According to \cite{DriesLew95}, the theory $T_{\textup{convex}}$ of the pair $(\mdl R, \OO)$, suitably axiomatized in the language $\lan{}{convex}{}$ that extends $\lan{T}{}{}$ with a new unary relation symbol, is complete. We further assume that $T$ admits quantifier elimination and is universally axiomatizable, which can always be arranged through definitional extension. Then $T_{\textup{convex}}$ admits quantifier elimination too. It also follows that  $\R$ is an elementary \LT-substructure of $\mdl R$.


We may also view $\mdl R$ as an $\lan{}{RV}{}$-structure. To construct Hrushovski-Kazhdan style integrals in this environment, however, we need to work with a different language, which extends $\lan{}{RV}{}$. Since $1 + \MM$ is a convex subset of $R^{\times}$, the total ordering on $R^{\times}$ induces a total ordering on $\RV$. This turns $\RV$ into an ordered group and $\K$ into an ordered field. By the general theory of \T-convexity, there is a canonical way of turning $\K$ further into a \T-model, which is isomorphic to the \T-model  $\R$, with the isomorphism  given by the residue map $\res$. Let $\K^+$ be the set of positive elements of $\K$ (similarly for other totally ordered sets with a distinguished element), which forms a convex subgroup of $\RV$.

\begin{nota}\label{TG:euler}
Denote the quotient map $\RV \fun \Gamma \coloneqq \RV / \K^+$ by $\vrv$. The composition $\vv \coloneqq \vrv \circ \rv: R^\times \fun \Gamma$ is referred to as a \memph{signed} valuation map. The corresponding value group is a ``double cover'' of the  traditional value group. Consequently, the Euler characteristics, still denoted by $\chi_g$ and $\chi_b$, are slightly different from the ones in \cite[Remark~4.2]{HL:modified}.
\end{nota}

All of this structure can be expressed in a  two-sorted first-order language $\lan{T}{RV}{}$, in which $R$ is referred to as the $\VF$-sort and $\RV$ is taken as a new sort. The resulting theory $\TCVF$ (see \cite[Definition~2.7]{Yin:tcon:I}) is complete and weakly \omin-minimal, and admits quantifier elimination. Informally and for all practical purposes, the language $\lan{T}{RV}{}$ may be viewed as an extension of the language $\lan{}{convex}{}$.

\begin{exam}\label{exam:RtQ}
If $T = \RCF$ then we can turn $\puR$ into a model of $\TCVF$, with signed valuation, as follows. First note that $\rv$ is just the leading term map described in Example~\ref{exam:pui:C}, and  we may identify $\RV$ with $\Q \oplus \R^\times$. Then the ordering on $\RV$ is the same as the lexicographic ordering on $\Q \oplus \R^+$ or $\Q \oplus \R^-$ (but not both of them together due to the issue of sign). The quotient group $\Gamma = (\Q \oplus \R^\times) / \R^{+}$ is naturally isomorphic to the subgroup $\pm e^{\Q} \coloneqq e^{\Q} \cup - e^{\Q}$ of $\R^\times$, where $e=\exp(1)$, so that $\Q$ is identified with $e^\Q$ via the map $q \efun e^q$. Adding a new symbol $\infty$ to $\RV$, now we interpret $\puR$ as an $\lan{T}{RV}{}$-structure, with the signed valuation given by
\[
x \efun \rv(x) = (q, a_q) \efun \sgn(a_q)e^{-q},
\]
where $\sgn(a_q)$ is the sign of $a_q$. It is also a model of $\TCVF$: all the axioms in \cite[Definition~2.7]{Yin:tcon:I} are more or less immediately derivable from the valued field structure, except (Ax.~7), which holds since $\RCF$ is polynomially bounded, and (Ax.~8), which follows from \cite[Proposition~2.20]{DriesLew95}.
\end{exam}

Henceforth we assume $T = \RCF$ and work in the $\TCVF$-model $\puR$, with all parameters allowed. The reason that here all parameters are allowed is that we really gain nothing by restricting to $\bb S = \R \dpar t$ since, as has been remarked at the beginning of \S~\ref{section:spec:hen},  every element in $\puR$ is definable over $\R \dpar t$.

The categories $\VF[k]$, $\RV[k]$, $\RES[k]$, and $\RES$ are defined as before. Of course all notions are now formulated relative to $\TCVF$, in particular, ``definable'' means ``$\lan{T}{RV}{}$-definable,'' and so on. To distinguish them from the previous similar-looking categories, we shall write $\TVF[k]$, $\TRV[k]$, $\TRES[k]$, and $\TRES$ instead.

The $\Gamma$-categories contain subtle differences, though.

\begin{defn}[$\TG$-categories]\label{def:TGa:cat}
The objects of the category $\TG[k]$ are the finite disjoint unions of definable subsets of $\Gamma^k$. Any definable bijection between two such objects is a \memph{morphism} of $\Gamma[k]$. The category $\TG^{\fin}[k]$ is the full subcategory of $\TG[k]$ such that $I \in \TG^{\fin}[k]$ if and only if $I$ is finite.
\end{defn}

Every $\TG[k]$-morphism is definably a piecewise $\mgl_k(\Q)$-transformation; compare with \cite[Remark~3.15]{HL:modified}.

Observe that $\ggk \TG^{\fin}[k]$ is naturally isomorphic to $\Z$ for all $k$ and hence $\ggk \TG^{\fin}[*] \cong \Z[X]$. There is still a $\ggk \TG^{\fin}[*]$-linear map
\[
\Psi^T : \ggk \TRES[*] \otimes_{\ggk \TG^{\fin}[*] } \ggk \TG[*] \fun \ggk \TRV[*],
\]
which is an isomorphism of graded rings.

\begin{rem}[Explicit description of ${\ggk \TRES}$]\label{expl:res}
The semiring $\gsk \TRES$ is actually generated by isomorphism classes $[U]$ with $U$ a set in $\K^+$. We have the following explicit description of $\gsk \TRES$. Its underlying set is $(0 \times \N) \cup (\N^+ \times \Z)$, where the first coordinate indicates the dimension and the second the \omin-minimal Euler characteristic. For all $(a, b), (c, d) \in \gsk \TRES$,
\[
(a, b) + (c, d) = (\max\{a, c\}, b+d), \quad (a, b) \times (c, d) = (a + c, b \times d).
\]
The dimensional part is lost in the groupification $\ggk \TRES$ of $\gsk \TRES$, that is, $\ggk \TRES \cong \Z$, which is  much simpler than $\gsk \TRES$.
\end{rem}

The elements $[1]$, $\bm P$, and $[\A]$ in $\ggk \TRV[*]$, the lifting map $\bb L$, and the semiring congruence relation $\isp$ are also defined as before.

Proposition~\ref{base:Eb:Eg} still holds in the current environment:

\begin{prop}[{\cite[Proposition~4.24]{Yin:tcon:I}}]\label{prop:retr:TRES}
There are two ring homomorphisms
\[
\bb E_{g}^T: \ggk \TRV[*] \fun \ggk \TRES \cong \Z \quad \text{and} \quad \bb E^T_{b}: \ggk \TRV[*] \fun \ggk \TRES \cong \Z
\]
such that
\begin{itemize}[leftmargin=*]
  \item $\bm P - 1$ vanishes under both of them,
  \item for all $x \in \ggk \TRES[k]$ and all $y \in \ggk \TG[l]$,
  \begin{equation*}
    \bb E_{g}^T(x \otimes y) = (-1)^{k} \chi_g(y) x\dand \bb E_{b}^T(x \otimes y) = (-1)^l\chi_b(y)x.
  \end{equation*}
  where $x \otimes y$ stands in for the  element  $(\Psi^T)^{-1}(x \otimes y) \in \ggk \TRV[*]$.
\end{itemize}
\end{prop}

We can also write the last two equalities in a form that is not simplified so to make the similarity to Proposition~\ref{base:Eb:Eg} apparent (the classes are replaced by their Euler characteristics in the residue field):
\begin{equation*}
    \bb E_{g}^T(x \otimes y) = \chi_g(y) x (-1)^{l}(-1)^{-(k+l)} \dand \bb E_{b}^T(x \otimes y) = \chi_b(y)x (-1)^l 1^{-(k+l)}.
\end{equation*}
Note that $-1$ in the expression $(-1)^{l}$ is the Euler characteristic of the half torus (think $\R^+$), not the torus (think $\R^\times$); this is related to the use of signed valuation map, see Notation~\ref{TG:euler}. Both $\bb E^T_{b}$ and $\bb E^T_{g}$ will be relevant to  our construction below.

\begin{thm}[{\cite[Theorem~5.40]{Yin:tcon:I}}]\label{main:prop:tcon}
For each $k \geq 0$ there exists a canonical isomorphism of semigroups
\[
\int_{+}^T : \gsk  \TVF[k] \fun \gsk  \TRV[{\leq} k] /  \isp
\]
such that $\int_{+}^T [A] = [\bm U]/  \isp$ if and only if $[A] = [\bb L\bm U]$. Passing to the colimit yields a canonical isomorphism of semirings
\[
\int_{+}^T : \gsk \TVF_* \fun \gsk  \TRV[*] /  \isp.
\]
\end{thm}

\begin{thm}\label{gen:Euler}
There are a generalized Euler characteristic and two specializations to $\Z$:
\[
\chi^T_g, \chi^T_b : \ggk \TVF_* \to^{\int^T} \ggk \TRV[*] / (\bm P - 1) \two^{\bb E^T_{g}}_{\bb E^T_{ b}} \ggk \TRES \cong \Z.
\]
\end{thm}

\begin{exam}\label{rem:MM}
Let us compute the images of $[\MM]$ under these two generalized Euler characteristics. To begin with, $\int^T[\MM] = [1] / (\bm P - 1)$. Since $[1] + [\A] = 0$ in $\ggk \TRES[1]$, we have
\[
\chi^T_g([\MM]) = [1][\A]^{-1} = -1 \in (\ggk \TRES[*][[\A]^{-1}])_0 \cong \ggk \TRES.
\]
Similarly $\chi^T_b([\MM])=1$ in $(\ggk \TRES[*][[1]^{-1}])_0 \cong \ggk \TRES$. Thus  $\chi^T_g([\MM^+])= -1$ and $\chi^T_b([\MM^+])= 0$. It follows that, for any interval $(0,a]$ with $a \in \R^+$, by additivity, we have
\[
\chi^T_g([(0,a] \mi \MM^+])=1 \dand \chi^T_b([(0,a] \mi \MM^+])=0.
\]
\end{exam}

\begin{rem}\label{relate:Tcon:alg}
We can relate the isomorphism in Theorem~\ref{main:prop:tcon} to the ``purely algebraic'' isomorphism in Theorem~\ref{main:prop} via the following commutative diagram, extending (\ref{hen:sum}) with $\bb M = \puR$:
\begin{equation}\label{uniadd:to:tcon}
\bfig
 \hSquares(0,0)/->`->`->`->`->`->`->/<400>[{\ggk \VF_{\puR}}`{\ggk \RV_{\puR}[*] / (\bm P - 1)}`{\sggk \RES_{\puR}}`{\ggk \TVF_*}`{\ggk \TRV[*] / (\bm P - 1)}`{\ggk \TRES}; \int_{\puR}`\bb E_{b, \puR}````\int^T`\bb E^T_b]
\efig
\end{equation}
where the vertical arrows are all induced by the subcategory functors. If $\bb E_{b, \puR}$, $\bb E^T_b$ are replaced by $\bb E_{g, \puR}$, $\bb E^T_g$ then the diagram still commutes, and extends (\ref{hen:sum}) with   $\bb E_{g}$, $\bb E_{g, \puR}$ instead of $\bb E_{b}$, $\bb E_{b, \puR}$.
\end{rem}

\subsection{Link with the topological Milnor fiber}\label{sect:top}

Denote by $\Def_T$ the category of \LT-definable sets and  \LT-definable bijections. So $\Def_T$ is a subcategory of $\TVF_*$ and we have an induced  homomorphism
\[
\bm i : \ggk \Def_T \fun \ggk \TVF_*.
\]
Let $\chi :  \ggk \Def_T \fun \Z$ be the \omin-minimal Euler characteristic, which is an isomorphism; see \cite{dries:1998}. On the other hand, $\ggk \Def_T$ is also canonically isomorphic to $\ggk \TRES$ (Remark~\ref{expl:res}). Since $\chi$, $\chi_g^T \circ \bm i $, and $\chi_b^T \circ \bm i$ all agree on the class of the singleton $\{1\}$, they must be equal.

\subsubsection{The real case}\label{sec:Treal}

In semialgebraic geometry, the Borel-Moore homology is defined for locally compact semialgebraic sets and satisfies a long exact sequence, which gives rise to an additive and multiplicative Euler characteristic $\chi^{BM}$. It is equal to the Euler characteristic of the singular cohomology with compact supports, also only defined for locally compact semialgebraic sets. One can compute $\chi^{BM}$ on a cell decomposition, and the formula obtained can be used  to extend the definition of $\chi^{BM}$ to any semialgebraic set; see \cite[\S~1.8]{cost:RAS}. Consequently, $\chi^{BM}$ coincides with $\chi$ (this holds in general for any \omin-minimal theory, but we do not know a reference that contains a complete account of it).

\begin{rem}\label{eu:add}
The Euler characteristic coming from, say, the singular cohomology on semialgebraic sets, is not additive in general. It is however the case if one restricts to varieties over $\C$, for which it actually coincides with $\chi^{BM}$.
\end{rem}

\begin{nota}\label{nota:top:mil}
Let $X$, $f$, and $z$ be as in \S~\ref{subsec:milnor}.
Recall that the (positive) closed topological Milnor fiber is instantiated  by \LT-definable sets (in $\R$) of the form
$$
\bar F_{a,r} = \set{x\in X(\R) \given \norm{x-z} \leq r \text{ and } f(x) =a }, \quad 0<a \ll r  \ll 1,
$$
where $\norm{} : \VF^d \fun \VF $ denotes the Euclidean norm restricted to $\R$. The (positive) open topological Milnor fiber is similarly instantiated  by \LT-definable sets $F_{a,r}$, but with $\norm{x-z} \leq r$ replaced by $\norm{x-z} < r$.

Fix a $t\in \MM^+$. For each $r\in \VF^+$, the set $\bar{\mdl F}_r$ is defined as $\bar F_{a,r}$, but with $X(\R)$ replaced by $X(\VF)$ and $a$ by $t$ (since $t$ does not vary anymore, we drop it from the notation); similarly for $\mdl F_r$. So $\bar{\mdl F}_r$ is the topological closure of $\mdl F_r$. Let $\partial \bar{\mdl F}_r$ be the boundary of $\bar{\mdl F}_r$, that is,
$$
\partial \bar{\mdl F}_r = \bar{\mdl F}_r \mi \mdl F_r = \set{x \in X(\VF) \given \norm{x-z} = r \text{ and } f(x)=t }.
$$
Set $\mdl F = \bigcap_{r\in \UU^+} \bar{\mdl F}_r = \bigcap_{r\in \UU^+} \mdl F_r$, where $\UU = \OO \mi \MM$, or equivalently,
$$
\mdl F = \set{x \in X(\OO) \given \norm{x-z} \in \MM \text{ and } f(x)=t }.
$$
Since $\OO$ is the convex hull of $\R$, we can also write $\mdl F = \bigcap_{r\in \R^+} \bar{\mdl F}_r = \bigcap_{r\in \R^+} \mdl F_r$. Note that $\mdl F$ is definable but is in general not \LT-definable.
\end{nota}

\begin{prop}\label{propF}
The \omin-minimal Euler characteristic $\chi([\bar F_{a,r}])$ of the closed topological Milnor fiber is equal to $\chi^T_b ([\mdl F])$. Similarly, for the open topological Milnor fiber $F_{a,r}$,  we have $\chi([F_{a,r}]) = \chi_g^T ([\mdl F])$.
\end{prop}

The proof essentially consists of the following two lemmas.

\begin{lem}\label{lem:hardt}
If $r\in \UU^+$ is sufficiently small then, in $\ggk \TVF_*$,
$$
[\mdl F] = [\bar{\mdl F}_r]-[(0,r]\mi \MM^+][\partial \bar{\mdl F}_r] = [\mdl F_r]-[(0,r)\mi \MM^+][\partial \bar{\mdl F}_r].
$$
\end{lem}
\begin{proof}
The second equality is clear. For the first equality, we shall think of the \LT-definable subset
$A =\bigcup_{r\in \VF^+} r \times \partial \bar{\mdl F}_r$ of $\VF \times \VF^d$
as a fibration over $\VF^+$. By \omin-minimal trivialization (see \cite[\S~9.2.1]{dries:1998}), there exists an interval $[a, b] \sub \VF^+$ such that the sets $[a,b] \cap \MM$, $[a,b] \mi \MM$ are both nonempty and the fibration $A$ is \LT-definably trivial over $[a,b]$, that is, there is an \LT-definable homeomorphism
\[
 h:[a,b]\times \partial \bar{\mdl F}_b \fun \bigcup_{r\in [a,b]} r \times \partial \bar{\mdl F}_r,
\]
compatible with the projections onto $[a,b]$. Now, by additivity, it suffices to compute $[\bar{\mdl F}_b \mi \mdl F]$ in ${\ggk}{\TVF_*}$. Since $h$ induces a definable bijection between $\bar{\mdl F}_b \mi \mdl F$ and the product $((0,b]\mi \MM^+) \times \partial \bar{\mdl F}_b$, the desired equality follows.
\end{proof}

\begin{lem}\label{lem:same:Euler}
$\chi([\bar F_{a,r}]) = \chi([\bar{\mdl F}_r])$ and $\chi( [F_{a,r}]) = \chi([\mdl F_r])$.
\end{lem}

Note that all the occurrences of $\chi$ here stand for the \omin-minimal Euler characteristic, but on one side of the equality it is taken in $\R$, and in $\mdl R$ on the other side.

\begin{proof}
Considering $\bar F_{a,r}$ as a definable set in $\mdl R$, it has the same Euler characteristic (since any cell decomposition in $\R$ is also a cell decomposition in $\mdl R$) and, by \omin-minimal trivialization, there is a $t' \in \MM^+$ such that $\chi([\bar F_{a,r}]) = \chi([\bar{\mdl F}'_r])$, where $\bar{\mdl F}'_r$ is defined as $\bar{\mdl F}_r$ but with $t$ replaced by $t'$. Since $t$, $t'$ make the same cut in $\R$, there is an automorphism of $\mdl R$ over $\R$ mapping $\bar{\mdl F}_r$ to $\bar{\mdl F}'_r$. The first equality follows. The second equality is similar.
\end{proof}


\begin{proof}[Proof of Proposition~\ref{propF}]
By Lemma~\ref{lem:same:Euler}, we may show  $\chi^T_b ([\mdl F])=\chi([\bar{\mdl F}_{r}])$ and $\chi_g^T ([\mdl F])=\chi([\mdl F_{r}])$ instead. This is immediate by Example~\ref{rem:MM} and Lemma \ref{lem:hardt}.
\end{proof}

Recall the (positive) nonarchimedean Milnor fiber $\mdl X^{1}$ from Remark~\ref{comp:real:non}, which, in the presence of the Euclidean norm, may now be written as
\[
\set{x \in X(\OO) \given \norm{x-z} \in \MM \text{ and } \rv(f(x)) = \rv(t) }.
\]

\begin{thm}\label{thmX}
We can recover the Euler characteristics of the closed and the open topological Milnor fibers by applying $\chi_b^T$, $\chi_g^T$ to  $\mdl X^{1}$. More precisely,
\[
\chi_b^T ([\mdl X^{1}]) = \chi([\bar{F}_{a,r}]) \dand \chi_g^T ([\mdl X^{1}]) = -\chi([F_{a,r}]).
\]
\end{thm}

This follows from the equality:

\begin{lem}\label{recov:top:euler}
In $\ggk \TRV[*] / (\bm P - 1)$, $\int^T [\mdl X^{1}] = [1]\int^T [\mdl F]$.
\end{lem}
\begin{proof}
The argument has already been given in Example~\ref{exam:two:poin}. In the current setting, the immediate automorphisms in question are provided by \cite[Lemma~2.22]{Yin:tcon:I}.
\end{proof}

\begin{proof}[Proof of Theorem~\ref{thmX}]
Since $\bb E^T_b([1]) = 1$ and $\bb E^T_g([1]) = -1$, this is immediate by Lemma~\ref{recov:top:euler} and Proposition \ref{propF}.
\end{proof}

\begin{rem}\label{recov:xb:xg}
Composing the two diagrams  (\ref{hen:sum}), (\ref{uniadd:to:tcon}) together, we  recover the real motivic Milnor fiber of $f$ as $\Vol_{\puR}([\mdl X^{1}])$ (this is not written $\Vol_{\puR}^{\mu_2}([\mdl X^{1}])$ as suggested in Remark~\ref{comp:real:non} because all parameters are allowed in the current setting, which kills all the $\mu_2$-actions) and its Euler characteristic as $\chi^T_b([\mdl X^{1}])$. In parallel with \cite[Remark~8.5.5]{hru:loe:lef}, the latter, by the preceding discussion, is equal to the (Borel-Moore) Euler characteristic of the closed topological Milnor fiber of $f$. This result has been previously obtained in \cite[Theorem 4.12]{Comte:fichou}, whose method involves heavy dosage of resolution of singularities.

We have also recovered the Euler characteristic of the open topological Milnor fiber of $f$ (up to sign) from $\mdl X^{1}$ (for other method, see \cite[Remark 4.10]{Comte:fichou}), but this happens solely in the \T-convex environment and, unlike the closed topological Milnor fiber, whether it comes from a geometric object, dual to $\Vol_{\puR}([\mdl X^{1}])$ in some sense, or not, is unclear; see Remark~\ref{g:comm:alm}. The following equality might be a faint trace of this perceived duality.
\end{rem}

\begin{cor}\label{link:open:closed}
$\chi([\bar{F}_{a,r}]) = (-1)^{d-1} \chi([F_{a,r}])$
\end{cor}
\begin{proof}
This is immediate from (\ref{off:fac:A}) and Theorem~\ref{thmX}.
\end{proof}

This result has  also been obtained in \cite[Theorem~4.4]{Comte:fichou}. It would be very interesting to categorify this equality, that is, lifting it to one between homology groups. One conceivable way to do this, as suggested by the work in \cite{fichou:shiota:pui}, is to develop a sort of homology theory for definable sets in   $\puR$, or in $\puC$, which might also shed light on the mystery alluded to in Remark~\ref{g:comm:alm}.  The existence of such a theory, however, is purely hypothetical.

\begin{exam}
Consider the polynomial function $f(x,y)=x^py^q$ on the affine plane, where $p,q \in \Z^+$, and take $z$ to be the origin. Let $m = \gcd(p, q)$. Without loss of generality, $p /m$ is odd. Then the assignment $(x,y)\efun (x^{p/m}y^{q/m},y)$ gives a definable bijection between $\mdl X^1$ and
$$
\set{(x,y)\in \MM^2 \given \rv(x^m) =\rv(t) \tand 0 <\vv(y) < 1/q},
$$
which means that the integral $\int_{\puR}[\mdl X^1]/ (\bm P - 1)$ works out at
\[
[\{x^m=\rv (t)\}] \otimes [(0,1/q)^\sharp] \in \ggk \RES_{\puR}[1] \otimes \ggk \Gamma_{\puR}[1].
\]
So, in $\ggk \RVar[[\A]^{-1}]$,  we have
\[
\Vol_{\puR,b} ([\mdl X^1])=-[\G_m][\{x^m=1\}] \dand \Vol_{\puR,g} ([\mdl X^1])=-[\G_m][\{x^m=1\}][\A]^{-2},
\]
where the extra letters in the subscripts indicate which Euler characteristic is being used. Set $m'=1$ if $m$ is odd and $m'=2$ if $m$ is even. Then $\chi_b^T ([\mdl X^{1}]) = 2m'$ is the Euler characteristic of the closed topological Milnor fiber and $- \chi_g^T ([\mdl X^{1}])=-2m' $ is the Euler characteristic of the open topological Milnor fiber.
\end{exam}

The last vertical arrow in (\ref{uniadd:to:tcon}) and $\Theta_{\tilde \R}$ in (\ref{theta:puR}) (forgetting the $\mu_2$-actions) induce a homomorphism
\[
\ggk \RVar[[\A]^{-1}] \cong \sggk \RES_{\tilde \R}[[\A]^{-1}] \fun \ggk \TRES \cong \Z,
\]
which is just the semialgebraic Euler characteristic. Applying it termwise to the coefficients of the zeta function as defined in (\ref{real:zeta}), but using $\Xi_{\R} \circ \Phi$ instead of $\Xi$ (recall  (\ref{tauhat:forget})), we obtain a power series $Z^{top}(T)$ in $\Z \dbra T$. This series is, up to sign, the positive topological zeta function considered in \cite{koiParu:moti}. In more detail, for each $m \geq 1$, let $\mdl X_{m}^+$ be the following set of  truncated arcs at $z$:
\[
\set{ \varphi \in X(\R[t] / t^{m+1}) \given f(\varphi) = a t^m \mod t^{m+1} \text{ with } a \in \R^+ \tand \varphi(0) = z }.
\]

\begin{prop}
$Z^{top}(T) = -\sum_{m\geq 1} (-1)^{md} \chi (\mdl X_{m}^+) T^m \in \Z \dbra T$.
\end{prop}
\begin{proof}
Since the map $\mdl X_{m}^+  \fun \R^+$ given by $\varphi \efun a$ is a trivial fibration by \cite[Remark~1.1]{koiParu:moti}, we have $\chi(\mdl X_{m}^+)=\chi(\R^+)\chi(\mdl X_{m}^{1})$. So the equality follows from (\ref{z:coeff:change}).
\end{proof}

\subsubsection{The complex case}\label{top:complex}
We may consider the complex geometry in  $\puC$ over $\C \dpar t$ in the $\TCVF$-model $\puR$, since there is an interpretation, in the model-theoretic sense, of the $\lan{}{RV}{}$-structure $\puC$  in the $\lan{T}{RV}{}$-structure $\puR$; this is just a fancy way to say that, after fixing a square root $\sqrt{-1}$ of $-1$, $\puC$ may be identified with $\tilde \R^2$, $\C \dpar t$ with $\R \dpar t^2$, $\RV(\puC)$ with $\RV(\puR)^2$, and so on.

\begin{exam}
We think of $c \in \puC$ as  $a+\sqrt{-1}b$ but write it simply as a pair $(a, b) \in \tilde \R^2$; also denote $a$ by  $\Re c$ and $b$ by $\Im c$. Let $g, h\in \C[x_1,\ldots,x_n]$. Then the definable set
\[
\set{c \in \tilde\C^n \given \vv f(c)\leq \vv g(c) } \eqqcolon \{\vv f\leq \vv g\}
\]
can also be described as the union of the following two subsets of $\tilde \R^{2n}$:
\begin{gather*}
 \{\vv \Re f \leq \vv \Re g \} \cap \{\vv \Re f \leq \vv \Im g \}, \quad
  \{\vv \Im f \leq \vv \Re g \} \cap \{\vv \Im f \leq \vv \Im g \}.
\end{gather*}
\end{exam}

Thus, definable sets in $\puC$ may be regarded as definable sets in $\puR$. Let $\VF_*$, $\RV[*]$, etc., be the categories in \cite[\S~3]{HL:modified} with $\bb S = \C \dpar t$. Then
there is an induced faithful functor $\VF_* \fun \TVF_{*}$, which in turn yields a homomorphism $\gD_{\VF} : \ggk \VF_* \fun \ggk \TVF_{*}$.

For the pair of $\RV$-categories $\RV[*]$ and $\TRV[*]$, although a similar functor is available, we need to be more careful since these categories are graded. To illustrate the concern, consider the object $\RV_\infty(\puC) = \RV_\infty(\puR)^2 $. Since the real line $(\puR, 0) \sub \puC$ and the imaginary line $(0, \puR) \sub \puC$ have only one nonzero coordinate, this object has nonempty components in all of the three categories $\TRV[0]$, $\TRV[1]$, and $\TRV[2]$. This interpretation leads to an issue since, for instance, the complex points $(1, 0)$ and $(1, 1)$ should certainly be isomorphic objects, but they cannot be since they do not even belong to the same graded piece.

To resolve this issue, we can work with a dimension-free version of  $\ggk\RV[*]$, namely the zeroth graded piece $(\ggk\RV[*][[1]^{-1}])_0$ of $\ggk\RV[*][[1]^{-1}]$. This ring is indeed isomorphic to $\ggk \RV_*$, where $\RV_*$ is the category of definable sets and bijections in $\RV$, and can also be obtained from $\RV[*]$ by forgetting  $f$ in  $(U, f) \in \RV[*]$.
There is the forgetful  epimorphism $\ggk\RV[*] \fun \ggk \RV_*$. The pushforward ideal of $(\bm P - 1)$ along this epimorphism is still denoted as such. It follows from the construction of $\bb E_b$ that there is a homomorphism
\[
\bb E_{b*} : \ggk \RV_* / (\bm P - 1) \fun \sggk \RES
\]
whose composition with the epimorphism
\[
\ggk \RV[*] / (\bm P - 1) \fun \ggk \RV_* / (\bm P - 1)
\]
is $\bb E_b$. All this also applies to  $\ggk \TRV[*]$, $\ggk \TRV_*$, and $\ggk \TRES$, and  the corresponding homomorphism is denoted by $\bb E^T_{b*}$. Since, in $\ggk \TRV_*$, we have,
\[
[(1,1)] - [\RV^{\circ\circ}_\infty(\puC)] = 1 - [\RV^{\circ\circ}_\infty(\puR)]^2 = 1 - (2 - \bm P)^2 = (\bm P - 1)(3 - \bm P),
\]
the ideal $(\bm P - 1)$ of $\ggk \RV_*$ is included in the eponymous ideal of  $\ggk \TRV_*$.

\begin{figure}[htb]
\begin{equation}\label{complex:RV:cat}
\bfig
 \Atrianglepair(0,400)/<-`->`->`->`->/<1200,400>[{\ggk \RV[*] / (\bm P - 1)}`{\ggk \VF_*}`{\ggk \RV_* / (\bm P - 1)}`{\sggk \RES}; \int``\bb E_b``]
 \square(0,0)|blla|/->`->`->`->/<1200, 400>[{\ggk \VF_*}`{\ggk \RV_* / (\bm P - 1)}`{\ggk \TVF_*}`{\ggk \TRV_* / (\bm P - 1)}; \int^*`\gD_{\VF}``\int^T_{*}]
 \square(1200,0)|blla|/->`->`->`->/<1200,400>[{\ggk \RV_* / (\bm P - 1)}`{\sggk \RES}`{\ggk \TRV_* / (\bm P - 1)}`{\ggk \TRES}; \bb E^*_b`\gD_{\RV}`\gD_{\RES}`{\bb E^T_{b*}}]
 \Vtrianglepair(0,-400)/->`->`->`<-`<-/<1200,400>[{\ggk \TVF_*}`{\ggk \TRV_* / (\bm P - 1)}`{\ggk \TRES}`{\ggk \TRV[*] / (\bm P - 1)}; ``\int^T``{\bb E^T_{b}}]
\efig
\end{equation}
\end{figure}

In conclusion, we have a commutative diagram (\ref{complex:RV:cat}). Similarly, (\ref{complex:RV:cat}) still commutes if $\bb E_b$, $\bb E^*_b$, etc., are replaced by $\bb E_g$, $\bb E^*_g$, etc.

\begin{rem}
Since $\gD_{\RES}([\K(\puC)]) = [\K(\puR)]^2 = 1$ in $\ggk \TRES$, we may, in (\ref{complex:RV:cat}), replace $\sggk \RES$ with $\sggk \RES/([\A] - 1)$. This has the effect of equalizing $\bb E_b$, $\bb E_g$ (see \cite[Remark~4.4]{HL:modified}) and thence the upper portions of the two versions of (\ref{complex:RV:cat}). So there is only one homomorphism from $\ggk \VF_*$ to $\ggk \TRES \cong \Z$, in other words, $\chi^T_b \circ \gD_{\VF} = \chi^T_g \circ \gD_{\VF}$; denote it by $\chi_{\C}$. This is perhaps another manifestation of the phenomenon  alluded to in Remark~\ref{eu:add}.
\end{rem}

Let $X$, $f$ be defined over $\C$ and $z$ a $\C$-rational point. In light of the complex version of (\ref{upsi:to:eb}) (see \cite[Remark~8.14]{HL:modified}), the  motivic Milnor fiber $\mathscr S$ of $f$ may be taken as $(\Theta \circ \bb E_{b} \circ \int)([\mdl X]) \in \gmv$, where $\mdl X$ is the nonarchimedean Milnor fiber of $f$ as defined in (\ref{com:nonarch}).

Now, we proceed to recover the result  in \cite[Remark~8.5.5]{hru:loe:lef}. The sets $\bar F_{a,r}$, $\bar {\mdl F}_r$, $\partial \bar{\mdl F}_r$, and $\mdl F$ are defined as in Notation~\ref{nota:top:mil}, but with $X(\R)$ replaced by $X(\C)$ and $X(\VF)$ by $X(\puC)$. Then the computations in the real case still go through almost verbatim. In $\ggk \TVF_*$, we still have, for all sufficiently small $r\in \UU^+$,
\[
[\mdl F] = [\bar{\mdl F}_r]-[(0,r] \mi \MM^+(\puR)][\partial \bar{\mdl F}_r].
\]
Thus $\chi^T_b([\mdl F]) = \chi([\bar{\mdl F}_r])$. On the other hand, since $\chi(\partial \bar{\mdl F}_r)=\chi(\partial F_{a,r})=0$ (the smooth compact complex manifold $\partial F_{a,r}$ is of odd dimension), $\chi^T_g([\mdl F]) = \chi([\bar{\mdl F}_r])$ as well.
Now, observe that, in the present setting,
\[
\int^T [\mdl X] = [1]^2\int^T [\mdl F] \in \ggk \TRV[*] / (\bm P - 1),
\]
and hence it makes no difference which one of the generalized Euler characteristics $ {\chi^T_g}$, ${\chi^T_b}$ is used to relate $\mdl X$ and $\mdl F$.

\begin{thm}
The Euler characteristic of the topological Milnor fiber of $f$ is equal to $\chi^T_b([\mdl F])$ and $\chi^T_g([\mdl F])$,  which in turn are  equal to $\chi_{\C}([\mdl X])$.
\end{thm}

To see that this is indeed the result  in \cite[Remark~8.5.5]{hru:loe:lef}, we need to show  $\chi_{\C}([\mdl X]) = (\gD_{\RES} \circ \Theta^{-1})(\mathscr S) =  \chi_c(\mathscr S)$, where $\chi_c$ is the homomorphism
\[
 \gmv \fun \ggk \Var_{\C} \fun \ggk \Var_{\C} / ([\A] - 1) \fun \Z
\]
with the last arrow given by the compactly supported Euler characteristic, see \cite[(5.5.2)]{hru:loe:lef}. Since both $\gD_{\RES} \circ \Theta^{-1}$ and $\chi_c$ are induced by the topological Euler characteristic of varieties over $\C$, they must be equal.

\section{Thom-Sebastiani formula}\label{section:TS}

Let $X$ be a smooth connected variety and $f$, $g$ nonconstant morphisms from $X$ to the affine line, all defined over $\C$. In this section we aim to establish a local motivic Thom-Sebastiani formula for composite morphisms on $X$ of the form $h(f, g)$, where $h(x, y)$ is a polynomial of the form
\[
y^{N} + \sum_{2 \leq \imath \leq \ell} x^{m_{\imath}}, \quad m_2 \ll N \ll m_3 \ll \ldots \ll m_\ell;
\]
here we may rename $N$ as $m_1$, but it plays a special role and hence is denoted differently. The actual condition we shall assume is somewhat weaker than this, see Hypothesis~\ref{hyp:seg:inc}.

\subsection{Combinatorial data and Galois actions of the torus}\label{com:dat}

The said formula expresses the motivic Milnor fiber of $h(f, g)$ as a sum of (iterated) motivic Milnor fibers of morphisms derived from $f$, $g$ and their convolution products. Before diving into technicalities, we first describe how the various terms in the sum are singled out based on certain combinatorial data that is read off from the tropical curve of $h(x, y)$.

Consider the planes in $(\Q^+)^3$ defined by the following equations: $z = 1$, $z = Ny$, and $z = m_{\imath} x$ for $2 \leq \imath \leq \ell$. The points with the lowest $z$-coordinates on these planes form the surface of a convex polyhedron whose edges are the pairwise intersections of the three planes $z = 1$, $z = Ny$, and $z = m_{2} x$. The tropical curve $H$ of  $h(x, y)$ is the orthogonal projection of these edges in the $(x, y)$-plane. Thus, $H$ consists of two rays $H_1$, $H_2$ and a line segment $H_3$, all emanating from the point $(1/m_2, 1/N)$, see the illustration on the left in Figure~\ref{TS:fig}. Both $H_1$ and $H_2$ contribute a term  of (iterated) motivic Milnor fiber to the formula.

\begin{figure}[htb]
\begin{tikzpicture}[scale = 1]
\draw [<->, help lines] (0, 3) -- (0, 0) -- (4, 0); \node [below left] at (0, 0) {$(0, 0)$}; \node [below left] at (4, 0) {$x$}; \node [below left] at (0, 3) {$y$};
\draw [thick, red] (0,0) -- (2, 1);
\draw [thick] (2, 3) -- (2,1) -- (4,1);
\draw [dotted] (0, 1) -- (2, 1); \node [left] at (0, 1) {$\frac 1 N$};
\draw [dotted] (2, 0) -- (2, 1); \node [below] at (2, 0) {$\frac 1 {m_2}$};
\draw [thick, red] (6,0) -- (14, 0); \node [below] at (6, 0) {$(0,0)$}; \node [below] at (14, 0) {$(\frac 1 {m_2}, \frac 1 N)$};
\node [right] at (2, 2.2) {$H_2$}; \node [above] at (3.2, 1) {$H_1$}; \node [above] at (1.4, .1) {$H_3$};
\draw [thick] (6, 0) -- (6,3) -- (14,3) -- (14, 0);
\draw [thick] (6, 0) -- (14,3) (6, 0) -- (12,3) (6, 0) -- (10.5,3) (6, 0) -- (7.5,3);
\draw [dotted] (7.5, 0) -- (7.5, 3); \node (1) [below] at (7.5, 0) {$(\alpha_\ell, \beta_\ell)$};
\draw [dotted] (10.5, 0) -- (10.5, 3); \node (2) [below] at (10.5, 0) {$(\alpha_4, \beta_4)$};
 \path (1) -- node[auto=false]{\ldots} (2);
\draw [dotted] (12, 0) -- (12, 3); \node [below] at (12, 0) {$(\alpha_3, \beta_3)$};
\draw [fill]  (7.5, .56) circle [radius=.05] (10.5, 1.69) circle [radius=.05] (12, 2.25) circle [radius=.05] (7.5, .75) circle [radius=.05] (10.5, 2.25) circle [radius=.05] (7.5, 1) circle [radius=.05] (7.5, 3) circle [radius=.05] (10.5, 3) circle [radius=.05] (12, 3) circle [radius=.05] (14, 3) circle [radius=.05];
\node [above] at (6.7, -1.2) {$L_{\ell}$}; \node (3) [above] at (8.2, -1.2) {$L_{\ell-1}$}; \node (4) [above] at (10, -1.2) {$L_{4}$}; \node [above] at (11.3, -1.2) {$L_3$}; \node [above] at (13, -1.2) {$L_2$};
\path (3) -- node[auto=false]{\ldots} (4);
\end{tikzpicture}
  \caption{The tropical curve $H$ of $h(x,y)$ and the vertical rectangular pane $P$ of height $1$ on the line segment $H_3$.}\label{TS:fig}
\end{figure}
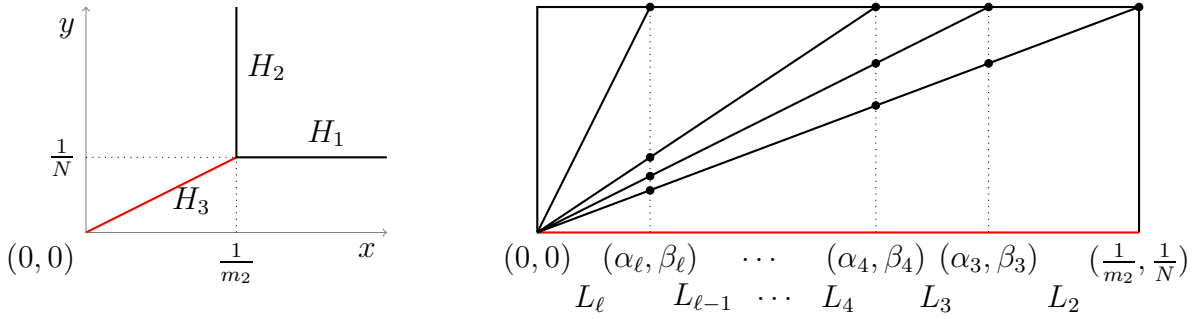

For the other terms, we need to examine the vertical rectangular pane $P \sub (\Q^+)^3$ of height $1$ standing on the line segment $H_3$, see the illustration on the right in Figure~\ref{TS:fig}. For each $2 \leq \imath \leq \ell$, let $\alpha_\imath = 1/m_\imath$ and $\beta_\imath = m_2 /N m_\imath$. Each plane $z = m_{\imath} x$ intersects $P$ at the oblique line segment connecting $(0,0,0)$ and $(\alpha_{\imath}, \beta_{\imath}, 1)$ in $(\Q^+)^3$. Let $L_\imath \sub (\Q^+)^2$ be the open line segment between the two points $(\alpha_\imath, \beta_\imath)$ and $(\alpha_{\imath+1}, \beta_{\imath+1})$, where we set $\alpha_{\ell+1} = \beta_{\ell+1} = 0$. Then:
\begin{itemize}[leftmargin=*]
  \item Each point $(\alpha_{\imath}, \beta_{\imath})$ with $\imath > 2$  contributes a term of (iterated) motivic Milnor fiber.
  \item For each $\imath \geq 2$, the sequence of points above $(\alpha_{\imath}, \beta_{\imath})$ that lie on the oblique line segments contribute another term,
  \item so do the corresponding  open line segments that lie over $L_\imath$.
\end{itemize}
These last two terms are jointly referred to as a term of convolution product.

\begin{rem}
In this section, we choose to work with varieties over $\C$ with $\G_m$-actions instead of $\hat \mu$-actions. This corresponds to  working in the $\ACVF$-model $\puC$ with $\bb S = \C \cup \Q$. Now, although  $\Gamma \cong \Q$ is only a definable sort of $\puC$ (see Remark~\ref{imag:Gam}), the Hrushovski-Kazhdan integration theory still goes through. This is not explicitly stated in \cite{hrushovski:kazhdan:integration:vf} but is included in the more general assumption of ``effectiveness'' there; in \cite{Yin:special:trans} and its sequels, $\vv(\VF(\bb S))$ is assumed to be nontrivial, but this is merely for convenience and is by no means an essential requirement. Note that $\Q$ will become redundant if any additional parameters from,  say, $\MM(\puC)$ or $\RV^{\circ\circ}(\puC)$ are used, since then every element in $\Q$ is definable.
\end{rem}

Let $z \in f^{-1}(0)$ be a $\C$-rational point. As before, since the discussion below will be of a local nature, we may assume that $X$ is  affine (hence a definable subset of $\VF^n$ for some $n$) and, without loss of generality, $z = 0$. Write $X \cap \MM^n$ as $X(\MM)$. We shall consider definable sets of the form
\[
\mdl X^\sharp_\gamma = \set{ x \in X(\MM) \given \vv(f(x)) = \gamma }, \quad \gamma \in \Gamma^+;
\]
for simplicity, $\mdl X^\sharp_1$ shall just be written as $\mdl X^\sharp$, which is of primary interest, and the restriction $f \rest \mdl X^\sharp_\gamma$  as $\mdl X^\sharp_\gamma$ (this will become a general notational scheme below). For each $u \in \RV$ and each $a \in u^\sharp \sub \VF$, let
\begin{gather*}
 \mdl X_a = \set{ x \in X(\MM) \given f(x) = a  } \dand    \mdl X_u = \set{ x \in X(\MM) \given \rv(f(x)) = u },
\end{gather*}
which are $a$-definable sets; so $\mdl X_{\rv(t)}$ is just the set  called the nonarchimedean Milnor fiber of $f$ above. The following equality relating $\int [\mdl X_u] $ and $\int [\mdl X_a]$ generalizes (\ref{fib:thick}), and shall be used frequently (and often implicitly); the argument for it is the same as the one given thereabout.
\begin{lem}\label{rv:int:bun}
$\int [\mdl X_u] =   [1]\int [\mdl X_a]$.
\end{lem}

For each $a \in \C^\times$, there is an automorphism $\C \dpar t \fun \C \dpar t$ sending $t$ to $at$. Thus, there is a subgroup of $\gal(\C \dpar t/ \C)$ that may be identified with $\C^\times$; the preimage of $\C^\times$ along the canonical surjective homomorphism
\[
\gal(\puC / \C) \fun \gal(\C \dpar t/ \C)
\]
is denoted by $\hat \tau$. We have $\hat \tau \cong \lim_n \C^\times_n$, where each $\C^\times_n$ is just a copy of $\C^\times$ and the transition morphisms are the same as in the limit $\hat \mu = \lim_n \mu_n$;  so for each $n$ there is a canonical epimorphism
$\tau_n : \hat \tau \fun \C^\times_n$,
which is a part of the limit construction (in the category of groups, say). More concretely, the elements in $\hat \tau$  may be identified as sequences $\hat a=(a_n)_n$ of $n$th roots of $a$, $a\in \C^\times$, satisfying $a_{kn}^n=a_k$. Such an element acts on $\puC$ by $\hat a \cdot t^{1/n}= a_n t^{1/n}$. We have a short exact sequence
$$
1 \fun \hat \mu \fun \hat \tau \fun \C^\times \fun 1.
$$
This sequence does not split, though.

\begin{rem}\label{tauhat:1:iden}
Here is a different perspective on $\hat \tau$. By the structural theory of valued fields, an element $\sigma \in \gal(\puC / \C \dpar t)$ is in the ramification subgroup if and only if it fixes $\RV$ pointwise (see \cite[Lemma~5.3.2]{engler:prestel:2005}). But it can be easily checked that every $\sigma \in \gal(\puC / \C \dpar t)$ moves some element of $\RV$ unless $\sigma = \id$. So $\gal(\puC / \C \dpar t) \cong \hat \mu$ may be identified with $\aut(\RV / \RV(\C \dpar t))$, where $\RV(\C \dpar t)$ is equal to the subgroup of $\RV$ generated by $\rv(t)$ over $\K^\times$.

For each $u \in \K^\times$, there is an automorphism $\aut(\RV / \K^\times)$ sending $\rv(t)$ to $u\rv(t)$; observe that an automorphism in $\aut(\RV / \K^\times)$  fixes $\Gamma \cong \RV / \K^\times$ pointwise if and only if it is of this form. So $\hat \tau$ may also be identified with a subgroup of $\aut(\RV / \K^\times)$, namely $\aut(\RV / \kuq)$.
\end{rem}

\begin{rem}\label{cross:Gm}
From yet a different perspective, recall from Remark~\ref{shat:section} that there exists a natural bijection between $\hat \mu$ and the set of reduced cross-sections $\rcsn : \Q \fun \RV$ with $\rcsn(1) = \rv(t)$ in $\puC$. Of course this is still the case if we change $\rv(t)$ to any other element of the form $u\rv(t)$, $u \in \K^\times$. Consequently, we may identify $\hat \tau$ with the set of all such reduced cross-sections.

Since every reduced cross-section $\rcsn$ determines a reduced angular component $\ac : \RV \fun \K^\times$ via the assignment $u \efun \tbk(u)$ and, conversely, every reduced angular component $\ac$ determines a reduced cross-section $\rcsn$ with $\rcsn(\Q) = \ac^{-1}(1)$, we see that $\hat \tau$ may also be identified with the set of all such reduced angular components.

Intuitively, as we have seen above, all this is just saying that if any reduced angular component or reduced cross-section is chosen and added to the structure of $\RV$ then we have an intrinsic isomorphism $\RV \cong {\K^\times} \oplus \Q$, and hence if both $\K^\times$ and $\Q$ are fixed pointwise then $\RV$ has no symmetries left other than the trivial one.
\end{rem}

\begin{rem}\label{whatis:good}
Similar to the case $\bb S = \C \dpar t$,  elements in $\sggk \RES$ now carry \memph{good} $\hat \tau$-actions, that is, those $\hat \tau$-actions that factor through some $\tau_n$ and hence may be considered as $\G_m$-actions. To emphasize this and to distinguish it from the similar ring with good $\hat \mu$-actions, we shall denote $\sggk \RES$ by $\sggk^{\hat \tau} \RES$   over $\bb S = \C \cup \Q$ and by $\sggk^{\hat \mu} \RES$ over $\bb S = \C \dpar t$.

An action of an algebraic group $G$ on a variety $Y$, all defined over $\C$, is \memph{good} if every orbit is contained in an affine open subset of $Y$. If $G$ is finite and $Y$ is quasi-projective then this condition always holds, which is why we have not brought it up until now.

From here on, only  good $\G_m$-actions --- in both senses when applicable --- will be considered, and hence we will drop the qualifier ``good'' from the terminology altogether.
\end{rem}

\begin{defn}\label{defn:weighted}
Let $Y$ be a variety over $\C$ with a $\G_m$-action $h$. We say that $h$ is \memph{$n$-weighted}, for some $n \in \Z^+$, if there is a morphism $\pi : Y \fun \G_m$ such that $\pi(c \cdot y) = c^{n} \pi(y)$ for all $c \in \G_m$ and all $y \in Y$. We also say that $h$ is \memph{$0$-weighted} if it is trivial, and the only witness to this is the morphism $Y \fun 1$. Observe that if there is a $\G_m$-equivariant isomorphism between $(Y, h)$ and $(Y', h')$ then $h$ is $n$-weighted if and only if $h'$ is $n$-weighted. Moreover, if $h$ is $n$-weighted with a witness $\pi$ and $h'$ is $n'$-weighted with a witness $\pi'$ then
\[
\pi(c \cdot y) \pi'(c \cdot y') = c^{n+n'} \pi(y)\pi'(y'), \quad \text{for all $c \in \G_m$, all $y \in Y$, and all $y' \in Y'$},
\]
and hence the diagonal $\G_m$-action $h \times h'$ on $Y \times Y'$ is $(n+n')$-weighted.

The category $\Var_{\C}^{\tau_n}$ consists of the varieties over $\C$ with $n$-weighted $\G_m$-actions and the $\G_m$-equivariant morphisms between them. Denote by $\Var_{\C}^{\hat \tau}$ the colimit of the inductive system of  $\Var_{\C}^{\tau_n}$, $n \in \N$, where the transition functors $\Var_{\C}^{\tau_n} \fun \Var_{\C}^{\tau_{kn}}$, corresponding to multiplication of integers, are given by $(Y, h) \efun (Y, h^k)$; so there are no functors between $\Var_{\C}^{\tau_0}$ and other $\Var_{\C}^{\tau_n}$.
\end{defn}

We may and do think of an object of $\Var_{\C}^{\hat \tau}$ as equipped with a $\hat \tau$-action that factors through some $\tau_n$, hence the notation.

The Grothendieck groups $\ggk^{\tau_n} \Var_{\C}$, $\ggk^{\hat \tau} \Var_{\C}$ are constructed subject to the usual condition on trivializing $\G_m$-actions on affine line bundles, analogous to (\ref{compl:flat}). So
$\ggk^{\hat \tau} \Var_{\C}$ is the colimit of $\ggk^{\tau_n} \Var_{\C}$, $n \in \N$, and is indeed a commutative ring, with the product induced by that in $\Var_{\C}^{\hat \tau}$.


\begin{rem}\label{const:theta}
Choose a reduced cross-section $\rcsn : \Q \fun \RV$; the point $\rv(t) \in \RV$ is not special  in the present setting (the object of interest shall be $\mdl X^\sharp$, not $\mdl X_{\rv(t)}$) and hence we no longer demand $\rcsn(1) = \rv(t)$. Let $U \sub \gamma^\sharp \sub \RV^n$ be an object of $\RES$, where $\gamma_1 \leq \ldots \leq \gamma_n$. Let $d$ be the least positive integer such that $U$ is a set in $\RV(\C \dpar{t^{1/d}})$; note that any other such integer is a multiple of $d$. So the $\hat \tau$-action on $U$ factors through $\tau_d$. Write $\gamma_n = e / d$ for some $e \in \Z^+$ and consider the function $\pi : U \fun \RV$ given by $u \efun u_n^d$. Then $\pi(c \cdot u) = c^{e d} \pi(u)$ for all $c \in \G_m$ and all $u \in U$. So, if $\tbk(U)$ is a variety over $\C$ then it is $ed$-weighted, which is witnessed by $\tbk(\pi)$.
\end{rem}

Recall that the isomorphism in \cite[\S~4.3]{hru:loe:lef} is constructed via twistback; henceforth we denote it by
$\Theta^{\hat \mu} : \sggk^{\hat \mu} \RES \fun \gmv$. Similarly, there is an isomorphism
$\Theta^{\hat \tau} : \sggk^{\hat \tau} \RES \fun \ggk^{\hat \tau} \Var_{\C}$,  and the diagram
\begin{equation*}
\bfig
\Square(0,0)/->`->`->`->/<400>[{\sggk^{\hat \tau} \RES}`{\ggk^{\hat \tau} \Var_{\C}}`{\sggk^{\hat \mu} \RES}`{\gmv}; \Theta^{\hat \tau}```\Theta^{\hat \mu}]
\efig
\end{equation*}
indeed commutes, where the first vertical arrow is induced by the subcategory relation and the second vertical arrow is induced by the obvious forgetful functor ($\hat \mu$ is a subgroup of $\hat \tau$).

\begin{rem}\label{theta:suj}
The modified argument for the surjectivity of $\Theta^{\hat \tau}$ is not as straightforward as that in Remark~\ref{first:theta}. A bit of  model theory is needed.

Let $(Y, h) \in \Var_{\C}^{\tau_n}$  and $\pi : Y \fun \G_m$  witness that $h$ is $n$-weighted. Without loss of generality, $Y$ is irreducible and quasi-projective. Since the first-order theory of algebraically closed fields admits elimination of imaginaries,  there are a definable set $Z$ in $\G_m$ and a definable surjection $\omega: Y   \fun Z$ each of whose fibers contains precisely one $h$-orbit. So $\omega \oplus \pi : Y \fun Z \times \G_m$ is a definable finite-to-one surjection each of whose fibers inherits a $\mu_n$-action from $h$. By the Kummer-theoretic construction in the proof of \cite[Proposition~4.3.1]{hru:loe:lef} and compactness, we may assume that there are a definable function $\eta : Z  \fun \G_m$ and a definable bijection
\[
\zeta_1 : (\omega \oplus \pi)^{-1}(Z \times 1) \fun \set{(z, 1 ,v) \in Z \times \G_m \times \G_m \given \eta(z) = v^n}.
\]
For $y \in Y$ with $\zeta_1(y) = (z, 1 , v)$ and $c \in \G_m$, set $\zeta(c \cdot y) = (z, c^n, cv)$. It can be checked that if $c \cdot y = c' \cdot y'$ then $\zeta(c \cdot y) = \zeta(c' \cdot y')$ and hence $\zeta$ is a $\G_m$-equivariant bijection from $(Y, h)$ onto a set $V \sub Z \times \G_m \times \G_m$, where the $\G_m$-action on $V$ is given by $c \cdot (z, w, v) = (z, c^nw, cv)$. Let $U \sub Z \times 1^\sharp \times(1/n)^\sharp$ such that $\tbk(U) = V$. Then $\Theta^{\hat \tau}([U]) = [(Y, h)]$.
\end{rem}

Following the notational scheme introduced in Remark~\ref{note:vol}, let us denote  $\Theta^{\hat \tau} \circ \bb E_b \circ \int$, $\Theta^{\hat \mu} \circ \bb E_b \circ \int$ by $\Vol^{\hat \tau}$, $\Vol^{\hat \mu}$. Relative to any chosen reduced cross-section $\rcsn$,  the fiber $\mdl X_{\rcsn(1)}$ gives the motivic Milnor fiber $\mathscr S_f$ as constructed in \cite[\S~8.5]{hru:loe:lef}, that is, $\Vol^{\hat \mu}([\mdl X_{\rcsn(1)}]) = \mathscr S_f$, but for any $v \in \K^\times$ other than $1$, $\mathscr S_f^v \coloneqq \Vol^{\hat \mu}([\mdl X_{v\rcsn(1)}])$ is not equal to $\mathscr S_f$ in general. The $\hat \mu$-action on $\mathscr S_f^v$ corresponds to a coset of $\hat \mu$ in $\hat \tau$, which in turn corresponds to the various reduced cross-sections $\rcsn' : \Q \fun \RV$ with $\rcsn'(1) = v \rcsn(1)$.

These constructions  still go through if we replace  $\bb S = \C \dpar t$ with $\bb S = \C \dpar{t^q}$ for any $q \in \Q^+$.


\subsection{Categories with angular components} \label{cat:ang}

Ultimately, we are only interested in the points $(\alpha_\imath, \beta_\imath) \in (\Q^+)^2$ as described above and the corresponding sequences $(m_2 / m_\imath, m_i / m_\imath)_{2 \leq i \leq \imath}$ of elements in $(0, 1] \sub \Q^+$.  But it is conceptually clearer to work in a more general setting. Thus, let $\vta = (\vta_1, \ldots, \vta_\ell)$ be a sequence of elements in $(0, 1] \sub \Q^+$ with $\vta_\ell = 1$ such that if $\ell > 1$ then $\vta_1 = \vta_2 < \ldots < \vta_\ell$; let $\lambda$ be the least positive integer such that every $\lambda \vta_i$ is an integer; so any other integer that has this property must be a multiple of $\lambda$. We take  $\vta_\ell = 1$ so to simplify the discussion, and can derive other cases by symmetry, see Remark~\ref{matchG}.

\begin{rem}
Similar to  $\Vol^{\hat \tau}$, we aim to construct  homomorphisms $\Vol^{\ac}_{\vta} = \Theta^{\ac}_{\vta} \circ \bb E^{\ac}_{b,\vta} \circ \int^{\ac}_\vta$, which is written as $\Vol^{\ac} = \Theta^{\ac} \circ \bb E^{\ac}_{b} \circ \int^{\ac}$ when $\ell = 1$ (hence $\vta = 1$), and show that  they commute with various convolution operators as illustrated in (\ref{convol:comm}).
\begin{figure}[htb]
\begin{equation}\label{convol:comm}
\bfig
 \square(0,0)|amlb|/<-`->`->`<-/<-975, 400>[{\ggk \RV^{\ac}_{\vta}[*]/ (\bm P - 1)}`{\ggk \VF^{\ac}_\vta}`{\ggk \RV^{\ac}[*]/ (\bm P - 1)}`{\ggk \VF^{\ac}}; \int^{\ac}_\vta`\dot \Pi_\vta`\dot \Sigma_\vta`\int^{\ac}]
 \hSquares(0,0)/->`->``->`->`->`->/<400>[{\ggk \RV^{\ac}_{\vta}[*]/ (\bm P - 1)}`{\sggk \RES^{\ac}_{\vta}}`{\ggk^{\vta} \Var_{\C} \cong \ggk^{\hat \mu} \Var_{\G_m}}`{\ggk \RV^{\ac}[*]/ (\bm P - 1)}`{\sggk \RES^{\ac}}`{\ggk^{\bm 1} \Var_{\C} \cong \gmv}; \bb E^{\ac}_{b,\vta}`\Theta^{\ac}_{\vta}``\dot \Pi_\vta`\dot \Psi_\vta`\bb E^{\ac}_{b}`\Theta^{\ac}]
\efig
\end{equation}
\end{figure}

This and the next subsections explain  the middle and right squares of the diagram. The construction of $\int^{\ac}_\vta$  is given in \S~\ref{const:intac}.
\end{rem}

First of all, to define convolution operators, we need to consider objects equipped with angular component maps and equivariant morphisms between them, as follows.

Let $Y$ be a variety over $\C$ with a $\G_m$-action and $\pi : Y \fun \G_m^\ell$ a morphism. For each $1 \leq i \leq \ell$, write $\pi_i = \pr_i \circ \pi$, $\pi_{>i} = \pr_{>i} \circ \pi$, etc. Suppose that, for some morphism $\pi^* : Y \fun \G_m$,
\[
\pi_i = (\pi^{*})^{\lambda \vta_i} \quad \text{for every $2 \leq i \leq \ell$}.
\]
For  $n \in \Z^+$ that is divisible by $\lambda$, we say that  $\pi$ is \memph{$(\vta, n)$-diagonal} if,
\begin{equation}\label{good:act}
\pi_1(c \cdot y) = c^{n \vta_1} \pi_1(y), \quad \pi^*(c \cdot y) = c^{n /\lambda} \pi^*(y), \quad \text{for all $c  \in  \G_m$ and all $y \in Y$}.
\end{equation}
This implies that $\pi(c \cdot y) = c^{n \vta} \pi(y)$, where $c^{n \vta} = (c^{n \vta_i})_{1 \leq i \leq \ell}$; we also refer to  $\pi$ as a \memph{variety over $\G_m^\ell$ with a $(\vta, n)$-diagonal $\G_m$-action}.

Recall that, for any functions $\phi : X \fun A$ and $\psi : Y \fun B$, we denote by  $\phi \oplus \psi$ the function $X \cap Y \fun A \times B$ given by $x \efun (\phi(x), \psi(x))$. Also, if $A$, $B$ are subsets of a group $(G, +)$ then $\phi + \psi$ is the function $X \cap Y \fun G$ given by $x \efun \phi(x) + \psi(x)$.

\begin{defn}\label{theta:C:cat}
An object of the category $\Var_{\C}^{\vta, n}$ is a variety  $\pi_1 \oplus \pi^* : Y \fun \G_m^2$ over $\G_m^2$ with a $\G_m$-action such that $\pi : Y \fun \G_m^\ell$ is  $(\vta, n)$-diagonal. A morphism between two such objects  is a morphism between the  varieties over $\G_m^2$ that is equivariant with respect to the $\G_m$-actions.

The category $\Var_{\C}^{\vta}$ is the colimit of the inductive system of  $\Var_{\C}^{\vta, n}$, $n \in \Z^+$ (the transition functors are  given as in Definition~\ref{defn:weighted}).
\end{defn}

For convenience, an object of $\Var_{\C}^{\vta, n}$ will sometimes be referred to as $(Y, \pi)$ or just $\pi$, with $\pi^*$ implicit, even though $\pi^*$ cannot be recovered from $\pi$ alone.

The Grothendieck groups $\ggk^{\vta, n} \Var_{\C}$ are constructed as usual. Since fiber product (reduced) over $\G_m^2$ with diagonal action induces a product operation,   $\ggk^{\vta, n} \Var_{\C}$ is a commutative ring. Set
\[
\ggk^{\vta} \Var_{\C} = \colim n \ggk^{\vta, n} \Var_{\C}.
\]


Observe that if $\ell = 1, 2$ then $\vta$ does not really have any bearing on the definitions of $(\vta, n)$-diagonality and the category $\Var_{\C}^{\vta, n}$. So,  in that case, we  shall just say \memph{$n$-diagonal} and write $\Var_{\C}^{\bm 1, n}$, etc., when $\ell = 1$ and $\Var_{\C}^{\bm 2, n}$, etc., when $\ell = 2$. The difference between $\Var_{\C}^{\bm 1, n}$ and $\Var_{\C}^{\tau_n}$ is  that, in the latter, witnesses for $n$-weightedness are implicit and hence morphisms are not required to respect them.

\begin{rem}\label{just:muhat}
Actually,  $\Var_{\C}^{\bm 1, n}$ is just the category $\Var_{\G_m}^{\G_m,n}$ as defined in \cite[\S~2.3]{guibert2006} and, by \cite[Lemma~2.5]{guibert2006}, it is equivalent to the category of varieties over $\C$ with $\mu_{n}$-actions, in particular, $\ggk^{\bm 1, n} \Var_{\C} \cong \ggk^{\mu_{n}} \Var_{\C}$ for all $n$ and hence we have an isomorphism
\begin{equation}\label{fib:at:1}
\Upsilon: \ggk^{\bm 1} \Var_{\C} \fun \ggk^{\hat \mu} \Var_{\C}.
\end{equation}

Denote by $\Var^{\mu_n}_{\G_m}$ the category of varieties $\xi : Z \fun \G_m$ over $\G_m$ with $\mu_n$-actions such that its fibers  are  $\mu_n$-invariant, and by $\Var^{\G_m, n}_{\G_m^2}$ the category of varieties $\xi : Z \fun \G_m^2$ over $\G_m^2$ with $\G_m$-actions such that the fibers of $\xi_1$ are  $\G_m$-invariant and $\xi_2$ is $n$-diagonal. The morphisms in both categories are those that are equivariant with respect to the group in question. By \cite[Lemma~2.5]{guibert2006} again, these two categories are equivalent.

Assume $\ell > 1$. For $(Y, \pi) \in \Var_{\C}^{\vta, n}$, let $\bar \pi$ be the morphism
\[
Y \to^{\pi_{\leq 2}} \G_m^2  \to^{(x,y) \efun xy^{-1}} \G_m
\]
By (\ref{good:act}), every fiber of $\bar \pi$ inherits a $\G_m$-action from $Y$ and hence $\bar \pi \oplus \pi^*$ is an object of  $\Var^{\G_m, n/\lambda}_{\G_m^2}$.  Conversely, for each object $\xi : Z \fun \G_m^2$  of  $\Var^{\G_m, n/\lambda}_{\G_m^2}$,  the morphism $\bar \xi$ on $Z$ given by $z \efun (\xi_1(z)\xi_2(z)^{\lambda \vta_1}, \xi_2(z))$ is an object of $\Var_{\C}^{\vta, n}$. These two assignments can be extended to functors that are inverse to each other.  So,
\[
\ggk^{\vta, n} \Var_{\C} \cong \ggk^{\G_m, n/\lambda} \Var_{\G_m^2} \cong \ggk^{\mu_{n/\lambda}} \Var_{\G_m}.
\]
Consequently, all the work  on $\ggk^{\vta} \Var_{\C}$  below may be considered as done on
\[
\ggk^{\hat \mu} \Var_{\G_m} \cong \colim n \ggk^{\mu_{n/\lambda}} \Var_{\G_m},
\]
where $\Var^{\hat \mu}_{\G_m}$ is the category of varieties over $\G_m$ whose fibers  are  $\hat \mu$-invariant with uniformly good $\hat \mu$-actions, in other words, $\Var^{\hat \mu}_{\G_m}$ is the colimit of the inductive system of $\Var^{\mu_n}_{\G_m}$, $n \in \Z^+$.

As in \cite{guibert2006}, the point here is that $\ggk^{\hat \mu} \Var_{\G_m}$ is much closer to objects that have been studied extensively in the literature and  hence for which we have a deeper understanding.
\end{rem}

If $Z \in \Var_{\C}^{\hat \tau}$ and $(Y, \pi) \in \Var_{\C}^{\vta, n}$ then $(Y \times Z, \pi \circ \pr_Y) \in \Var_{\C}^{\vta, n}$, where the $\G_m$-action on $Y \times Z$ is given by $c \cdot (y, z) = (c \cdot y, c^n \cdot z)$ for all $c \in \G_m$. This is compatible with the inductive systems in question  (the action $c^n \cdot z$ on the $Z$-factor is forced for this reason) and hence, after passing to the colimits, we see that  $\ggk^\vta \Var_{\C}$ is indeed a $\ggk^{\hat \tau} \Var_{\C}$-module.

\begin{defn}\label{C:cat:conv}
Assume $\ell > 1$. We construct a $\ggk^{\hat \tau} \Var_{\C}$-module homomorphism
\[
\Psi_\vta : \ggk^{\vta} \Var_{\C} \fun \ggk^{\bm 1} \Var_{\C}
\]
by induction on $\ell$ as follows.

For $(Y, \pi) \in \Var_{\C}^{\vta, n}$, let $(\pi_1 + \pi_2)^{-1}(0)$ and $Y \mi (\pi_1 + \pi_2)^{-1}(0)$ denote, in $\Var_{\C}$, the pullbacks of $\pi_{\leq 2} : Y \fun \G_m^2$ along the antidiagonal of $\G_m^2$ and its complement, respectively. By (\ref{good:act}), both varieties inherit a $\G_m$-action from $Y$.

For the base case $\ell =2$, we consider the $\G_m$-action on $(\pi_1 + \pi_2)^{-1}(0) \times \G_m$ whose second factor is given by $c \cdot z =  c^nz$. Then the expressions
\begin{equation}\label{conv:1:2}
[(Y \mi (\pi_1 + \pi_2)^{-1}(0), \pi_1 + \pi_2)], \quad [ ((\pi_1 + \pi_2)^{-1}(0) \times \G_m, \pr_{\G_m})]
\end{equation}
designate two elements in $\ggk^{\bm 1, n}\Var_{\C}$; they only depend on the class of $(Y, \pi)$ and hence may be denoted by $\dot \Psi_{\bm 2, n} ([(Y, \pi)])$, $\ddot \Psi_{\bm 2, n} ([(Y, \pi)])$, respectively. These assignments respect the defining relations of $\ggk^{\bm 2, n} \Var_{\C}$ and hence may be extended uniquely to two group homomorphisms $\dot \Psi_{\bm 2, n}$, $\ddot \Psi_{\bm 2, n}$. These group homomorphisms in turn are compatible with the inductive systems in question and hence, after passing to the colimits, we obtain two group homomorphisms  $\dot \Psi_{\bm 2}$, $\ddot \Psi_{\bm 2}$, which also respect the $\ggk^{\hat \tau} \Var_{\C}$-module structure. Set $\Psi_{\bm 2} = - (\dot \Psi_{\bm 2} - \ddot \Psi_{\bm 2})$.

For the inductive step $\ell > 2$, let $\vta' = (\vta_3, \vta_3, \vta_4, \ldots, \vta_\ell)$ and $\lambda'$ be the least positive integer such that $\lambda' \vta_i$ is an integer for every $3 \leq i \leq \ell$. Then $\lambda'$ divides $\lambda$. We consider the $\G_m$-action on $\G_m \times (\pi_1 + \pi_2)^{-1}(0)$  whose first factor is given by $c \cdot z = c^{n \vta_3}z$. Then the expression
\begin{equation}\label{conv:induc}
(Y^{\vta'}, \pi^{\vta'}) \coloneqq (\G_m \times (\pi_1 + \pi_2)^{-1}(0),  1_{\G_m} \times (\pi^{*})^{\lambda/\lambda'})
\end{equation}
designates an object of $\Var_{\C}^{\vta', n}$ whose class only depends on that of $(Y, \pi)$ and hence may be denoted by $\Psi_{\vta, n}^{\vta'}([(Y, \pi)])$. The assignments $\Psi_{\vta, n}^{\vta'}$, $n \in \Z^+$, may be extended to group homomorphisms  and their colimit $\Psi_{\vta}^{\vta'} : \ggk^{\vta} \Var_{\C} \fun \ggk^{\vta'} \Var_{\C}$ is a $\ggk^{\hat \tau} \Var_{\C}$-module homomorphism. Now we set
\[
\dot \Psi_\vta = \dot \Psi_{\vta'} \circ \Psi_{\vta}^{\vta'}, \quad \ddot \Psi_\vta = \ddot \Psi_{\vta'} \circ \Psi_{\vta}^{\vta'}, \quad \Psi_\vta = \Psi_{\vta'} \circ \Psi_{\vta}^{\vta'} = -(\dot \Psi_\vta - \ddot \Psi_\vta).
\]
\end{defn}

We could have defined $\Psi_\vta$ to be $\dot \Psi_\vta - \ddot \Psi_\vta$ instead of $-(\dot \Psi_\vta - \ddot \Psi_\vta)$. The negative sign at the front is inherited from the literature.

\begin{rem}
The case $\ell =2$ is special. For $(X, \pi_X) \in \Var_{\C}^{\bm 1, m}$ and $(Y, \pi_Y) \in \Var_{\C}^{\bm 1, n}$, let $\pi_X \times \pi_Y$ be the obvious morphism  $X \times Y \fun \G_m^{2}$. Then $(X \times Y, \pi_X \times \pi_Y)$ is an object of $\Var_{\C}^{\bm 2, mn}$ whose class only depends on those of $(X, \pi_X)$ and $(Y, \pi_Y)$. We  define a binary map on $\ggk^{\bm 1} \Var_{\C}$ by
\begin{equation}\label{C:convol}
  [(X, \pi_X)] * [(Y, \pi_Y)] = \Psi_{\bm 2}([(X \times Y, \pi_X \times \pi_Y)]) \in \ggk^{\bm 1} \Var_{\C}.
\end{equation}
Although the category $\Var_{\C}^{\bm 2}$ is not the same one used in \cite[\S~5.1]{guibert2006}, the proof of \cite[Proposition~5.2]{guibert2006} still goes through verbatim, which justifies referring to (\ref{C:convol}) as a convolution product (it is commutative,  associative, and unital).
\end{rem}

Recall that definability is construed in  $\puC$ with $\bb S = \C \cup \Q$, unless additional parameters are used.

\begin{defn}\label{resac}
An object of the category $\RV^{\ac}_{\vta}[k]$ is a definable triple of the form
\[
(U, f :  U \fun \RV^k, g = \ac_1 \oplus \ac^* : U \fun \vta_1^\sharp \times (1 / \lambda)^\sharp)
\]
such that, for every $r \in \ran(g)$, the pair $(g^{-1}(r), f \rest g^{-1}(r))$ is an object of $\RV[k]$, in other words, $f \rest g^{-1}(r)$ is finite-to-one (the category $\RV[k]$ here is of course formulated relative to the additional parameters $r$, that is, $\bb S$ is in effect  the subgroup of $\RV$ generated by $r$ over $\K^\times$). The function
\[
\ac = \ac_1 \oplus \bigoplus_{2 \leq i \leq \ell} (\ac^*)^{\lambda \vta_i}: U \fun \vta^\sharp
\]
is referred to as an \memph{angular component map} on $\bm U = (U, f)$. A definable bijection $F : U \fun V$  is a morphism between two such objects $(\bm U, g_{U})$, $(\bm V, g_{V})$ if $g_{U} = g_{V} \circ F$. Set $\RV^{\ac}_{\vta}[*] = \coprod_k \RV^{\ac}_{\vta}[k]$.

The category $\RES^{\ac}_{\vta}$ is formulated in the same way, but with $\RV[k]$ replaced by $\RES$.
\end{defn}

\begin{rem}\label{ac:cat:sub}
The assignment $(U, f, g) \efun (U, f \oplus g)$ induces a functor $\RV^{\ac}_{\vta}[*] \fun \RV[*]$ that identifies $\RV^{\ac}_{\vta}[*]$ with a subcategory of $\RV[*]$; similarly for $\RES^{\ac}_{\vta}$, $\RES$.
\end{rem}

As is the case with $\Var_{\C}^{\vta, n}$, we shall sometimes refer to  an object of $\RV^{\ac}_{\vta}[k]$ as $(\bm U, \ac)$, with $\ac^*$ implicit. For each $1 \leq i \leq \ell$, let $\ac_i = \pr_i \circ \ac$, $\ac_{>i} = \pr_{>i} \circ \ac$, etc.

The ring structure of $\ggk \RV^{\ac}_\vta[*]$ is induced by fiberwise disjoint union and fiberwise cartesian product in $\RV^{\ac}_\vta[*]$; similarly for other such categories. We also think of $\ggk \RV^{\ac}_\vta[*]$ as a $\ggk \RV[*]$-module and $\sggk \RES^{\ac}_{\vta}$ as a $\sggk^{\hat \tau} \RES$-module (the extra defining condition for ``$\sggk$'' in $\sggk \RES^{\ac}_{\vta}$ is in effect  imposed fiberwise).

If $\ell = 1$ (hence $\vta = 1$) then the subscript $\vta$ shall be dropped from the notation.

\begin{defn}\label{conv:RV}
Assume $\ell > 1$. We construct a $\ggk \RV[*]$-module homomorphism
\[
\Pi_\vta : \ggk \RV^{\ac}_{\vta}[*] \fun \ggk \RV^{\ac}[*]
\]
by induction on $\ell$, as follows.

Let $(\bm U, \ac) = (U, f, \ac_1 \oplus \ac^*) \in \RV^{\ac}_{\vta}[k]$.  Let $U'$ be the subset of $U$  determined by the antidiagonal condition $\ac_1(u) = - \ac_2 (u)$.

For the base case $\ell =2$, that is, $\vta = (1, 1)$, let $f_1 : U \fun \RV^{k+1}$ be the function given by $u \efun (f(u), \ac_1(u))$, similarly for $f_2$. The pairs $(U, f_1)$, $(U \mi U', f_1)$, $(U', f_1)$ are more suggestively denoted, respectively, by
\[
\bm U_1, \quad \bm U_1   \mi (\ac_1 + \ac_2)^{-1}(0), \quad (\ac_1 + \ac_2)^{-1}(0).
\]
The elements $\dot \Pi_{(1,1)}([(\bm U, \ac)])$, $\ddot \Pi_{(1,1)}([(\bm U, \ac)])$ in $\ggk \RV^{\ac}[k{+}1]$ are then given, respectively, by
\begin{equation}\label{RV:conv}
[(\bm U_1   \mi (\ac_1 + \ac_2)^{-1}(0), \ac_1 + \ac_2)], \quad [ ((\ac_1 + \ac_2)^{-1}(0) \times 1^\sharp, \pr_{1^\sharp})];
\end{equation}
here the second term is such that each fiber of $\pr_{1^\sharp}$ is a copy of $(\ac_1 + \ac_2)^{-1}(0)$, and hence is indeed an element in $\ggk \RV^{\ac}[k{+}1]$. These two assignments do not depend on the representative $(\bm U, \ac)$ or the choice between $f_1$ and $f_2$, and hence may be extended uniquely to two $\ggk \RV[*]$-module homomorphisms  $\dot \Pi_{(1,1)}$, $\ddot \Pi_{(1,1)}$ (the gradation has been shifted by $1$). Then set
$\Pi_{(1,1)} = - (\dot \Pi_{(1,1)} - \ddot \Pi_{(1,1)})$.

For the inductive step $\ell > 2$, let $\vta'$ and $\lambda'$ be as in Definition~\ref{C:cat:conv}. For every $r \in (1 / \lambda')^\sharp$, since there are only finitely many $(r_1, r_2) \in \vta_1^\sharp \times (1 / \lambda)^\sharp$ with $r_1 = - r_2^{\lambda \vta_2}$ and $r_2^{\lambda/\lambda'} = r$, the restriction  $f \rest (U' \cap ((\ac^{*})^{\lambda/\lambda'})^{-1}(r))$ is finite-to-one. So the triple
\begin{equation}\label{ind:step:con}
(\bm U^{\vta'}, \ac^{\vta'}) \coloneqq (\vta_3^\sharp \times U', f \circ \pr_{U'}, \id \times (\ac^{*})^{\lambda/\lambda'})
\end{equation}
designates an object of $\RV^{\ac}_{\vta'}[k]$ whose class only depends on that of $(\bm U, \ac)$ and the assignment $[(\bm U, \ac)] \efun [(\bm U^{\vta'}, \ac^{\vta'})]$ determines a $\ggk \RV[*]$-module homomorphism
\[
\Pi^{\vta'}_{\vta}: \ggk \RV^{\ac}_{\vta}[*] \fun \ggk \RV^{\ac}_{\vta'}[*].
\]
Thus, we may set
\[
\dot \Pi_\vta = \dot \Pi_{\vta'} \circ \Pi_{\vta}^{\vta'}, \quad \ddot \Pi_\vta = \ddot \Pi_{\vta'} \circ \Pi_{\vta}^{\vta'}, \quad \Pi_\vta = \Pi_{\vta'} \circ \Pi_{\vta}^{\vta'} = - (\dot \Pi_\vta - \ddot \Pi_\vta).
\]

There is a similar construction resulting in a $\sggk^{\hat \tau} \RES$-module homomorphism
\[
\sggk \RES^{\ac}_{\vta} \fun \sggk \RES^{\ac},
\]
which is denoted by $\Pi_\vta$ as well. Its construction is actually simpler since the categories $\RES^{\ac}_{\vta}$, $\RES^{\ac}$ are not graded and the function $f$ is irrelevant. Moreover, in light of the ring homomorphism $\bb E_b : \ggk \RV[*] \fun \sggk^{\hat \tau} \RES$, this $\Pi_\vta$ may be viewed as a $\ggk \RV[*]$-module homomorphism too.
\end{defn}

For $(\bm U, \ac_{U}) \in \RV^{\ac}[k]$ and $(\bm V, \ac_{V}) \in \RV^{\ac}[l]$, let $\ac_{U} \times \ac_{V}$ be the obvious function from  $U \times V$ into $(1,1)^{\sharp}$. Then the class
\[
[(\bm U \times \bm V, \ac_{U} \times \ac_{V})] \in \ggk \RV^{\ac}_{(1,1)}[k{+}l]
\]
only depends on the classes $[(\bm U, \ac_{U})]$, $[(\bm V, \ac_{V})]$, not their representatives. Set
\[
[(\bm U, \ac_{U})] * [(\bm V, \ac_{V})] = \Pi_{(1,1)}([(\bm U \times \bm V, \ac_{U} \times \ac_{V})]) \in \ggk \RV^{\ac}[k{+}l{+}1],
\]
which may be thought of as a convolution product of the two classes.

The following lemma only serves to confirm the structural resemblance of the binary map $*$ here to the convolution product (\ref{C:convol}). It will not be of any use beyond this point.

\begin{lem}
The binary map $*$ on $\ggk \RV^{\ac}[*]$ is commutative and associative. Moreover, for all $(\bm U, \ac) = (U, f, \ac) \in \RV^{\ac}[k]$,
\[
[(\bm U, \ac)] * 1 = [(\bm U, \ac)][1] \in \ggk \RV^{\ac}[k{+}1].
\]
\end{lem}

Here $1 \in \ggk \RV^{\ac}[0]$ is the multiplicative identity in $\ggk \RV^{\ac}[*]$ and is represented by the triple $(1^\sharp, \infty, \id)$, and $[1] \in \ggk \RV^{\ac}[1]$ is represented by the triple $(1^\sharp, \id, \id)$.

\begin{proof}
The formal computations involved  are essentially the same as those in the proof of \cite[Proposition 5.2]{guibert2006}. We shall just write down some details for the second claim since the expected convolution identity $1$ is actually off by a factor, namely $[1]$, in this setting.

For ease of notation, the function $\ac \circ \pr_U$ on $U \times 1^\sharp$ is still denoted by $\ac$, similarly for $f$ and $\id_{1^\sharp}$. Write $(U \times 1^\sharp, f \oplus \ac)$ as $\bm U  \times 1^\sharp$. Then $[(\bm U, \ac)] * 1 \in \ggk \RV^{\ac}[k{+}1]$ is given by
\begin{equation}\label{conv:id}
- ( [(\bm U \times 1^\sharp  \mi (\ac + \id_{1^\sharp})^{-1}(0), \ac + \id_{1^\sharp})] - [ ((\ac + \id_{1^\sharp})^{-1}(0) \times 1^\sharp, \pr_{1^\sharp})]).
\end{equation}
In this expression, for each $r \in 1^\sharp$, we have, in $\ggk \RV[k{+}1]$,
\[
  [(\ac + \id_{1^\sharp})^{-1}(r)]  = [(U \mi \ac^{-1}(r), f \oplus \ac)] \dand
  [\pr_{1^\sharp}^{-1}(r)]  = [(U, f \oplus \ac)].
\]
So (\ref{conv:id}) may also be written as
\[
-([(\bm U \times 1^\sharp  \mi (\ac - \id_{1^\sharp})^{-1}(0), \pr_{1^\sharp})] - [(\bm U \times 1^\sharp, \pr_{1^\sharp})]).
\]
Of course this is just $[((\ac - \id_{1^\sharp})^{-1}(0), \pr_{1^\sharp})]$. In $\ggk \RV[k{+}1]$, since we now have, for every $r \in 1^\sharp$, $[\pr_{1^\sharp}^{-1}(r)] = [(\ac^{-1}(r) \times r, f \times \id_r)]$, it follows that $[((\ac - \id_{1^\sharp})^{-1}(0), \pr_{1^\sharp})] = [(\bm U, \ac)][1]$.
\end{proof}

\subsection{Commuting with the convolution operators}

Next, we show that the middle and right squares of (\ref{convol:comm}) indeed commute.

\begin{rem}\label{ebac}
Let $(\bm U, g = \ac_1 \oplus \ac^*) \in \RV^{\ac}_{\vta}[*]$. By Proposition~\ref{red:D:iso} and compactness, there exists a definable finite partition $(B_i)_i$ of $\vta_1^\sharp \times (1 / \lambda)^\sharp$ such that every fiber $g^{-1}(r)$ is, uniformly over each $B_i$, $r$-definably bijective to a disjoint union of products $\bm U_{rij} \times D_{rij}^\sharp$ with $\bm U_{rij} \in \RES[*]$ and $D_{rij} \in \Gamma[*]$. Actually, by stable embeddedness (see Remark~\ref{pillars}),  we may write $D_{rij}$ as $D_{ij}$ since it must be the case that $D_{rij} = D_{r'ij}$ for any other $r' \in B_i$; similarly for $\vrv(\bm U_{rij})$, which may be assumed to be a singleton. Let $\bm U_{ij} = \bigcup_{r \in B_i} \bm U_{rij} \times r$ and $g_{ij}: \bm U_{ij} \fun B_i$ be the obvious coordinate projection.
\end{rem}

\begin{rem}\label{theta:EB}
Keeping the notation of Remark~\ref{ebac}, we see that, over each $B_i$, there are  elements $[(V_{i}, g_{i})]$, $[(V_{i'}, g_{i'})]$ in $\sggk \RES^{\ac}_{\vta}$, depending on the choice of $\bm U_{ij}$ and $D_{ij}$, such that
\begin{equation*}
\bb E_b([g^{-1}(r)]) = [g_{i}^{-1}(r)] - [g_{i'}^{-1}(r)].
\end{equation*}
This equality is construed in $\sggk^{\hat \mu} \RES$ with, say, $\bb S = \C \dpar{t^{1 / \lambda}}$, but in general not with $\bb S = \C \dpar{t}$, unless  the isomorphism $\Theta^{\hat \mu}$ is applied on both sides, in which case the difference between $\C \dpar{t^{1 / \lambda}}$ and $\C \dpar{t}$ is neutralized and the equality happens in $\ggk^{\hat \mu} \Var_{\C}$. Thus, $[(V_{i}, g_{i})] - [(V_{i'}, g_{i'})]$ does not depend on the choice of $\bm U_{ij}$ and $D_{ij}$. Setting
\begin{equation}\label{fib:E}
[(\bm U, g)] \efun \sum_{i} ([(V_{i}, g_{i})] - [(V_{i'}, g_{i'})])
\end{equation}
yields a ring homomorphism
\[
\bb E^{\ac}_{b,\vta} : \ggk \RV^{\ac}_{\vta}[*] \fun \sggk \RES^{\ac}_{\vta},
\]
which is also an $\ggk \RV[*]$-module homomorphism.

Note that $\bb E^{\ac}_{b,\vta}([(\bm U, g)])$ may be understood as a ``function'' into $\vta_1^\sharp \times (1 / \lambda)^\sharp$ whose ``fibers'' are of the form $\bb E_b([g^{-1}(r)])$, which has nothing to do with the partition $(B_i)_i$, but it takes extra work to make sense of what ``function'' and ``fibers'' mean here, and the point of the partition $(B_i)_i$ is just to show the existence of such a finite sum as in (\ref{fib:E}) and thereby that $\bb E^{\ac}_{b,\vta}$ is well-defined.

Taking quotient by $(\bm P - 1)$ fiberwise, we see that the corresponding ideal of $\ggk \RV^{\ac}_{\vta}[*]$, still denoted by $(\bm P - 1)$, vanishes along $\bb E^{\ac}_{b,\vta}$.
\end{rem}

The map $\Pi_\vta$ on $\ggk \RV^{\ac}_{\vta}[*]$ is indeed related to the map $\Pi_\vta$ on $\sggk \RES^{\ac}_{\vta}$ via $\bb E^{\ac}_{b,\vta}$ as indicated in the middle square of (\ref{convol:comm}):

\begin{lem}\label{RV:RES:conv}
As $\ggk \RV[*]$-module homomorphisms, $\bb E^{\ac}_{b} \circ \dot \Pi_\vta = \dot \Pi_\vta \circ \bb E^{\ac}_{b,\vta}$, similarly for $\ddot \Pi_\vta$ and hence for $\Pi_\vta$.
\end{lem}
\begin{proof}
Although the case of $\Pi_\vta$ follows immediately from those of $\dot \Pi_\vta$ and $\ddot \Pi_\vta$, we shall show this for $\Pi_\vta$ directly using the same argument. To that end,  it is enough to consider elements in $\ggk \RV^{\ac}_{\vta}[*]$ of the form $[(\bm U, g)]$, since the general case would follow from $\ggk \RV[*]$-linearity.

We do induction on $\ell$. For the base case $\ell = 2$, that is, $\vta = (1,1)$, by Remark~\ref{ebac}, using the notation there, we have
\[
 [(\bm U, g)] =  \sum_{ij} [(\bm U_{ij}, g_{ij})][(D_{ij}^\sharp \times B_i, \pr_{B_i})].
\]
Since $\Pi_{(1,1)}$ is a $\ggk \RV[*]$-module homomorphism, it follows that
\[
\Pi_{(1,1)}([(\bm U, g)] ) =   \sum_{ij} [D_{ij}^\sharp]\Pi_{(1,1)}( [(\bm U_{ij}, g_{ij})]).
\]
Let $n_{ij} = \chi_b(D_{ij})$. The gradation is forgotten by $\bb E^{\ac}_{b}$, so we have
\[
(\bb E^{\ac}_{b} \circ \Pi_{(1,1)})([(\bm U, g)] ) =  \sum_{ij} n_{ij} \Pi_{(1,1)}( [(U_{ij}, g_{ij})]) ,
\]
where $(U_{ij}, g_{ij})$ stands for the obvious object of $\RES^{\ac}_{(1,1)}$ in relation to $(\bm U_{ij}, g_{ij})$. The right-hand side of this equality also equals $(\Pi_{(1,1)} \circ \bb E^{\ac}_{b,(1,1)})([(\bm U, g)] )$.

For the inductive step $\ell > 2$, let $\vta'$ be as in Definition~\ref{conv:RV}. Then a similar computation  shows that, for all $ij$,
\begin{equation}\label{conv:commu}
(\bb E^{\ac}_{b,\vta'} \circ \Pi_\vta^{\vta'})([(\bm U_{ij}, g_{ij})]) = (\Pi_\vta^{\vta'} \circ \bb E^{\ac}_{b,\vta})([(\bm U_{ij}, g_{ij})]).
\end{equation}
So the desired equality follows from the definition of $\Pi_\vta$ and the inductive hypothesis.
\end{proof}

The $\hat \tau$-action on an object $(U, \ac)\in \RES^{\ac}_{\vta}$ must factor through some $\tau_n$ such that $n  / \lambda$ and hence every $n \vta_i$ are  integers. It may also be interpreted as a $\G_m$-action subject to the condition
\[
\ac(c \cdot u)= \rv(c^{n \vta}) \ac(u), \quad \text{for all $u \in U$ and all $c \in \G_m$}.
\]
Let $\RES^{\ac}_{\vta, n}$ be the full subcategory of $\RES^{\ac}_{\vta}$ of those objects for which this condition holds. Thus, we have obtained an inductive system of categories $\RES^{\ac}_{\vta, n}$ such that
\[
\RES^{\ac}_{\vta} = \colim n \RES^{\ac}_{\vta, n} \dand \sggk \RES^{\ac}_{\vta} = \colim n \sggk \RES^{\ac}_{\vta, n};
\]

\begin{lem}\label{thetaac}
For each $n$ there is a ring isomorphism
\[
\Theta^{\ac}_{\vta, n} : \sggk \RES^{\ac}_{\vta, n} \fun \ggk^{\vta,n} \Var_{\C},
\]
determined by the assignment $[(V, g)] \efun [\tbk(V, g)]$ for $\vrv(V)$ a singleton,  and hence a ring isomorphism
\[
\Theta^{\ac}_{\vta} : \sggk \RES^{\ac}_{\vta} \fun \ggk^{\vta} \Var_{\C}.
\]
Moreover, under the ring homomorphism $\Theta^{\hat \tau} \circ \bb E_b$, we have $\Theta^{\ac} \circ \dot \Pi_\vta = \dot \Psi_\vta \circ \Theta^{\ac}_{\vta}$ as $\ggk \RV[*]$-module homomorphisms, similarly for $\ddot \Pi_\vta$, $\ddot \Psi_\vta$ and hence for $\Pi_\vta$, $\Psi_\vta$.
\end{lem}

This shows that the right square of (\ref{convol:comm}) commutes.

Note that the set $\tbk(V, g)$ is definable without using the implicit reduced cross-section $\rcsn$; in other words, varying $\rcsn$ will not change $\tbk(V, g)$, but does change the bijection in question, and that is why $\tbk(V, g)$ inherits the $\hat \tau$-action on $(V, g)$.

\begin{proof}
The situation here is very similar to that in \cite[\S~4.3]{hru:loe:lef} or in Remarks~\ref{const:theta} and \ref{theta:suj}, so we shall be brief. If $\vrv(V)$ is a singleton then the graph of $\tbk(g)$ is just a constructible set, in fact uniformly so fiberwise. So the assignment induces a homomorphism $\Theta^{\ac}_{\vta, n}$ at the semiring level and hence at the ring level. If $\tbk(V, g)$ and $\tbk(V', g')$ are isomorphic in $\Var_{\C}^{\vta,n}$ then the isomorphism may be twisted to one between $(V, g)$ and $(V', g')$ in $\RES^{\ac}_{\vta, n}$. So $\Theta^{\ac}_{\vta, n}$ is injective. On the other hand, since objects in  $\Var_{\C}^{\vta,n}$ and $\RES^{\ac}_{\vta, n}$ are all endowed fiberwise with $\mu_{n/\lambda}$-actions via restriction, the argument for surjectivity in Remark~\ref{theta:suj} can be modified to work  for $\Theta^{\ac}_{\vta, n}$.

The second claim follows from an inductive argument, similar to the one in the proof of Lemma~\ref{RV:RES:conv}.
\end{proof}

\begin{rem}\label{matchG}
Let $\eta = (\eta_1, \ldots, \eta_\ell)$ be another sequence of elements in the interval $(0, 1] \sub \Q^+$ such that $\eta_i / \eta_\ell = \vta_i$ for every $1 \leq i \leq \ell$. Then there is a $\sigma \in \aut(\RV/ \K^\times)$ with $\sigma(\eta_\ell) = 1$ and hence $\sigma(\eta_i) = \vta_i$ for every $2 \leq i \leq \ell$. If $\sigma'$ is another such automorphism then there is a $\tau \in \aut(\RV / \kuq)$ such that $\sigma = \sigma' \circ \tau$, and hence $\sigma$, $\sigma'$ induce the same endofunctors of $\RV[*]$, $\RES$. So, in light of Remark~\ref{ac:cat:sub}, there are categories $\RV^{\ac}_{\eta}[*]$, $\RES^{\ac}_{\eta}$ and  canonical isomorphisms, both denoted by $\Delta_{\eta}$ for simplicity, that fit in the following commutative diagram:
\begin{equation*}
\bfig
\Square(0,0)/->`->`->`->/<400>[{\ggk \RV^{\ac}_{\eta}[*]}`{\sggk \RES^{\ac}_{\eta}}`{\ggk \RV^{\ac}_{\vta}[*]}`{\sggk \RES^{\ac}_{\vta}}; {\bb E_{b,\eta}^{\ac}}`\Delta_{\eta}`\Delta_{\eta}`{\bb E^{\ac}_{b, \vta}}]
\efig
\end{equation*}
This means that, in particular, we do not have to restrict the discussion, both above and below, to the case $\vta_\ell = 1$ (we have done so above for simplicity).
\end{rem}

\subsection{Decomposing the composite Milnor fiber}\label{const:intac}
%



\begin{defn}\label{defn:multi:ac:int}
An object of the category $\VF_\vta^{\ac}$ is a definable function of the form  $f = f_1 \oplus f^* : A \fun \vta_1^{\sharp\sharp} \times (1 / \lambda)^{\sharp\sharp}$  such that, for all $r \in \vta_1^\sharp \times (1 / \lambda)^\sharp$,
\begin{itemize}[leftmargin=*]
  \item\label{lam:full:fib}  if $f(A) \cap r^\sharp \neq \0$  then $r^\sharp \sub f(A)$,
  \item\label{gen:lam:int}  $\int [f^{-1}(a)] = [\bm U_{r}] /  (\bm P - 1)$ only depends on $r$, not the choice of $a \in r^\sharp$.
\end{itemize}
Note that if $\ell = 1$ then both conditions are redundant: for the first one, the function $f : A \fun \vta_1^\sharp$ must be surjective because, with $\bb S = \C \cup \Q$,  no proper nonempty subset of $\gamma^{\sharp\sharp}$ is definable for any $\gamma \in \Q^\times$ (this fact will be used several times below), and the second one is guaranteed by Lemma~\ref{rv:int:bun} (but not for functions into $\MM^\ell$ with $\ell > 1$). The function
\[
\bar f = f_1 \oplus \bigoplus_{2 \leq i \leq \ell} (f^*)^{\lambda \vta_i}: A \fun \vta^{\sharp\sharp}
\]
is referred to as an \memph{angular component map} on $A$.

If $g : B \fun \vta_1^{\sharp\sharp} \times (1 / \lambda)^{\sharp\sharp}$ is another object of $\VF^{\ac}_\vta$ then any definable bijection $F : A \fun B$ with $g \circ F  = f$ is a \memph{morphism} of $\VF^{\ac}_\vta$.
\end{defn}

\begin{rem}
Assume $\ell > 1$. Let $f$ be as above and $r \in \rv(\bar f(A)) \sub \vta^\sharp$. Then, for every $2 \leq i \leq \ell$,  by the first condition in Definition~\ref{defn:multi:ac:int}, it must be the case that $r_1^\sharp \times r_i^\sharp \sub (f_1 \oplus (f^*)^{\lambda \vta_i})(A)$. On the other hand, for all  $a, b \in \MM$ with $a^{\lambda \vta_2} = b^{\lambda \vta_2}$ and all $2 < i \leq \ell$, if $\rv(a^{\lambda \vta_i}) = \rv(b^{\lambda \vta_i})$ then  $a^{\lambda \vta_i} = b^{\lambda \vta_i}$. So $r^\sharp \cap \bar f(A)$ may be written as the graph of an $r$-definable function $r^\sharp_{\leq 2} \fun r^\sharp_{> 2}$ that only depends on $r$, not $f$, and whose projection into each $r_i^\sharp$, $2 < i \leq \ell$, is surjective.
\end{rem}

\begin{rem}\label{int:to:resac:2}
Given two  functions $\phi, \psi : A \fun \MM$ that are definable over $\bb S = \C$, we manufacture an object of $\VF^{\ac}_\vta$ from them as follows.
The trick is to replace one of them with a sufficiently large power of itself.

For $a \in \MM^2$, write  $\int [(\phi \oplus \psi)^{-1}(a)] = [\bm U_{a}] /  (\bm P - 1)$, which makes sense over the larger substructure $\bb S \la a\ra $. However, as we have just pointed out above, unlike in Lemma~\ref{rv:int:bun},  $\bm U_{a}$ does depend on the parameter $a$. Anyway, by  compactness, there is a definable finite partition $(B_i)_i$ of $\MM^2$ such that the objects $\bm U_{a}$ are defined uniformly over each $B_i$. Since each $\vv(B_i) \sub (\Q^+)^2$ is a cone based at the origin (because $\bb S = \C$), there are a $B_i$ and an $M \in \Z^+$ such that $\alpha^{\sharp\sharp} \times (0, \alpha / (M\vta_1)]^{\sharp\sharp} \sub B_i$ for all $\alpha \in \Q^+$. Let $\xi(x, y, \ldots)$ be a quantifier-free formula that defines the object $\bm U_{a}$ over $a \in B_i$. Then $M$ may be chosen to be so large that for all $a, a' \in B_i$ with $\rv(a) = \rv(a')$ and every term in $\xi(x, y, \ldots)$ of the form $\rv(F(x, y))$, where $F(x,y) \in \C[x, y]$, we have $\rv(F(a)) = \rv(F(a'))$ and hence $\bm U_{a} = \bm U_{a'}$. Therefore, for all $r  \in (\RV^{\circ\circ})^2$ with $\vrv(r_1) / \vta_1 \geq M \vrv(r_2)$, $\bm U_{a}$  does not depend on the choice of $a \in r^\sharp$. It follows that if $M$ is divisible by $\lambda$ as well, in particular, if $\lambda$ itself is sufficiently large, then the restriction of $\phi \oplus \psi^{M / \lambda}$ to the set $(\phi \oplus \psi)^{-1}(\vta_1^{\sharp\sharp} \times (1 / M)^{\sharp\sharp})$ satisfies the two conditions in Definition~\ref{defn:multi:ac:int}.

Alternatively, there are a $B_i$ and an $N \in \Z^+$ such that $\alpha^{\sharp\sharp} \times [N\alpha / (\lambda\vta_1), \infty)^{\sharp\sharp} \sub B_i$ for all $\alpha \in \Q^+$. If $N$ is sufficiently large then, for all $r = (r_1, r_2) \in (\RV^{\circ\circ})^2$ with $N \vrv(r_1) \leq  \lambda \vta_1 \vrv(r_2)$, in particular, for $r \in (\vta_1/N)^{\sharp} \times (1 / \lambda)^{\sharp}$, $\bm U_{a}$  does not depend on the choice of $a \in r^\sharp$ and hence the restriction of $\phi^N \oplus \psi$ to the set $(\phi \oplus \psi)^{-1}((\vta_1/N)^{\sharp\sharp} \times (1 / \lambda)^{\sharp\sharp})$ is as desired.
\end{rem}

The ring structure of $\ggk \VF^{\ac}_\vta$ is induced by fiberwise disjoint union and fiberwise cartesian product in $\VF^{\ac}_\vta$. It is also a $\ggk \VF_*$-module.

\begin{defn}
Let $S_\vta \sub \vta^{\sharp\sharp}$ be the subset such that $a \in S_\vta$ if and only if $\vv(\sum_i a_i) = \vta_\ell = 1$. For each $f \in \VF^{\ac}_\vta$, let $\Sigma_\vta(f) : \bar f^{-1}(S_\vta) \fun 1^{\sharp\sharp}$ be the object  of $\VF^{\ac}$ given by $x \efun \sum_i (\bar f(x))_i$. This assignment induces a   $\ggk \VF_*$-module homomorphism  $\Sigma_\vta : \ggk \VF^{\ac}_\vta \fun \ggk \VF^{\ac}$.
\end{defn}

Note that $\Sigma_\vta$ is not a ring homomorphism except when $\ell = 1$.


\begin{rem}\label{multi:ac:int}
For each   $f \in \VF^{\ac}_\vta$, let $\bm U = \bigcup_{r \in \vta_1^\sharp \times (1 / \lambda)^\sharp} \bm U_r  \times r$, where $\int [f^{-1}(a)] = [\bm U_{r}] /  (\bm P - 1)$ for $a \in r^\sharp$; the factor $r$ in $\bm U$ is an object in $\RV[0]$, and is just a bookkeeping device. Let $\ac_1 \oplus \ac^* : \bm U \fun \vta_1^\sharp \times (1 / \lambda)^\sharp$ be the obvious coordinate projection on $\bm U$. Then  $[(\bm U, \ac_1 \oplus \ac^*)]/ (\bm P - 1)$ is an element in $\ggk \RV^{\ac}_{\vta}[*] / (\bm P - 1)$  that only depends on the class $[f] \in \ggk \VF^{\ac}_\vta$ of $f$.  This indeed gives a ring and $\ggk \VF_*$-module isomorphism
\[
\int^{\ac}_\vta : \ggk \VF^{\ac}_\vta \fun \ggk \RV^{\ac}_\vta[*] / (\bm P - 1).
\]
As usual, abbreviate  $\Theta^{\ac}_\vta \circ \bb E^{\ac}_{b,\vta} \circ \int^{\ac}_\vta$ as $\Vol^{\ac}_\vta$ (or $\Vol^{\ac}$ if $\ell = 1$, per our convention).
\end{rem}

\begin{rem}\label{sigma:Galois}
By the argument in Remark~\ref{matchG}, now in terms of $\aut(\puC/ \C)$ and $\aut(\puC/ \C \cup \Q)$, and using the notation there, we may replace $\vta$ with $\eta$ without altering the structure of  the category $\VF^{\ac}_\vta$. In particular, we have the $\ggk \VF_*$-module homomorphisms $\Sigma_\eta$ and $\int^{\ac}_\eta$.
\end{rem}


\begin{nota}\label{int:gen:the}
From here on, fix a sequence $m_2 < N < m_3 < \ldots < m_\ell$ of positive integers and take $\vta_i = m_i / m_\ell$; note that $N$ is $m_1$, but its role is somewhat different and hence is denoted differently. We assume $\gcd(m_i, m_j) = 1$ for all $1 < i < j \leq \ell$; this is not really necessary, but does simplify the discussion a bit, for instance, the various values of $\lambda$ that will appear below are just $m_3, \ldots, m_\ell$.

Let $\psi \oplus \phi : X \fun \vta_{1}^{\sharp\sharp} \times (1/ m_\ell)^{\sharp\sharp}$ be an object of $\VF^{\ac}_\vta$. For $2 \leq \imath \leq \jmath \leq \ell$, let
\[
\phi_{(\imath, \jmath)} = \sum_{\imath  \leq i \leq \jmath} \phi^{m_i} : X \fun \vta_\imath^{\sharp\sharp} \dand \vta_{[\imath, \jmath)} = (\vta_{\imath}, \vta_{\imath}, \vta_{\imath+1}, \ldots, \vta_{\jmath}).
\]
Also let $\phi_{(2, 1)}$ be the zero function.

For any function $f : X \fun \VF^n$ and each $\gamma \in \Q^n$, denote the  set $f^{-1}(\gamma^{\sharp\sharp})$ by $\mdl X^\sharp_{f,\gamma}$; as usual, if $n=1$ and $\gamma = 1$ then the latter may be dropped from the notation. For any  set $A \sub X$, the restriction $f \rest (\mdl X^\sharp_{f, \gamma} \cap A)$ is just denoted by $\mdl X^\sharp_{f, \gamma} \cap A$.
\end{nota}

\begin{lem}\label{chop:lam}
For all $1 \leq \imath < \jmath \leq \ell$,  $\mdl X^\sharp_{\psi + \phi_{(2, \imath)}, \vta_{\imath+1}} \oplus \phi$ is an object of $\VF^{\ac}_{\vta_{[\imath+1, \jmath)}}$. If $\imath > 1$ then it is surjective on $\vta_{\imath+1}^{\sharp\sharp} \times (1/ m_\ell)^{\sharp\sharp}$ and $\int (\mdl X^\sharp_{\psi + \phi_{(2, \imath)}, \vta_{\imath+1}} \oplus \phi)^{-1}(a)$ only depends on $\rv(a_2) \in (1/ m_\ell)^{\sharp}$.
\end{lem}
\begin{proof}
The case $\imath = 1$ is assumed. For $\imath > 1$, we have, for $a \in \vta_{2}^{\sharp\sharp} \times (1/ m_\ell)^{\sharp\sharp}$, $\vv  (a_1 + a_2^{m_2}) = \vta_3$ if and only if  $a_1 \in - a_2^{m_2} + \vta_3^{\sharp\sharp} \sub \rv(- a^{m_2}_2)^\sharp \sub \vta_2^{\sharp\sharp}$. So, by the case $\imath = 1$ and the two conditions in Definition~\ref{defn:multi:ac:int}, the claim holds for $\imath = 2$. Iterating this argument,  the lemma follows.
\end{proof}

%

\begin{lem}\label{open:dia:int}
For $2 \leq \imath \leq \ell$ and $\vta_{\imath} \leq \eta < \vta_{\imath+1}$, we have the following equalities:
\begin{equation}\label{null:diag}
\int^{\ac}_\eta [\mdl X^\sharp_{\psi + \phi_{(2, \imath)}, \eta}] =
\begin{cases*}
(\dot \Pi_{\vta_{[2,\imath)}} \circ \int^{\ac}_{\vta_{[2,\imath)}}) ([\psi \oplus \phi]) & if $\vta_{\imath} = \eta$,\\
(\ddot \Pi_{\vta_{[2,\imath)}} \circ \int^{\ac}_{\vta_{[2,\imath)}}) ([\psi \oplus \phi]) & if $\vta_{\imath} < \eta < \vta_{\imath+1}$.
\end{cases*}
\end{equation}
\end{lem}
\begin{proof}
For each $b \in \vta_{2}^{\sharp\sharp}$ with $\rv(b) = r$, denote $\rv(\bigcup_{a \in \vta_{2}^{\sharp\sharp} \mi r^\sharp} a \times \sqrt[m_2]{b-a}) \sub \vta_{2}^{\sharp} \times (1/ m_\ell)^{\sharp}$ by $\dot W_r$, where we write $\dot W_r$ because it only depends on $r$.
For every $a \in \vta_{2}^{\sharp\sharp} \times (1/ m_\ell)^{\sharp\sharp}$ with $\rv(a) = s$, write $\int [(\psi \oplus \phi)^{-1}(a)] = [\bm U_{s}] / (\bm P - 1)$. Then $\int [(\mdl X^\sharp_{\psi + \phi_{(2, 2)}, \vta_2})^{-1}(b)] = [\bigcup_{s \in \dot W_r} \bm U_{s} \times (s, \pr_{\vta_2})] / (\bm P - 1)$, where $(s, \pr_{\vta_2})$ belongs to $\RV[1]$. Upon examination of the construction of $\dot \Pi_{\vta_{[2,2)}}$ in Definition~\ref{conv:RV}, we see that $\int^{\ac}_{\vta_2} [\mdl X^\sharp_{\psi + \phi_{(2, 2)}, \vta_2}]$  indeed works out at $(\dot \Pi_{\vta_{[2,2)}} \circ \int^{\ac}_{\vta_{[2,2)}}) ([\psi \oplus \phi])$.

Let  $\ddot W = \bigcup_{r \in \vta_{2}^{\sharp}} r \times \sqrt[m_2]{-r} \sub \vta_{2}^{\sharp} \times (1/ m_\ell)^{\sharp}$. If $\vta_{2} < \eta < \vta_{3}$  and $b \in \eta^{\sharp\sharp}$ then $\int [(\mdl X^\sharp_{\psi + \phi_{(2, 2)}, \eta})^{-1}(b)] = [\bigcup_{s \in \ddot W} \bm U_{s} \times (s, \pr_{\vta_2})] / (\bm P - 1)$. So $\int^{\ac}_\eta [\mdl X^\sharp_{\psi + \phi_{(2, 2)}, \eta}] = (\ddot \Pi_{\vta_{[2,2)}} \circ \int^{\ac}_{\vta_{[2,2)}}) ([\psi \oplus \phi])$.

Assume $\imath > 2$. Let $\psi' = \mdl X^\sharp_{\psi + \phi_{(2, 2)}, \vta_3}$. So $\mdl X^\sharp_{\psi + \phi_{(2, \imath)}, \eta} = \mdl X^\sharp_{\psi' + \phi_{(3, \imath)}, \eta}$. If $\vta_{\imath} = \eta$ then
\begin{align*}
\textstyle \int^{\ac}_\eta [\mdl X^\sharp_{\psi + \phi_{(2, \imath)}, \eta}] & = \textstyle \int^{\ac}_\eta [\mdl X^\sharp_{\psi' + \phi_{(3, \imath)}, \eta}]  \\
    & = \textstyle (\dot \Pi_{\vta_{[3,\imath)}} \circ \int^{\ac}_{\vta_{[3,\imath)}})([\psi' \oplus \phi]) \\
    & = \textstyle (\dot \Pi_{\vta_{[3,\imath)}} \circ  \Pi_{\vta_{[2,\imath)}}^{\vta_{[3,\imath)}} \circ \int^{\ac}_{\vta_{[2,\imath)}})([\psi \oplus \phi]) \\
    &= \textstyle (\dot \Pi_{\vta_{[2,\imath)}} \circ \int^{\ac}_{\vta_{[2,\imath)}})([\psi \oplus \phi]),
\end{align*}
where the second line is by an obvious inductive hypothesis and the third line follows from Lemma~\ref{chop:lam};  similarly if $\vta_{\imath} < \eta < \vta_{\imath+2}$.
\end{proof}

Taking $\eta = 1$, we have shown that the left square of (\ref{convol:comm}) commutes; similarly for $\ddot \Pi_{(\vta, \eta)}$ and $\Sigma_{(\vta, \eta)}$ if $\eta = m / m_\ell > 1$ for some $m \in \Z^+$ (see Remark~\ref{sigma:Galois}).

\begin{nota}\label{decom:mil}
Let $f, g : X \fun \MM$ be functions that are definable over $\bb S = \C$, for instance, polynomials in $\C[x]$ with $X = X(\MM)$ for a smooth connected variety $X$. Write $f_{(\imath, \jmath)}$ as in Notation~\ref{int:gen:the}.

For each $2 \leq \imath \leq \ell$, we identify three  pairwise disjoint subsets of $\{\vv \circ (g^N + f_{(2, \imath)}) = 1 \}$:
\begin{gather*}
  \mdl X^{-}_{g^N + f_{(2, \imath)}}=  \{\vv \circ (g^N + f_{(2, \imath - 1)}) < 1 < \vv \circ f^{m_{\imath+1}}\}, \\
  \mdl X^{\bullet}_{g^N + f_{(2, \imath)}} =  \{\vv \circ (g^N + f_{(2, \imath - 1)}) = \vv \circ f^{m_\imath} = 1 \}, \\
  \mdl X^{+}_{g^N + f_{(2, \imath)}} = \{\vv \circ (g^N + f_{(2, \imath - 1)}) > 1 \}.
\end{gather*}
Note that $\vv \circ (g^N + f_{(2, \imath - 1)}) = \vv \circ f^{m_{i}}$ for every $2 \leq i < \imath$ is implied in all three conditions and for $\imath$ also in the first one,  and $\vv \circ f^{m_\imath} =  1$ is implied in the third one.

Write $Z_f = f^{-1}(0)$ and $\mdl Z_{f} =  \{\vv \circ f > 1 \}$, similarly for other functions into $\MM$. Since  $\rv \circ f_{(\imath, \jmath)} = \rv \circ f^{m_{\imath}}$ for all $2 \leq \imath \leq \jmath \leq \ell$, we have
\begin{equation}\label{pos:tail:cut}
Z_{f_{(\imath, \jmath)}} = Z_f, \quad \mdl Z_{f_{(\imath, \jmath)}} = \mdl Z_{f^{m_{\imath}}}, \quad \mdl X^{+}_{g^N + f_{(2, \imath)}} = \mdl X^\sharp_{f_{(\imath, \jmath)}} \cap \mdl Z_{g^N + f_{(2, \imath - 1)}}.
\end{equation}
Actually, for any $r \in (\gamma / m_{\imath})^{\sharp}$, $f_{(\imath, \jmath)}$ restricts to a bijection $r^\sharp \fun (r^{m_{\imath}})^\sharp$ (it is surjective because no proper nonempty subset of $(r^{m_{\imath}})^\sharp$ is $r$-definable, and injectivity is a consequence of Hensel's lemma), and hence, for any definable set $A \sub X$ and any $a, b \in (r^{m_{\imath}})^\sharp$,
\begin{equation}\label{fell:trunc}
\int [(\mdl X^\sharp_{f_{(\imath, \jmath)}, \gamma} \cap A)^{-1}(a)] = \int [(\mdl X^\sharp_{f^{m_{\imath}}, \gamma} \cap A)^{-1}(b)].
\end{equation}

The set $\mdl X^\sharp_{g^N + f_{(2, \ell)}}$ is decomposed into five parts:
\begin{equation}\label{TS:setlevel}
\begin{multlined}
\mdl X^\sharp_{g^N + f_{(2, \ell)}} = (\mdl X^\sharp_{g^N} \cap  \mdl Z_{f_{(2, \ell)}})  \cup (\mdl X^\sharp_{f_{(2, \ell)}} \cap \mdl{Z}_{g^N}) \cup \phantom{xxxxxxxxxxxxxxxxxxxx}\\
 \bigcup_{2 < \imath \leq \ell} \mdl X^{+}_{g^N + f_{(2, \imath)}} \cup \bigcup_{2 \leq \imath \leq \ell} (\mdl X^{\bullet}_{g^N + f_{(2, \imath)}} \cup \mdl X^-_{g^N + f_{(2, \imath)}}),
\end{multlined}
\end{equation}
corresponding to the five terms encoded by the combinatorial data in Figure~\ref{TS:fig}, in the same order as presented thereabout. The restriction of $g^N + f_{(2, \ell)}$ to each part is a definable function onto $1^{\sharp\sharp}$ and hence is an object of $\VF^{\ac}$; to curb excess of notation,  these functions and other similar ones below shall just be denoted by their respective domains.

Write $\Vol^{\ac}([\mdl X^\sharp_{f} \cap A])$ as $\mathscr S_{f}^\sharp([A])$, and simply $\mathscr S_{f}^\sharp$ if $A = X$.
\end{nota}

Thus, computing $\mathscr S_{g^N + f_{(2, \ell)}}^\sharp$ boils down to computing the $\Vol^{\ac}$-values of the five terms on the right-hand side of (\ref{TS:setlevel}). The computation of the last two terms makes use of Lemma~\ref{open:dia:int}, details are given in the next subsection. For the first three terms, this can be done together:


\begin{lem}\label{Xfg:gen}
Let $\phi, \psi : X \fun \MM$ be definable functions. If $M \in \Z^+$ is  sufficiently large then
\[
\mathscr S_{\phi^M}^\sharp([\mdl Z_{\psi}]) = \mathscr S_{\phi^M}^\sharp([Z_{\psi}]) \dand \mathscr S_{\psi}^\sharp([\mdl Z_{\phi^M}]) = \mathscr S_{\psi}^\sharp.
\]
\end{lem}
\begin{proof}
We begin by considering $\phi$ instead of $\phi^M$. For $\gamma, \beta \in \Q^+$, denote by $\mdl X^\sharp_{\phi, \gamma, \psi, \beta}$ the restriction of $\phi$ to the set $\mdl X^\sharp_{\phi, \gamma} \cap \mdl X^\sharp_{\psi, \beta}$, which is a definable function onto $\gamma^{\sharp\sharp}$, and write
\[
\int^{\ac}_\gamma [\mdl X^\sharp_{\phi, \gamma, \psi, \beta}] = [\bm W_{\gamma, \beta}] /  (\bm P - 1).
\]
By \cite[\S~5]{HL:modified},  there is a $(\gamma, \beta)$-definable finite partition $(D_{\gamma, \beta, i})_i$ of $\vrv(\bm W_{\gamma, \beta})$ such that, for  each $i$, the set $\bm W_{\gamma, \beta} \cap D_{\gamma, \beta, i}^\sharp$ is a bipolar twistoid. By compactness, there is a $\gamma$-definable finite partition $(E_{\gamma, j})_j$ of $\Q^+$ such that, over each piece $E_{\gamma, j}$, the partitions  $(D_{\gamma,\beta, i})_i$ may be achieved uniformly and, for each $i$, the corresponding twistbacks are the same. So each class
$\int^{\ac}_\gamma [ \bigcup_{\beta \in E_{\gamma, j}} \mdl X^\sharp_{\phi, \gamma, \psi, \beta} ]$
is indeed represented by a finite disjoint union of bipolar twistoids $\bm W_{\gamma,i,j} \in \RV^{\ac}_\gamma[*]$ and each $\vrv(\bm W_{\gamma,i,j})$ is of the form $\bigcup_{\beta \in E_{\gamma, j}} D_{\gamma,\beta, i} \times \beta$, where $\chi_b(D_{\gamma,\beta, i}) \in \Z$ is constant over $E_{\gamma, j}$.

Working over $\bb S = \C$, by compactness, these partitions $(E_{\gamma,j})_j$ may be achieved uniformly over $\gamma \in \Q^+$; in other words, they form a definable finite partition of $(\Q^+)^2$ whose pieces are cones based at the origin.  This implies that there are a $j$ and a $p \in \Q^+$ such that $(p\gamma, \infty) \sub E_{\gamma, j}$ for all  $\gamma \in \Q^+$.  Since $M$ is sufficiently large, we have $(1, \infty) \sub (p / M, \infty) \sub E_{1/M, j}$. Therefore, the class
$\int^{\ac}_{1/M} [ \bigcup_{\beta \in (1, \infty)} \mdl X^\sharp_{\phi,  1/M, \psi, \beta} ]$
is  represented by a finite disjoint union of bipolar twistoids $\bm W_{1/M, i} \in \RV^{\ac}_{1/M}[*]$ such that each $\vrv(\bm W_{1/M,i})$ may be written in the form $\bigcup_{\beta \in (1, \infty)} D_{1/M, \beta, i} \times \beta$, where $\chi_b(D_{1/M, \beta, i}) \in \Z$ is constant over $\beta \in (1, \infty)$. Observe that the class
$\int^{\ac} [ \bigcup_{\beta \in (1, \infty)} \mdl X^\sharp_{\phi^M, \psi, \beta} ]$
must admit a representative of this form as well, which then is annihilated by $\bb E^{\ac}_{b}$ because $\chi_b((1, \infty)) = 0$. This leaves only $Z_{\psi}$ in the computation. The first equality follows.

For the second equality, since  the roles of $\phi^M$, $\psi$ are not exactly symmetric, a slightly different argument is needed. Let the restrictions $\mdl X^\sharp_{\psi, \gamma, \phi, \beta}$ of $\psi$ and the partitions $(D_{\gamma,\beta, i})_i$, $(E_{\gamma, j})_j$ be as above. In fact, we  only need the case $\gamma = 1$ and hence can write $\mdl X^\sharp_{\psi, \phi, \beta}$, $D_{\beta, i}$, $E_{j}$ instead.  Since $M$ is sufficiently large, $(0, 1/M] \sub E_j$ for some $j$. Then the class $\int^{\ac} [\bigcup_{\beta \in (0, 1/M]} \mdl X^\sharp_{\psi, \phi, \beta}]$ is represented by a finite disjoint union of bipolar twistoids similar to $\bm W_{1/M, i}$, hence so is the class $\int^{\ac} [\bigcup_{\beta \in (0, 1]} \mdl X^\sharp_{\psi, \phi^M, \beta}]$. This is again annihilated by $\bb E^{\ac}_{b}$ because $\chi_b((0,1]) = 0$. Since $\mdl X^\sharp_{\psi}$ is the union of $\mdl X^\sharp_{\psi} \cap \mdl Z_{\phi^M}$ and $\bigcup_{\beta \in (0,1]} \mdl X^\sharp_{\psi, \phi^M, \beta}$, the lemma follows.
\end{proof}

\begin{rem}\label{fub:why:xb}
It is not essential to use the bounded Euler characteristic $\chi_b$ for the second equality as the interval $(0, 1]$ vanishes under both, but $\chi_b$ is needed for the first equality.
\end{rem}

\begin{hyp}\label{hyp:seg:inc}
Henceforth we assume that each $m_i$ is sufficiently large relative to the data in question that involve only the  numbers before it. This condition is needed and will become clear  whenever Remark~\ref{int:to:resac:2} is invoked (implicitly).
\end{hyp}

\begin{cor}
Substituting suitable functions from Notation~\ref{decom:mil} for $\phi$, $\psi$ in Lemma~\ref{Xfg:gen}, we obtain the following equalities:
\begin{itemize}[leftmargin=*]
  \item $\mathscr S_{g^N}^\sharp([\mdl Z_{f_{(2,\ell)}}]) = \mathscr S_{g^N}^\sharp([\mdl Z_{f^{m_{2}}}]) = \mathscr S_{g^N}^\sharp([Z_{f^{m_{2}}}]) = \mathscr S_{g^N}^\sharp([Z_{f}])$,
  \item for $2 < \imath \leq \ell$, $\Vol^{\ac}([\mdl X^{+}_{g^N + f_{(2, \imath)}}]) = \mathscr S^\sharp_{f^{m_{\imath}}}([Z_{g^N + f_{(2, \imath - 1)}}])$,
  \item $\mathscr S_{f_{(2, \ell)}}^\sharp ([\mdl Z_{g^N}]) = \mathscr S_{f^{m_{2}}}^\sharp ([\mdl Z_{g^N}]) =  \mathscr S_{f^{m_2}}^\sharp$.
\end{itemize}
\end{cor}
\begin{proof}
The second  equality needs additional explanation. By (\ref{fell:trunc}) and (\ref{pos:tail:cut}),
\[
\Vol^{\ac}([\mdl X^{+}_{g^N + f_{(2, \imath)}}]) = \mathscr S^\sharp_{f_{(\imath, \ell)}}([\mdl Z_{g^N + f_{(2, \imath - 1)}}]) = \mathscr S^\sharp_{f^{m_{\imath}}}([\mdl Z_{g^N + f_{(2, \imath - 1)}}]).
\]
So the first equality of Lemma~\ref{Xfg:gen} may be applied with $\phi = f$, $M = m_{\imath}$, and $\psi = g^N + f_{(2, \imath - 1)}$.
\end{proof}

\subsection{A local Thom-Sebastiani formula}
For $2 \leq i \leq \imath \leq \ell$ and $\alpha \in \Q^+$, let $\vta^\imath_i = m_i / m_\imath$ and  $\vta^{(\imath)}_{\alpha} = (\alpha \vta^\imath_2, \alpha \vta^\imath_i)_{2 \leq i \leq \imath}$. Let $g^N_{\imath, \alpha} = \mdl X^\sharp_{g^N, \alpha \vta^\imath_2}$ and $f_{\imath, \alpha} = \mdl X^\sharp_{f, \alpha /m_\imath}$. Then $g^N_{\imath, \alpha} \oplus f_{\imath, \alpha} \in \VF^{\ac}_{\vta^{(\imath)}_{\alpha}}$. If $\alpha = 1$ then it is dropped from the notation. Write $\Vol^{\ac}_{\vta^{(\imath)}_\alpha}([g^N_{\imath, \alpha} \oplus f_{\imath, \alpha}])$ as $\mathscr S^\sharp_{g^N_{\imath, \alpha} \oplus f_{\imath, \alpha}}$.

Let $L'_{\imath}$ be the open interval $(m_\imath/m_{\imath+1}, 1)\sub \Q^+$. If $a \in \mdl X^{\bullet}_{g^N + f_{(2, \imath)}}$ then $\vv ( f(a)) = 1/ m_\imath$ and if $a \in \mdl X^-_{g^N + f_{(2, \imath)}}$ then $\vv ( f(a)) \in L'_{\imath}/ m_\imath$. For each $\alpha \in L'_{\imath}$, let $\mdl X^{-, \alpha}_{g^N + f_{(2, \imath)}}$ be the restriction of $\mdl X^-_{g^N + f_{(2, \imath)}}$ determined by the condition $\vv ( f(a)) = \alpha / m_\imath$.

Since (\ref{convol:comm}) commutes, taking $\eta = 1$ and $\psi \oplus \phi = g^N_{\imath} \oplus f_{\imath}$  in Lemma~\ref{open:dia:int}, we obtain
\begin{equation*}
\Vol^{\ac}([\mdl X^{\bullet}_{g^N + f_{(2, \imath)}}]) = (\Vol^{\ac} \circ \Sigma_{\vta^{(\imath)}})([g^N_{\imath} \oplus f_{\imath}]) =  \dot \Psi_{\vta^{(\imath)}} (\mathscr S^\sharp_{g^N_{\imath} \oplus f_{\imath}})
\end{equation*}
and taking $\eta = 1$ and $\psi \oplus \phi = g^N_{\imath, \alpha} \oplus f_{\imath, \alpha}$ for $\alpha \in L'_{\imath}$ in Lemma~\ref{open:dia:int}, we obtain
\begin{equation*}
\Vol^{\ac}([\mdl X^{-, \alpha}_{g^N + f_{(2, \imath)}}]) = \ddot \Psi_{\vta^{(\imath)}_\alpha} (\mathscr S^\sharp_{g^N_{\imath, \alpha} \oplus f_{\imath, \alpha}}).
\end{equation*}
By Remark~\ref{sigma:Galois}, the right-hand side of this second equality is actually the same for any  $\alpha \in \Q^+$ and hence, in particular,  may be written as $\ddot \Psi_{\vta^{(\imath)}} (\mathscr S^\sharp_{g^N_{\imath} \oplus f_{\imath}})$. From another perspective, if we write $\int^{\ac} [\mdl X^{-, \alpha}_{g^N + f_{(2, \imath)}}] = [\bm U_{\alpha}] / (\bm P - 1)$ then there is an $\alpha$-definable partition of $\vv(\bm U_{\alpha})$ of the form $(D_{k\alpha} \times \alpha)_k$, uniform over $\Q^+$, such that each $\bm U_{\alpha} \cap (D_{k\alpha} \times \alpha)^\sharp$ is a bipolar twistoid. Thus,  $\int^{\ac} [\mdl X^{-}_{g^N + f_{(2, \imath)}}]$ is represented by a finite disjoint union of bipolar twistoids $\bm W_k \in \RV[*]$ with $\vrv(\bm W_k) = \bigcup_{\alpha \in L'_{\imath}} D_{k\alpha} \times \alpha$. Since $\chi_b(D_{k\alpha})$ is constant over $L'_{\imath}$ for every $k$, it follows that
\begin{equation}\label{spec:dia:int}
\Vol^{\ac}([\mdl X^{-}_{g^N + f_{(2, \imath)}}]) = \chi_b(L'_{\imath}) \ddot \Psi_{\vta^{(\imath)}} (\mathscr S^\sharp_{g^N_{\imath} \oplus f_{\imath}}) = - \ddot \Psi_{\vta^{(\imath)}} (\mathscr S^\sharp_{g^N_{\imath} \oplus f_{\imath}}).
\end{equation}
So the minus sign on $\ddot \Psi_{\vta}$ in the definition of $\Psi_{\vta}$ may now be interpreted as the Euler characteristic of a bounded open interval.


\begin{thm}\label{TS:main}
In conclusion, we have derived a local Thom-Sebastiani formula in $\ggk^{\bm 1} \Var_{\C}$:
\[
\mathscr S_{g^N + f_{(2,\ell)}}^\sharp = \mathscr S_{g^N}^\sharp([Z_{f}]) + \mathscr S_{f^{m_{2}}}^\sharp + \sum_{2 < \imath \leq \ell} \mathscr S^\sharp_{f^{m_{\imath}}}([Z_{g^N + f_{(2,\imath-1)}}]) - \sum_{2 \leq \imath \leq \ell} \Psi_{\vta^{(\imath)}} (\mathscr S^\sharp_{g^N_{\imath} \oplus f_{\imath}}).
\]
\end{thm}

The special case $\ell= 2$ and $m_2 = 1$ is related to the local Thom-Sebastiani formula in \cite[Corollary~5.16]{guibert2006} as follows. In terms of motivic Milnor fibers instead of motivic vanishing cycles, this latter formula may be written as
\begin{equation}\label{GLM:form:mil}
\mathscr S_{f, z} - \mathscr S_{g^N + f, z} = \Psi_\Sigma(\mathscr S_{g^N, z}(\mathscr  S_{f})) - \mathscr S_{g^N, z}([f^{-1}(0)]).
\end{equation}
Here $z \in f^{-1}(0)$ is a $\C$-rational point, which is implicit in Theorem~\ref{TS:main} (recall the simplification made at the beginning of this section).  The (local) motivic Milnor fibers $\mathscr S_{f, z}$ and $\mathscr S_{g^N + f, z}$ are constructed via motivic zeta functions with coefficients in $\mathscr{M}_{\G_m}^{\G_m}$; see \cite[\S~3.6]{guibert2006} for details. The meaning of the term $\mathscr S_{g^N, z}([f^{-1}(0)])$ is established in \cite[Theorem~3.9]{guibert2006}, and it belongs to $\mathscr{M}_{\G_m}^{\G_m}$. According to the nearby cycles formalism of \cite[\S~4.6]{guibert2006}, $\mathscr S_{g^N, z}(\mathscr  S_{f})$ belongs to $\mathscr{M}_{\G_m^2}^{\G^2_m}$. But then, after applying the operator $\Psi_\Sigma$ as defined in \cite[\S~5.1]{guibert2006}, it comes down to $\mathscr{M}_{\G_m}^{\G_m}$ as well. In a nutshell, the expression (\ref{GLM:form:mil}) is well-typed.

There is an isomorphism $\Upsilon: \mathscr{M}_{\G_m}^{\G_m} \fun \mathscr{M}^{\hat \mu}$, which is just (\ref{fib:at:1}) localized at $[\A]$. It can be checked that a similar construction via ``taking the fiber at $\rcsn(1)$'' also yields an isomorphism $\sggk \RES^{\ac} \fun \sggk^{\hat \mu} \RES$, which shall also be denoted by $\Upsilon$, and indeed $\Theta^{\hat \mu} \circ \Upsilon = \Upsilon \circ \Theta^{\ac}$. Consequently, by \cite[Remark~3.13]{guibert2006} and the complex version of Theorem~\ref{direct:mil} (see \cite[Theorem~8.11]{HL:modified}), we have
\begin{equation}\label{mil:fib:coin}
\mathscr S_{g^N + f}^\sharp = \mathscr S_{g^N + f, z}, \quad \mathscr S_{f}^\sharp = \mathscr S_{f, z}, \quad \mathscr S_{g^N}^\sharp([Z_{f}]) = \mathscr S_{g^N, z}([f^{-1}(0)]).
\end{equation}
This implies that, for any sufficiently large $N \in \Z^+$,
\begin{equation}\label{relate:conv}
  \Psi_{\bm 2} (\mathscr S^\sharp_{g^N \oplus f}) = \Psi_\Sigma(\mathscr S_{g^N, z}(\mathscr  S_{f})).
\end{equation}
The methodology of \cite{guibert2006} offers a geometric interpretation of ``sufficiently large $N \in \Z^+$'' in terms of log-resolutions. Our interpretation lies in the proof of Lemma~\ref{Xfg:gen} and Remark~\ref{int:to:resac:2}, and is not as informative since it depends on compactness. It is not clear how to relate the two thresholds. Also note that the left-hand side of (\ref{relate:conv}) is obviously commutative in the sense that  $\mathscr S^\sharp_{g^N \oplus f} = \mathscr S^\sharp_{f \oplus g^N}$, and perhaps this can be translated into an expression on the right-hand side through a resolution-based analysis of the motivic zeta functions involved.

The setup for the motivic Thom-Sebastiani formula in \cite{DL:tom:seb} involves a morphism $f'$ on another smooth variety $X'$ and the obvious morphism $f + f'$ on the product $Y = X \times X'$. This formula is a special case of \cite[Corollary~5.16]{guibert2006}, as demonstrated in \cite[Theorem~5.18]{guibert2006}, and hence can be recovered from Theorem~\ref{TS:main} as well, although we do need to check that it holds for $N = 1$ in that situation. Anyway, we can give a more direct proof. To begin with, write (\ref{TS:setlevel}) as
\begin{equation*}
\mdl Y^\sharp_{f + f'} = (\mdl X^\sharp_{f} \times \mdl Z_{f'}) \cup (\mdl X^\sharp_{f'} \times \mdl Z_{f})  \cup  \mdl Y^+_{f + f'}  \cup  \mdl Y^-_{f + f'}.
\end{equation*}
Observe that the conclusion of Remark~\ref{int:to:resac:2} already holds for the function  $f \oplus f'$ on $Y(\MM) = X(\MM) \times X'(\MM)$ and indeed
\[
\Vol^{\ac}([\mdl Y^+_{f + f'}  \cup  \mdl Y^-_{f + f'}]) = - \Psi_{\bm 2} (\mathscr S_{f \oplus f'}^\sharp) = - \mathscr S_f^\sharp * \mathscr S_{f'}^\sharp.
\]
To compute the other two terms, now symmetric, the key is the following equality.

\begin{lem}\label{Zf0}
$(\bb E_b \circ \int) ([Z_f]) = 1$.
\end{lem}
\begin{proof}
We actually show a more general claim: Over $\bb S = \C$, if $A$ is a definable set in $\MM$ then $(\bb E_b \circ \int) ([A]) = 0$ if $0 \notin A$ and $(\bb E_b \circ \int) ([A]) = 1$ otherwise. This is enough since enlarging the language (new parameters, new function symbols, etc.) will not change these equalities.

Since there is no definable point in $\Gamma \cong \Q$ except $0$, we see that if $(U, f) \in \RES[k]$ then $U$, $f(U)$ are just constructible sets in $\C$.  Let $A$ be a definable set. Then $\int [A]$ may be expressed as a finite sum $\sum_i [\bm U_i] \otimes [D_i]$ modulo $(\bm P - 1)$, where $[\bm U_i] \in \ggk \RES[*]$ and $[D_i] \in \ggk \Gamma[*]$. We may assume that either $[D_i] = 1$ (if $D_i \in \Gamma^{\fin}[*]$ then it may be absorbed into $\bm U_i$) or  $D_i$ is infinite. In the latter case, for some coordinate projection, say $\pr_1$, we may further assume that $\pr_1(D_i)$ is $(-\infty, 0)$ or $(0, \infty)$ or $\Q \mi 0$ and hence, by \omin-minimality, $\bb E_b(D_i) = 0$.

Thus, to compute $(\bb E_b \circ \int) ([A])$, we may write $\int [A]$  as $\sum_i [\bm U_i]$ modulo $(\bm P - 1)$. By Theorem~\ref{main:prop}, there is a definable injection $g : \biguplus_i \bb L \bm U_i \fun A$. By orthogonality (Remark~\ref{pillars}),  $\vv(g(\biguplus_i \bb L \bm U_i))$ is finite and hence only $0$ and $\infty$ can occur in its coordinates; in the case we are interested in, that is, $A \sub \MM^n$ for some $n$, only $\infty$ can occur, but then $A$ must contain the point $0$. So $g(\biguplus_i \bb L \bm U_i)$ is either empty or is the singleton $0$, which means that $\sum_i [\bm U_i]$ is either $0$ or $1$, respectively. Since $(\bb E_b \circ \int) ([A \mi 0])$ also equals $1$ or $0$, we see that $(\bb E_b \circ \int) ([A]) = 1$ if and only if $0 \in A$.
\end{proof}

By the same reasoning that leads to (\ref{spec:dia:int}), the class  $\int [\mdl Z_{f'} \mi Z_{f'}]$ is represented by a finite disjoint union of bipolar twistoids $\bm W_i \in \RV[*]$ such that $\vrv(\bm W_i)$ is of the form $\bigcup_{\gamma \in (1, \infty)} D_{i\gamma} \times \gamma$ and $\chi_b(D_{i\gamma})$ is constant over $(1, \infty)$ for every $i$. So $\Vol^{\hat \mu}([\mdl Z_{f'} \mi Z_{f'}]) = 0$. Then, by Lemmas~\ref{rv:int:bun} and~\ref{Zf0}, for every $r \in 1^\sharp$,
\begin{equation*}\label{dag:Zf}
\Vol^{\hat \mu} ([(\mdl X^\sharp_f \times \mdl Z_{f'})^{-1}(r^\sharp)]) = \Vol^{\hat \mu} ([\mdl X_{f, r}][\mdl Z_{f'} \mi Z_{f'}] + [\mdl X_{f, r}][Z_{f'}]) = \Vol^{\hat \mu}( [\mdl X_{f, r}]).
\end{equation*}
This shows that $\Vol^{\ac}([\mdl X^\sharp_f \times \mdl Z_{f'}]) = \mathscr S_f^\sharp$ and hence
\begin{equation}\label{TS:class}
\mathscr S_{f+f'}^\sharp = \mathscr S_f^\sharp + \mathscr S_{f'}^\sharp - \mathscr S_f^\sharp * \mathscr S_{f'}^\sharp.
\end{equation}

\subsubsection{The real case}\label{sect:TS:real}

If we work in the $\ACVF$-model $\puC$ with $\bb S = \R \cup \Q$ and let the variety $X$, the morphism $f$, etc., be defined over $\R$ then the preceding discussion is still valid. In more detail,
there is a subgroup of $\gal(\C \dpar t/ \R)$ that may be identified with $\gal(\C/\R) \ltimes \C^\times$; its preimage along the canonical surjective homomorphism
\[
\gal(\puC / \R) \fun \gal(\C \dpar t/ \R)
\]
is denoted by $\bm c \hat \tau$, which may  be identified with $\lim_n (\gal(\C/\R) \ltimes \C^\times)_n$. There is an isomorphism
$\sggk^{\bm c \hat \tau} \RES \cong \ggk^{\bm c \hat \tau} \Var_{\R}$ (for surjectivity, combine the arguments in Remarks~\ref{first:theta} and \ref{theta:suj}).

The categories in Definition~\ref{theta:C:cat} and the corresponding Grothendieck groups are now written as $\Var_{\R}^{\vta, n}$ and $\ggk^{\vta, n} \Var_{\R}$. As in Definition~\ref{ct:var:over:R}, for an object $(Y, \pi)$ of $\Var_{\R}^{\vta, n}$, the  $\gal(\C/\R) \ltimes \C^\times$-action on $Y$ and the morphism $\pi : Y  \fun \G_m^\ell$ are required to be compatible with the antiholomorphic involution  in question; in particular, for the generator $\bm c \in \gal(\C/\R)$,  (\ref{good:act}) should read
\begin{equation}\label{good:act:conj}
\pi_1(\bm c \cdot y) = \bm c \cdot \pi_1(y) \dand \pi^*(\bm c \cdot y) = \bm c \cdot \pi^*(y),
\end{equation}
so if $y$ is a real point then $\pi_1(y)$, $\pi^*(y)$ must be real points too. We construct a $\ggk^{\bm c \hat \tau} \Var_{\R}$-module homomorphism
$\Psi_\vta : \ggk^{\vta} \Var_{\R} \fun \ggk^{\bm 1} \Var_{\R}$
as in Definition~\ref{C:cat:conv}. Then Theorem~\ref{TS:main} holds in $\ggk^{\bm 1} \Var_{\R}$ as well.

However, as in \S~\ref{subsec:milnor}, we are more interested in a statement that is indigenous to the real algebraic environment. In addition,  we shall point out how to deduce the real Thom-Sebastiani formula in \cite{Campesato} from ours.

Let $\ggk^{\hat \rho} \RVar$ be the real analogue of $\ggk^{\hat \tau} \Var_{\C}$, that is, the Grothendieck ring of the category of real varieties with weighted $\R^\times$-actions. A morphism $\pi : Y(\R) \fun (\R^\times)^\ell$ on a real variety $Y(\R)$ with an $\R^\times$-action is  $(\vta, n)$-diagonal if the obvious analogue of (\ref{good:act}) holds. The categories $\RVar^{\vta, n}$, $\RVar^{\vta}$, etc., are defined accordingly. The $\ggk^{\hat \rho} \RVar$-module homomorphism $\Psi_\vta$ in the bottom row of (\ref{TS:down:real}) is constructed as in Definition~\ref{C:cat:conv} again.

Given any $n$-weighted $\bm c \hat \tau$-action $\hat h$ on $Y \otimes_{\R}\C$, by considering the induced $\delta_n$-action in each fiber and the orbit size of each real point as in Definition~\ref{mu2:var}, one sees that $\hat h$ gives rise to an $n$-weighted $\R^\times$-action on $Y(\R)$. Consequently, as in (\ref{tauhat:forget}),  taking real points yields $\mdl A_{\C}$-module homomorphisms $\Xi^{\vta}$, $\Xi^{\bm 1}$ in (\ref{TS:down:real}) (also one $\ggk^{\bm c \hat \tau} \Var_{\R} \fun \ggk^{\hat \rho} \RVar$).
\begin{equation}\label{TS:down:real}
\bfig
\hSquares(0,0)/->`->`->`->`->`->`->/<400>[{\ggk^{\vta} \Var_{\R}}`{\ggk^{\bm 1} \Var_{\R}}`{\gdv}`{\ggk^{\vta} \RVar}`{\ggk^{\bm 1} \RVar}`{\gsv}; \Psi_\vta`\Upsilon^1`\Xi^{\vta}`\Xi^{\bm 1}`\Xi`\Psi_\vta`\Upsilon^1]
\efig
\end{equation}

By an inductive argument similar to the one in the proof of Lemma~\ref{RV:RES:conv}, noting also that, by (\ref{good:act:conj}), fibers of $\pi$ over genuinely complex points make no contributions to fibers over real points in (\ref{conv:1:2}) and (\ref{conv:induc}), we deduce that the first square of (\ref{TS:down:real}) commutes. So Theorem~\ref{TS:main} holds in $\ggk^{\bm 1} \RVar$ too,  as a direct specialization of the same equality in $\ggk^{\bm 1} \Var_{\R}$ via $\Xi^{\vta}$ and $\Xi^{\bm 1}$.

This state of affairs may seem somewhat unsatisfactory as the supposedly real formula is in actuality  computed from the complex objects in (\ref{TS:setlevel}) and the volume operators    $\Vol^{\ac}_\vta$ over $\puC$. To remedy this, we can start the specialization procedure earlier, using the technique in \S~\ref{section:spec:hen}, as has been done in Remark~\ref{comp:real:non}, and obtain the same formula using the $\puR$-trace of (\ref{TS:setlevel}) and the corresponding volume operators    over $\puR$. No new perspective lies herein and hence we shall not labor further on it.

\begin{rem} \label{rem-compG}
The second square of (\ref{TS:down:real}) also commutes, where the two horizontal arrows are constructed via taking the fiber at $1$ as  in (\ref{fib:at:1}). However, as another manifestation of the duality of the sign, taking the fiber at $-1$ yields a genuinely different ring homomorphism
\[
\Upsilon^{-1}: \ggk^{\bm 1} \RVar \fun \gsv
\]
Neither $\Upsilon^{1}$ nor $\Upsilon^{-1}$ is injective, not even taken as a pair (for instance any even power function on the torus gives the same class).
\end{rem}

Now, the said formula in \cite{Campesato} is formulated in a specialization $\mathscr{M}_{\mathcal{AS}}$ of $\ggk^{\bm 1} \RVar[[\A]^{-1}]$, which is constructed using arc-symmetric (semialgebraic) sets and maps. In more detail, adapting the method of \cite{guibert2006}, the (generalized) real motivic Milnor fiber $\mathscr S_f^\times$ of $f$ is the limit of a motivic zeta function $Z^\times(T)$ whose coefficients are given by sets of  truncated arcs of the form
\[
\set{ \varphi \in X(\R[t] / t^{m+1}) \given f(\varphi) = a t^m \mod t^{m+1} \text{ with } a \in \R^\times \tand  \varphi(0) = z }
\]
together with the built-in angular component map sending $\varphi$ to $a$. Then an equality similar to the special case (\ref{TS:class}) may be established in $\mathscr{M}_{\mathcal{AS}}$; see \cite[Corollary~6.20]{Campesato}. Here we point out that the process of ``taking the limit'' forces the $\R^\times$-actions on the coefficients of $Z^\times(T)$ to factor through a $\R^+$-action and, consequently, the negative part of $\R^\times$ does not really figure in $\mathscr S_f^\times$; this is but another manifestation of what has been said in Remark~\ref{mu2:failure} about the construction in \cite{Fseries}.

Let us rather consider the same construction at the level of $\ggk^{\bm 1} \RVar$ (hence finer, since full $\R^\times$-actions are retained). In order to show that \cite[Corollary~6.20]{Campesato} can  be obtained from the specialization of (\ref{TS:class}) to $\ggk^{\bm 1} \RVar$, one needs to check that $\mathscr S_f^\times$ can indeed be recovered as $\Vol^{\ac}(\mdl X^\sharp_f)$ over $\puR$, similar to (\ref{mil:fib:coin}). We may try to reproduce the argument given there. To begin with, taking the fiber at $1$ coefficientwise, we recover from $Z^\times(T)$  the motivic zeta function $Z^1(T)$ in (\ref{real:zeta}) (taking the fiber at $-1$ gives its negative counterpart $Z^{-1}(T)$), and it is straightforward to check that this operation commutes with the operator ``$ - \lim_{T \limplies \infty}$'' in (\ref{pos:mil:no:act}); actually this is just an analogue of  \cite[Remark~3.13]{guibert2006}, which we have also gone through in \S~\ref{subsec:milnor}. However, this is as far as we can go since, unlike $\Upsilon$ in  (\ref{fib:at:1}), $\Upsilon^1$ is not an isomorphism. In other words, although we know that the images of $\mathscr S_f^\times$, $\Vol^{\ac}(\mdl X^\sharp_f)$ under $\Upsilon^1$ coincide in $\gsv$, we cannot conclude that they themselves coincide in  $\ggk^{\bm 1} \RVar$.

Thus the apparent shortcut is blocked in the real environment, and we shall have to revert back to the zeta function point of view, that is, we need to show a version of Theorem~\ref{direct:mil}  with respect to $Z^\times(T)$ and $\mdl X^\sharp_f(\puR)$. Although some extra care is needed  concerning the use of the integral $\int^\diamond$, there is no new insight arising in this endeavor and, as above, we choose not to labor further on technicalities.

\providecommand{\bysame}{\leavevmode\hbox to3em{\hrulefill}\thinspace}
\providecommand{\MR}{\relax\ifhmode\unskip\space\fi MR }
\providecommand{\MRhref}[2]{%
  \href{http://www.ams.org/mathscinet-getitem?mr=#1}{#2}
}
\providecommand{\href}[2]{#2}

%

\end{document}